\documentclass[12pt,leqno]{amsart}
\usepackage{amssymb}
\usepackage{tikz}
\usepackage{pgfplots}
\pgfplotsset{compat=1.10}
\usetikzlibrary{pgfplots.fillbetween}
\pgfplotsset{compat = newest}
\usepackage{xcolor}
\colorlet{olive1}{olive!25}
\colorlet{olive2}{olive!35}
\colorlet{olive3}{olive!45}
\usepackage[left=2cm,right=2cm,top=2.5cm,bottom=2.5cm]{geometry}

\usepackage{caption} 
\usepackage[colorlinks,linkcolor=blue,pagebackref=true, citecolor=red]{hyperref} 
\newcommand{\colvec}[2][.8]{%
  \scalebox{#1}{%
    \renewcommand{\arraystretch}{.8}%
    $\begin{bmatrix}#2\end{bmatrix}$%
  }
}
\begin{document}
\newtheorem{thm}{Theorem}[section]
\newtheorem*{thm1}{Theorem 1}
\newtheorem*{thm2}{Theorem 2}
\newtheorem{lemma}[thm]{Lemma}
\newtheorem{lem}[thm]{Lemma}
\newtheorem{cor}[thm]{Corollary}
\newtheorem*{cor*}{Corollary}
\newtheorem{prop}[thm]{Proposition}
\newtheorem{propose}[thm]{Proposition}
\newtheorem{variant}[thm]{Variant}
\theoremstyle{definition}
\newtheorem{notations}[thm]{Notations}
\newtheorem{rem}[thm]{Remark}
\newtheorem{rmk}[thm]{Remark}
\newtheorem{rmks}[thm]{Remarks}
\newtheorem{defn}[thm]{Definition}
\newtheorem{ex}[thm]{Example}
\newtheorem{quest}[thm]{Question}
\newtheorem{claim}[thm]{Claim}
\newtheorem{ass}[thm]{Assumption}
\numberwithin{equation}{section}
\newcounter{elno}                
\def\points{\list{\hss\llap{\upshape{(\roman{elno})}}}{\usecounter{elno}}}
\let\endpoints=\endlist


\catcode`\@=11
%
%
\def\opn#1#2{\def#1{\mathop{\kern0pt\fam0#2}\nolimits}} 
\def\bold#1{{\bf #1}}%
\def\underrightarrow{\mathpalette\underrightarrow@}
\def\underrightarrow@#1#2{\vtop{\ialign{$##$\cr
 \hfil#1#2\hfil\cr\noalign{\nointerlineskip}%
 #1{-}\mkern-6mu\cleaders\hbox{$#1\mkern-2mu{-}\mkern-2mu$}\hfill
 \mkern-6mu{\to}\cr}}}
\let\underarrow\underrightarrow
\def\underleftarrow{\mathpalette\underleftarrow@}
\def\underleftarrow@#1#2{\vtop{\ialign{$##$\cr
 \hfil#1#2\hfil\cr\noalign{\nointerlineskip}#1{\leftarrow}\mkern-6mu
 \cleaders\hbox{$#1\mkern-2mu{-}\mkern-2mu$}\hfill
 \mkern-6mu{-}\cr}}}

%

%
\def\:{\colon}
\let\oldtilde=\tilde
\def\tilde#1{\mathchoice{\widetilde{#1}}{\widetilde{#1}}%
{\indextil{#1}}{\oldtilde{#1}}}
\def\indextil#1{\lower2pt\hbox{$\textstyle{\oldtilde{\raise2pt%
\hbox{$\scriptstyle{#1}$}}}$}}
\def\pnt{{\raise1.1pt\hbox{$\textstyle.$}}}
%

%
\let\amp@rs@nd@\relax

\newdimen\ex@\ex@.2326ex
\newdimen\bigaw@
\newdimen\minaw@
\minaw@16.08739\ex@
\newdimen\minCDaw@
\minCDaw@2.5pc
\newif\ifCD@
\def\minCDarrowwidth#1{\minCDaw@#1}
\newenvironment{CD}{\@CD}{\@endCD}
\def\@CD{\def\A##1A##2A{\llap{$\vcenter{\hbox
 {$\scriptstyle##1$}}$}\Big\uparrow\rlap{$\vcenter{\hbox{%
$\scriptstyle##2$}}$}&&}%
\def\V##1V##2V{\llap{$\vcenter{\hbox
 {$\scriptstyle##1$}}$}\Big\downarrow\rlap{$\vcenter{\hbox{%
$\scriptstyle##2$}}$}&&}%
\def\={&\hskip.5em\mathrel
 {\vbox{\hrule width\minCDaw@\vskip3\ex@\hrule width
 \minCDaw@}}\hskip.5em&}%
\def\verteq{\Big\Vert&&}%
\def\noarr{&&}%
\def\vspace##1{\noalign{\vskip##1\relax}}\relax\let\amp@rs@nd@&\iffalse}\fi
\CD@true\vcenter\bgroup\relax\let\\=\cr\iffalse}\fi\tabskip\z@skip\baselineskip20\ex@
 \lineskip3\ex@\lineskiplimit3\ex@\halign\bgroup
 &\hfill$\m@th##$\hfill\cr}
\def\@endCD{\cr\egroup\egroup}
%
\def\>#1>#2>{\amp@rs@nd@\setbox\z@\hbox{$\scriptstyle
 \;{#1}\;\;$}\setbox\@ne\hbox{$\scriptstyle\;{#2}\;\;$}\setbox\tw@
 \hbox{$#2$}\ifCD@
 \global\bigaw@\minCDaw@\else\global\bigaw@\minaw@\fi
 \ifdim\wd\z@>\bigaw@\global\bigaw@\wd\z@\fi
 \ifdim\wd\@ne>\bigaw@\global\bigaw@\wd\@ne\fi
 \ifCD@\hskip.5em\fi
 \ifdim\wd\tw@>\z@
 \mathrel{\mathop{\hbox to\bigaw@{\rightarrowfill}}\limits^{#1}_{#2}}\else
 \mathrel{\mathop{\hbox to\bigaw@{\rightarrowfill}}\limits^{#1}}\fi
 \ifCD@\hskip.5em\fi\amp@rs@nd@}
\def\<#1<#2<{\amp@rs@nd@\setbox\z@\hbox{$\scriptstyle
 \;\;{#1}\;$}\setbox\@ne\hbox{$\scriptstyle\;\;{#2}\;$}\setbox\tw@
 \hbox{$#2$}\ifCD@
\global\bigaw@\minCDaw@\else\global\bigaw@\minaw@\fi
 \ifdim\wd\z@>\bigaw@\global\bigaw@\wd\z@\fi
 \ifdim\wd\@ne>\bigaw@\global\bigaw@\wd\@ne\fi
 \ifCD@\hskip.5em\fi
 \ifdim\wd\tw@>\z@
 \mathrel{\mathop{\hbox to\bigaw@{\leftarrowfill}}\limits^{#1}_{#2}}\else
 \mathrel{\mathop{\hbox to\bigaw@{\leftarrowfill}}\limits^{#1}}\fi
 \ifCD@\hskip.5em\fi\amp@rs@nd@}
%
%
\newenvironment{CDS}{\@CDS}{\@endCDS}
\def\@CDS{\def\A##1A##2A{\llap{$\vcenter{\hbox
 {$\scriptstyle##1$}}$}\Big\uparrow\rlap{$\vcenter{\hbox{%
$\scriptstyle##2$}}$}&}%
\def\V##1V##2V{\llap{$\vcenter{\hbox
 {$\scriptstyle##1$}}$}\Big\downarrow\rlap{$\vcenter{\hbox{%
$\scriptstyle##2$}}$}&}%
\def\={&\hskip.5em\mathrel
 {\vbox{\hrule width\minCDaw@\vskip3\ex@\hrule width
 \minCDaw@}}\hskip.5em&}
\def\verteq{\Big\Vert&}
\def\novarr{&}
\def\noharr{&&}
\def\SE##1E##2E{\slantedarrow(0,18)(4,-3){##1}{##2}&}
\def\SW##1W##2W{\slantedarrow(24,18)(-4,-3){##1}{##2}&}
\def\NE##1E##2E{\slantedarrow(0,0)(4,3){##1}{##2}&}
\def\NW##1W##2W{\slantedarrow(24,0)(-4,3){##1}{##2}&}
\def\slantedarrow(##1)(##2)##3##4{%
\thinlines\unitlength1pt\lower 6.5pt\hbox{\begin{picture}(24,18)%
\put(##1){\(##2){24}}%
\put(0,8){$\scriptstyle##3$}%
\put(20,8){$\scriptstyle##4$}%
\end{picture}}}
\def\vspace##1{\noalign{\vskip##1\relax}}\relax\let\amp@rs@nd@&\iffalse}\fi
 \CD@true\vcenter\bgroup\relax\let\\=\cr\iffalse}\fi\tabskip\z@skip\baselineskip20\ex@
 \lineskip3\ex@\lineskiplimit3\ex@\halign\bgroup
 &\hfill$\m@th##$\hfill\cr}
\def\@endCDS{\cr\egroup\egroup}

\newdimen\TriCDarrw@
\newif\ifTriV@
\newenvironment{TriCDV}{\@TriCDV}{\@endTriCD}
\newenvironment{TriCDA}{\@TriCDA}{\@endTriCD}
\def\@TriCDV{\TriV@true\def\TriCDpos@{6}\@TriCD}
\def\@TriCDA{\TriV@false\def\TriCDpos@{10}\@TriCD}
\def\@TriCD#1#2#3#4#5#6{%
\setbox0\hbox{$\ifTriV@#6\else#1\fi$}
\TriCDarrw@=\wd0 \advance\TriCDarrw@ 24pt
\advance\TriCDarrw@ -1em
\def\SE##1E##2E{\slantedarrow(0,18)(2,-3){##1}{##2}&}
\def\SW##1W##2W{\slantedarrow(12,18)(-2,-3){##1}{##2}&}
\def\NE##1E##2E{\slantedarrow(0,0)(2,3){##1}{##2}&}
\def\NW##1W##2W{\slantedarrow(12,0)(-2,3){##1}{##2}&}
\def\slantedarrow(##1)(##2)##3##4{\thinlines\unitlength1pt
\lower 6.5pt\hbox{\begin{picture}(12,18)%
\put(##1){\(##2){12}}%
\put(-4,\TriCDpos@){$\scriptstyle##3$}%
put(12,\TriCDpos@){$\scriptstyle##4$}%
\end{picture}}}
\def\={\mathrel {\vbox{\hrule
   width\TriCDarrw@\vskip3\ex@\hrule width
   \TriCDarrw@}}}
\def\>##1>>{\setbox\z@\hbox{$\scriptstyle
 \;{##1}\;\;$}\global\bigaw@\TriCDarrw@
 \ifdim\wd\z@>\bigaw@\global\bigaw@\wd\z@\fi
 \hskip.5em
 \mathrel{\mathop{\hbox to \TriCDarrw@
{\rightarrowfill}}\limits^{##1}}
 \hskip.5em}
\def\<##1<<{\setbox\z@\hbox{$\scriptstyle
 \;{##1}\;\;$}\global\bigaw@\TriCDarrw@
 \ifdim\wd\z@>\bigaw@\global\bigaw@\wd\z@\fi
 \mathrel{\mathop{\hbox to\bigaw@{\leftarrowfill}}\limits^{##1}}}
 \CD@true\vcenter\bgroup\relax\let\\=\cr\iffalse}\fi

\tabskip\z@skip\baselineskip20\ex@
 \lineskip3\ex@\lineskiplimit3\ex@
 \ifTriV@
 \halign\bgroup
 &\hfill$\m@th##$\hfill\cr
#1&\multispan3\hfill$#2$\hfill&#3\\
&#4&#5\\
&&#6\cr\egroup%
\else
 \halign\bgroup
 &\hfill$\m@th##$\hfill\cr
&&#1\\%
&#2&#3\\
#4&\multispan3\hfill$#5$\hfill&#6\cr\egroup
\fi}
\def\@endTriCD{\egroup}
\newcommand{\mc}{\mathcal}
\newcommand{\mb}{\mathbb}
\newcommand{\surj}{\twoheadrightarrow}
\newcommand{\inj}{\hookrightarrow} \newcommand{\zar}{{\rm zar}}
\newcommand{\an}{{\rm an}} \newcommand{\red}{{\rm red}}
\newcommand{\Rank}{{\rm rk}} \newcommand{\codim}{{\rm codim}}
\newcommand{\rank}{{\rm rank}} \newcommand{\Ker}{{\rm Ker \ }}
\newcommand{\Pic}{{\rm Pic}} \newcommand{\Div}{{\rm Div}}
\newcommand{\Hom}{{\rm Hom}} \newcommand{\im}{{\rm im}}
\newcommand{\Spec}{{\rm Spec \,}} \newcommand{\Sing}{{\rm Sing}}
\newcommand{\sing}{{\rm sing}} \newcommand{\reg}{{\rm reg}}
\newcommand{\Char}{{\rm char}} \newcommand{\Tr}{{\rm Tr}}
\newcommand{\Gal}{{\rm Gal}} \newcommand{\Min}{{\rm Min \ }}
\newcommand{\Max}{{\rm Max \ }} \newcommand{\Alb}{{\rm Alb}\,}
\newcommand{\GL}{{\rm GL}\,} 
\newcommand{\ie}{{\it i.e.\/},\ } \newcommand{\niso}{\not\cong}
\newcommand{\nin}{\not\in}
\newcommand{\soplus}[1]{\stackrel{#1}{\oplus}}
\newcommand{\by}[1]{\stackrel{#1}{\rightarrow}}
\newcommand{\longby}[1]{\stackrel{#1}{\longrightarrow}}
\newcommand{\vlongby}[1]{\stackrel{#1}{\mbox{\large{$\longrightarrow$}}}}
\newcommand{\ldownarrow}{\mbox{\Large{\Large{$\downarrow$}}}}
\newcommand{\lsearrow}{\mbox{\Large{$\searrow$}}}
\newcommand{\dlog}{{\rm dlog}\,} 
\newcommand{\longto}{\longrightarrow}
\newcommand{\vlongto}{\mbox{{\Large{$\longto$}}}}
\newcommand{\limdir}[1]{{\displaystyle
{\mathop{\rm lim}_{\buildrel\longrightarrow\over{#1}}}}\,}
\newcommand{\liminv}[1]{{\displaystyle{\mathop{\rm lim}_{\buildrel
\longleftarrow\over{#1}}}}\,}
\newcommand{\norm}[1]{\mbox{$\parallel{#1}\parallel$}}
\newcommand{\boxtensor}{{\Box\kern-9.03pt\raise1.42pt\hbox{$\times$}}}
\newcommand{\into}{\hookrightarrow} \newcommand{\image}{{\rm image}\,}
\newcommand{\Lie}{{\rm Lie}\,} 
\newcommand{\CM}{\rm CM}
\newcommand{\sext}{\mbox{${\mathcal E}xt\,$}} 
\newcommand{\shom}{\mbox{${\mathcal H}om\,$}} 
\newcommand{\coker}{{\rm coker}\,} 
\newcommand{\sm}{{\rm sm}}
\newcommand{\tensor}{\otimes}
\renewcommand{\iff}{\mbox{ $\Longleftrightarrow$ }}
\newcommand{\supp}{{\rm supp}\,}
\newcommand{\ext}[1]{\stackrel{#1}{\wedge}}
\newcommand{\onto}{\mbox{$\,\>>>\hspace{-.5cm}\to\hspace{.15cm}$}}
\newcommand{\propsubset} {\mbox{$\textstyle{
\subseteq_{\kern-5pt\raise-1pt\hbox{\mbox{\tiny{$/$}}}}}$}}
\renewcommand{\propsubset} {\mbox{$\textstyle{
\subseteq_{\kern-5pt\raise-1pt\hbox{\mbox{\tiny{$/$}}}}}$}}
\newcommand{\oOplus}{\mbox{\fontsize{17.28}{21.6}\selectfont\( \oplus\)}}
\newcommand{\sA}{{\mathcal A}}
\newcommand{\sB}{{\mathcal B}} \newcommand{\sC}{{\mathcal C}}
\newcommand{\sD}{{\mathcal D}} \newcommand{\sE}{{\mathcal E}}
\newcommand{\sF}{{\mathcal F}} \newcommand{\sG}{{\mathcal G}}
\newcommand{\sH}{{\mathcal H}} \newcommand{\sI}{{\mathcal I}}
\newcommand{\sJ}{{\mathcal J}} \newcommand{\sK}{{\mathcal K}}
\newcommand{\sL}{{\mathcal L}} \newcommand{\sM}{{\mathcal M}}
\newcommand{\sN}{{\mathcal N}} \newcommand{\sO}{{\mathcal O}}
\newcommand{\sP}{{\mathcal P}} \newcommand{\sQ}{{\mathcal Q}}
\newcommand{\sR}{{\mathcal R}} \newcommand{\sS}{{\mathcal S}}
\newcommand{\sT}{{\mathcal T}} \newcommand{\sU}{{\mathcal U}}
\newcommand{\sV}{{\mathcal V}} \newcommand{\sW}{{\mathcal W}}
\newcommand{\sX}{{\mathcal X}} \newcommand{\sY}{{\mathcal Y}}
\newcommand{\sZ}{{\mathcal Z}} \newcommand{\ccL}{\sL}
 \newcommand{\A}{{\mathbb A}} \newcommand{\B}{{\mathbb
B}} \newcommand{\C}{{\mathbb C}} \newcommand{\D}{{\mathbb D}}
\newcommand{\E}{{\mathbb E}} \newcommand{\F}{{\mathbb F}}
\newcommand{\G}{{\mathbb G}} \newcommand{\HH}{{\mathbb H}}
\newcommand{\I}{{\mathbb I}} \newcommand{\J}{{\mathbb J}}
\newcommand{\M}{{\mathbb M}} \newcommand{\N}{{\mathbb N}}
\renewcommand{\P}{{\mathbb P}} \newcommand{\Q}{{\mathbb Q}}

\newcommand{\R}{{\mathbb R}} \newcommand{\T}{{\mathbb T}}
\newcommand{\U}{{\mathbb U}} \newcommand{\V}{{\mathbb V}}
\newcommand{\W}{{\mathbb W}} \newcommand{\X}{{\mathbb X}}
\newcommand{\Y}{{\mathbb Y}} \newcommand{\Z}{{\mathbb Z}}
\title{Density functions for Epsilon multiplicity and families of ideals}
\author{Suprajo Das}
\address{Department of Mathematics, Indian Institute of Technology Madras, Chennai, Tamil Nadu 600036, India}
\email{dassuprajo@gmail.com}
\author{Sudeshna Roy}
\address{Department of Mathematics, Indian Institute of Technology Madras, Chennai, Tamil Nadu 600036, India}
\email{sudeshnaroy.11@gmail.com}
\author{Vijaylaxmi Trivedi}
\address{School of Mathematics, Tata Institute of Fundamental Research, Homi Bhabha Road, Colaba, Mumbai, Maharashtra 400005, India AND Department of Mathematics, University at Buffalo (SUNY),  Buffalo, NY 14260}
\email{vija@math.tifr.res.in; vijaylax@buffalo.edu}


\begin{abstract}
A density function for an algebraic invariant is a measurable function on $\mathbb{R}$ which {\em measures} the invariant on an $\mathbb{R}$-scale. This function carries a lot more information related to the invariant without seeking extra data. It has turned out to be a useful tool, which was introduced by the third author in \cite{Tri18}, to study the characteristic $p$ invariant, namely Hilbert-Kunz multiplicity of a homogeneous ${\bf m}$-primary ideal.

Here we construct {\em density functions} $f_{A,\{I_n\}}$ for a Noetherian filtration $\{I_n\}_{n\in\N}$ of homogeneous ideals and $f_{A,\{\widetilde{I^n}\}}$ for a filtration given by the saturated powers of a homogeneous ideal $I$ in a
standard graded domain $A$. As a consequence, we get a density function $f_{\varepsilon(I)}$ for the epsilon multiplicity $\varepsilon(I)$ of a homogeneous ideal $I$ in $A$. We further show that the function $f_{A,\{I_n\}}$ is continuous everywhere except possibly at one point, and $f_{A,\{\widetilde{I^n}\}}$  is a continuous function everywhere and is continuously differentiable except possibly at one point.
As a corollary the epsilon density function $f_{\varepsilon(I)}$ is a compactly supported continuous function on $\mathbb{R}$ except at one point, such that $\int_{\mathbb{R}_{\geq 0}} f_{\varepsilon(I)} = \varepsilon(I)$.

All the three functions $f_{A,\{I^n\}}$, $f_{A,\{\widetilde{I^n}\}}$ and $f_{\varepsilon(I)}$ remain invariant under passage to the integral closure of $I$.

As a corollary of this theory, we observe that the `rescaled' Hilbert-Samuel multiplicities of the diagonal subalgebras form a continuous family.
\end{abstract}

\maketitle {}

\section{Introduction}
Several well-known algebraic invariants or properties of an ideal are studied  in terms of the powers of the ideal and often in an asymptotic way. One such invariant is the  epsilon multiplicity $\varepsilon(I)$ of an ideal $I$. It is a well-known fact, which was established by Rees, that two ${\bf m}$ primary deals $I\subseteq J$ in an equidimensional universally catenary Noetherian local ring $(A, {\bf m})$ have the same integral closure if and only if they have the same (Hilbert-Samuel) multiplicity. Due to its immense
applications in commutative algebra and algebraic geometry (especially, in the singularity theory), researchers started searching for an analogue of the multiplicity for non ${\bf m}$-primary ideals to examine integral dependence. For this the notion of epsilon multiplicity was introduced by Kleiman \cite{KUV} and Ulrich-Validashti \cite{UV08}, which was initially defined as follows:

If $(A, {\bf m})$ is a $d$-dimensional local ring and $I$ is an ideal in $A$, then the epsilon multiplicity $\varepsilon(I)$ of $I$ is $$\varepsilon(I) = \limsup_{n\to \infty}\dfrac{\ell_A(H^0_{\bf m}(A/I^n))}{n^d/d!} = \limsup_{n\to \infty}\frac{\ell_A({\widetilde{I^n}}/I^n)}{n^d/d!},$$ because by definition $H^0_{\bf m}(A/I^n) = {\widetilde{I^n}}/I^n$, where ${\widetilde{I^n}} = (I^n:_A{\bf m}^{\infty})$ denotes the saturation of the ideal $I^n$. In particular, if $I$ is ${\bf m}$-primary then it coincides with the usual multiplicity of $I$. It was shown by Cutkosky \cite{Cut14} that the `limsup' can be replaced by `lim' under very mild conditions on $A$. In \cite{CHST05}, the authors gave an example of an ideal in a regular local ring having irrational epsilon multiplicity, which suggests that it might be a difficult invariant to compute in general.
 
In this paper we work in a graded setup and  construct a {\em density function} for $\varepsilon(I)$.
Roughly speaking, a density function for a given invariant, say $\epsilon$, is an integrable function $f_{\epsilon}:\R\longto \R$ which gives a {\em new measure} $\mu_{f_{\epsilon}}$ on $\R$ such that the integration $\int_Ef_{\epsilon}$  on a subset $E\subset \R$ is the  measure of the invariant $\epsilon$ on $E$.
The notion of density function, though seemingly complicated, keeps track of the behaviour of the `original' invariant in more than one way, and also carries information about other related invariants. Such properties were used while studying Hilbert-Kunz multiplicity, which seems hard to compute, and one was able to answer two long-standing questions in \cite{Tri23} and \cite{Tri20}. In \cite{Tri20} one of the crucial steps is the fact (\cite[Theorem 6.4]{TW21}) that the support of the Hilbert-Kunz density function determines the $F$-threshold $c^I({\bf m})$ of $I$ at ${\bf m}$. See \cite{MT20}, \cite{TW22}, \cite{MM24}, and \cite{CSTZ24} for some recent works on these kind of density functions.

Following is a quick  summary of the main results in this paper. Let $A = \oOplus_{n\geq 0}A_n$ be a standard graded finitely generated algebra over a field $A_0 = k$, with homogeneous maximal ideal ${\bf m} = \oOplus_{n\geq 1}A_n$. Further, assume that $A$ is a domain of Krull dimension $d\geq 2$. Let $I\subsetneq A$ be a nonzero homogeneous ideal. Suppose that $I$ is minimally generated in degrees $d_1 < d_2< \cdots < d_l$
(which is well-defined by Nakayama's lemma).
We define a function $f_{A,\{I^n\}}:\R_{\geq 0}\longto 
\R_{\geq 0}$
as $$f_{A,\{I^n\}}(x) = \limsup_{n\to \infty}\frac{\ell_k \big((I^n)_{\lfloor xn\rfloor}\big)}{n^{d-1}/d!}.$$
The following result is proved in Theorem \ref{t3}
in a more general form. A crucial input comes from a result of Sturmfels \cite{Stu95} on vector partition functions.

\begin{thm}\label{0.1}Let the notations be as above. Then the following statements are true:
 \begin{enumerate}
 \item $f_{A,\{I^n\}}(x)$ exists as a limit for all $x\in\mathbb{R}_{\geq 0}$ and it is continuous everywhere except possibly at $x=d_1$.
  \item Further $$f_{A,\{I^n\}}(x) = \begin{cases}
                                      0 & \forall x\in [0, d_1),\\
                                      {\bf p}_1(x) & \forall x\in (d_1, d_2],\\
                                      {\bf p}_j(x) & \forall x\in [d_j, d_{j+1}]~~\mbox{and}~~j=2, \ldots, l-1,\\
                                      {\bf p}_l(x) & \forall x\in [d_l, \infty),
                                     \end{cases}$$
  where for each $j=1,\ldots, l$, ${\bf p}_j(x)$ is
  a nonzero rational polynomial of degree $\leq d-1$. Moreover, ${\bf p}_l(x)$ has degree $d-1$.
  \item For all integers $c>0$, we have 
  $$\int_0^cf_{A,\{I^n\}}(x)dx = \lim_{n\to \infty} \frac{\sum_{m=0}^{cn}\ell_k\big((I^n)_m\big)}{n^d/d!}.$$
 \end{enumerate}
\end{thm}

Next we formulate and study a density function for the family of ideals $\{\widetilde {I^n}\}_n$.
If the ideal $I$ is $\mathfrak{m}$-primary then $\widetilde{I^n} = A$ for all $n\geq 0$.
Therefore we assume that $I$ is non $\mathfrak{m}$-primary. Let $\sI$ be the ideal sheaf associated to $I$ on $V = \mbox{Proj}~A$. Let
$$\pi:X = {\bf Proj} \left(\oplus_{n\geq 0}\mathcal{I}^n\right) \longto V$$
be the blow up of $V$ along $\sI$ with the exceptional divisor $E$. Let $H$ be the pull back of a hyperplane 
section on $V$. Define the constants
$$\alpha_I = \min\{x\in \R_{\geq 0}\mid xH-E~~~~\mbox{is pseudoeffective}\},\quad 
   \beta_I = \min\{x\in \R_{\geq 0}\mid xH-E~~~~\mbox{is nef}\}
$$
and define the density  function $f_{A,\{\widetilde {I^n}\}}:\R_{\geq 0}\longto \R_{\geq 0}$ as
$$f_{A,\{\widetilde {I^n}\}}(x) = \limsup_{n\to \infty}\frac{\ell_k\big((\widetilde {I^n})_{\lfloor xn\rfloor}\big)}{n^{d-1}/d!}.$$ When the saturated powers and symbolic powers coincide (e.g. if $\dim A/I=1$) then $\alpha_I$ is precisely the `Waldschmidt constant'. The definition of $\beta_I$ is taken from \cite{CEL01}, where it is called the `$s$-invariant' of the ideal sheaf $\mathcal{I}$.

We prove the following result in Section \ref{sec5} under the additional assumption that the underlying field $k$ is algebraically closed. However, Remark \ref{rmkfield} justifies the validity of these results over an arbitrary field. For the proof, we use several results about the volume function of $\R$-divisors, as given in \cite{Laz04a}, \cite{Laz04b}, \cite{LM09}, \cite{Cut14}, and \cite{Cut15}. A slightly weaker result in positive characteristic is obtained simply by using Frobenius morphism and some basic results in algebraic geometry.

\begin{thm}\label{0.2} Let the notations be as before. Then the following statements are true:
 \begin{enumerate}
 \item  $f_{A,\{\widetilde {I^n}\}}(x)$ exists as a limit for all $x\in \mathbb{R}_{\geq 0}$, $f_{A,\{\widetilde {I^n}\}}:\R_{\geq 0}\longto \R_{\geq 0}$ is continuous and
 $$ f_{A,\{\widetilde {I^n}\}}(x) = d\cdot \mathrm{vol}_X(xH-E).$$ Further, it is a continuously differentiable function everywhere except possibly at $x=\alpha_I$.
 \item If $x \leq \alpha_I$, then
 $f_{A,\{\widetilde {I^n}\}}(x) = 0$.
 \item If $x\geq \beta_I$, then 
$$ f_{A,\{\widetilde {I^n}\}}(x) = d\cdot (xH-E)^{d-1}
= \sum_{i=0}^{d-1}(-1)^i
\frac{d!}{(d-1-i)!i!}(H^{d-1-i}\cdot E^i)x^{d-1-i},$$
where $H^{d-1-i}\cdot E^i$ denotes the intersection number of the Cartier divisors $H^{d-1-i}$ and $E^i$.
 \item For all real numbers $y\geq x \geq \alpha_I$, there is an 
 inequality
$$
 \left(f_{A,\{\widetilde {I^n}\}}(y)\right)^{\frac{1}{d-1}} \geq \left(f_{A,\{\widetilde {I^n}\}}(x)\right)^{\frac{1}{d-1}} + (y-x)(d\cdot e_0(A))^{\frac{1}{d-1}},$$
where $e_0(A)$ denotes the multiplicity of $A$.
\item For all integers $c>0$, we have 
  $$\int_0^cf_{A,\{\widetilde {I^n}\}}(x)dx = \lim_{n\to \infty} \frac{\sum_{m=0}^{cn}\ell_k\big((\widetilde {I^n})_m\big)}{n^d/d!}.$$
 \end{enumerate}
\end{thm}

Now we are ready to construct a density function associated to the $\varepsilon$-multiplicity.

\begin{thm}\label{0.3}
Let the notations be as before. Then the following statements are true:
\begin{enumerate}\item The limit 
 $$f_{\varepsilon(I)}(x) = \lim_{n\to\infty}
 \frac{\ell_k\left(H^0_{\bf m}(A/I^n)\right)_{\lfloor xn \rfloor}}
 {n^{d-1}/d!}$$
 exists for all $x\in \R_{\geq 0}$. Moreover,
 $f_{\varepsilon(I)}(x) = f_{A,\{\widetilde {I^n}\}}(x)-
 f_{A,\{I^n\}}(x)$.
 \item The function $f_{\varepsilon(I)}$ is continuous everywhere except possibly at $x= d_1$. Further, it is continuously differentiable outside the finite set $\{\alpha_I, d_1,\ldots, d_l\}$.
 \item The support of $f_{\varepsilon(I)}$
 is contained in the closed interval $[\alpha_I, d_l]$.
 \item 
 $$\int_0^{\infty} f_{\varepsilon(I)}(x)dx = \lim_{n\to\infty}
 \frac{\ell_A\left(H^0_{\bf m}(A/I^n)\right)}{n^d/d!} =: \varepsilon(I).$$
 \end{enumerate}
\end{thm}

The function $f_{\varepsilon(I)}$ will be called the $\varepsilon$-density function. Here we list some applications of density functions (see Sections \ref{sec5}, \ref{sec6} and \ref{sec7} for more details).

\begin{cor} Let the notations be as before. Then the following statements are true:
 \begin{enumerate}
  \item We have $0<\alpha_I \leq d_1$, (it is known that $\alpha_I\leq \beta_I \leq d_l$).
  \item If $\alpha_I = \beta_I$, then $\varepsilon(I)$ is an algebraic number.
  \item If $\dim A = 2$, then $f_{\varepsilon(I)}$ is a piecewise (with finitely many pieces) rational polynomial of degree at most one and
  $\varepsilon(I)\in \Q$.
  \item If $\dim A = 3$ then $f_{\varepsilon(I)}$ is a piecewise (with countably many pieces) real polynomial of degree at most two.
  \item There exist constants $c\geq 0$ and $\beta\geq 0$ such that for all $m\geq nd_l+c$ and $n\geq \beta$,
  $$\ell_k\left((I^n)_m\right) =  \sum_{i=0}^{d-1}
  \dfrac{(-1)^i}{(d-1-i)!i!}(H^{d-1-i}\cdot E^i)m^{d-1-i}n^i
  +Q_l(m,n),$$
  where $Q_l(X,Y) \in \Q[X,Y]$ is a polynomial of degree $<d-1$ (see Remark~\ref{rmkintro}). In particular, the mixed multiplicity $e_i(A[It])$ of the Rees algebra of the ideal $I$ is given in terms of the intersection number (upto sign) of the Cartier divisors $H^{i}$ and $E^{d-1-i}$ on
  $X$.
  \item The multiplicities of the diagonal subalgebras (see Definition \ref{dg}) $R_{\Delta_{I(p,q)}}$ of $R:=A[It]$ are determined by the multiplicity $e\big(R_{\Delta_{I(d_1,1)}}\big)$ and polynomials ${\bf p}_1(x),\ldots,{\bf p}_l(x)$ as in Theorem \ref{0.1}: Given $(p,q)\in\mathbb{N}^2$,
  \begin{align*}
   p/q = d_1 &\implies e(R_{\Delta_{I(d_1,1)}}) = q^{d-1}e(R_{\Delta_{I(d_1,1)}}),\\
   p/q \in I(d_j,d_{j+1}) &\implies e(R_{\Delta_{I(p,q)}}) = \tfrac{q^{d-1}}{d}\cdot \mathbf{p}_j(p/q).
  \end{align*}
  In particular, the rescaled multiplicities of the diagonal subalgebras $\{R_{\Delta_{I(p,q)}}\mid p/q \neq d_1\}$
  form a continuous family.
 \end{enumerate}
\end{cor}

It should be pointed out that the notion of mixed multiplicities of the Rees algebra \cite[Section 4]{HT03} differs from the classical notion of mixed multiplicities due to Teissier, see \cite{KV89}. However, in Subsection \ref{mixed_mult}, we show how these two notions are related. We also obtain their interpretation in terms of intersection numbers, which does seem to appear in the literature.

We also observe (Theorem \ref{6.1}) that all the above three density functions remain invariant under passage to the integral closure of $I$. In particular, if $I$ and $J$ are two homogeneous ideals with the same integral closure  then
$$f_{\varepsilon(I)}(x) = f_{\varepsilon(J)}(x),\quad\mbox{for all}\;\; x\in \R_{\geq 0}.$$

\begin{rmk}
The converse question about equality of these density functions has been recently answered by the authors in \cite{DRT24}. By using the above theory of density functions, we gave a set of numerical characterizations of the integral dependence of ideals in a graded setup, where the invariants involved are computatable and well-studied.
\end{rmk}

We briefly describe the techniques which go in the proofs. We first study the length function
$$\ell_k:\cup_j{\mathfrak R}C_j\longto \Z_{\geq 0}\quad \mbox{given by}\quad  (m,n)\mapsto  \ell_k(I^n)_m,$$
 where
$\{{\mathfrak R}C_j\}_{1\leq j\leq l}$ is a set of cones in $\R_{\geq 0}^2$. In fact we do it for more general family, namely any Noetherian filtration $\{I_n\}_{n\in\N}$ of ideals, but to avoid technical notations here we  assume  $\{I_n\}_{n\in\N} = \{I^n\}_{n\in\N}$, where   $I$ is generated by homogeneous elements of degrees $d_1<d_2<\cdots <d_l$. So far, other than on diagonal subalgebras,
the behaviour of length function has been studied (see \cite{Rob98}, \cite{HT03}),
on the set 
$\{(m,n)\mid m\geq d_ln+c,~~~n\geq \beta\}$,  where $c$ and $\beta$ are nonnegative constants, {\em i.e.}, for $(m,n)$ lying in the last cone ${\mathfrak R}C_l$.  We recall that for $(m,n)$ belonging to this cone has the nice property that $(m,n)$-diagonal subalgebra is a standard graded ring generated  by $(I^n)_m$.

To elaborate further we need to recall some notations used in the paper. Here $A=\oOplus_{n\geq 0}A_n$ denotes a standard graded Noetherian domain over a field $k$ and $I$ be a homogeneous ideal in $A$. The bigraded algebra $A[It] = \oOplus_{n\geq 0}I^nt^n$ is generated by finitely many homogeneous elements of bidegrees $\{(d_i, 1)\mid 1\leq i\leq l\}$ over $A$.
The bidegrees of $A[It]$ and the Betti numbers of the resolution of  $A[It]$ give the following data in the Euclidean plane.
\begin{enumerate}
 \item For $1\leq j <l$ let ${\mathfrak{C_j}}$ denote the cone in $\R^2$ generated by the vectors
$(d_j, 1)$ and $(d_{j+1}, 1)$ and ${\mathfrak{C_l}}$ be the cone generated by $(d_l, 1) $ and $(1,0)$.
\item  $(\lambda_1, \beta_1),  \ldots, 
(\lambda_l, \beta_l)\in \Q^2_{\geq 0}$ are associated to the Betti numbers (see Remark~\ref{restcone}).
\item ${\mathfrak R}C_j = {\mathfrak C_j}+ (\lambda_j, \beta_j)$ are the  {\em restricted cones} which  are nothing but  the translates of ${\mathfrak C_j}$ by the vector $(\lambda_j, \beta_j)$ (see Figure \ref{cone_polynomials}).
\end{enumerate}

We prove the following (for more general statement on Noetherian family $\{I_n\}_{n\in\N}$ see Theorem \ref{t1} and Theorem \ref{t3}).

\begin{thm}\label{0.5}
 There exist a set of homogeneous polynomials $\{P_1(X,Y)$, $\ldots, P_l(X,Y)\}$ and a set of quasi-polynomials $\{Q_1(X,Y), \ldots, Q_l(X,Y)\}$ in $\Q[X, Y]$ such that for all $j$, $\deg Q_j(X,Y) < \deg P_j(X,Y) = d-1$ and
 $$(m,n)\in {\mathfrak R}C_j\cap \N^2 \implies \ell_k (I^n)_m = P_j(m,n) + Q_j(m,n),$$
and $Q_l(X,Y)$ is a polynomial.
\end{thm}

In Example \ref{Nfiltration} we list a few interesting families of homogeneous ideals which do not arise as a powers of an ideal but do form Noetherian families. So even for the last restricted cone the result is more general.

The proof of Theorem \ref{0.5} makes crucial use of vector partition theorem (Theorem \ref{corrector}) due to Sturmfels \cite{Stu95}. The vector partition theorem gives that the length function on  the restricted cones ${\mathfrak R}C_j$ is a quasi-polynomial function with period given in terms of $(d_j,1)$. Further (see Corollary \ref{HT}) the length function restricted to the last chamber ${\mathfrak R}C_l$ has no periodicity and therefore is a polynomial function. This recovers of \cite[Theorem 1.1]{HT03}.

However, if ${\mathfrak R}C_j$ is not the last chamber then Example \ref{7.4} shows that the length function  may not be a polynomial function on ${\mathfrak R}C_j$. The main point of  Theorem \ref{0.5} is to show, nevertheless, that the top degree terms, namely, $P_j(X,Y)$ of the length function is a homogeneous polynomial function of degree $d-1$.

Now if $x\in I(d_j, d_{j+1})$ then (see Lemma \ref{density2} and Figure \ref{restrictedCone})  for $n\gg 0$ the vector $(\lfloor xn \rfloor, n)$ belongs to the restricted cone ${\mathfrak R}C_j$.
This along with the above theorem  establishes that the density function $f_{A,\{I^n\}}$ is a polynomial function on each of the open intervals $I(d_j, d_{j+1})$.

However, if  $x=d_j$, where $j>1$, then the vector $(d_j, 1)$ belongs to the  boundaries of the cones $\mathfrak{C}_j$ and $\mathfrak{C}_{j+1}$ and therefore for any $n\gg 0$, the vector  $ ( xn, n)$ does not belong to any of the restricted cones. 

But in the case when $A$ is a domain, we have a nonzerodivisor of degree $(d_1,1)$ in $A[It]$. We use this nonzerodivisor along with elementary Lemmas~\ref{i-1} and \ref{lrr1} in Euclidean plane geometry  to approximate the elements of degrees $( d_jn, n)$, for $n\gg 0$, with some elements in ${\mathfrak R}C_{j-1}$ and ${\mathfrak R}C_j$ and thus prove the continuity of $f_{A,\{I^n\}}$ on $\R\setminus \{d_1\}$.
 
Next we construct the density function $f_{A,\{\widetilde {I^n}\}}$ for the family $\{\widetilde {I^n}\}_{n\in\N}$. Note that in general $\{\widetilde {I^n}\}_{n\in\N}$ (see Example \ref{nagata}) need not form a Noetherian filtration of ideals. In particular, the above approach for constructing $f_{A,\{I_n\}}$ does not work here.

Here we  consider a sequence $\{g_n:\R_{\geq 0}\longto \R_{\geq 0}\}_n$, where  $g_n(x) =  \frac{\ell_k\left((\widetilde {I^n})_{\lfloor xn\rfloor}\right)}{n^{d-1}/d!}$. To construct density function $f_{A,\{\widetilde {I^n}\}}$ in characteristic $p>0$ we give two approaches. The first one (Theorem \ref{t1sat}) uses basic results from algebraic geometry. Here we restrict to a subsequence $\{g_q\mid q = p^n\}_{n\in \N}$ and prove that it is uniformly convergent sequence on any given compact subset of $\R$ and converges to a continous function denoted as $f_{A,\{\widetilde {I^q}\}}$. Further the function $f_{A, \{\widetilde {I^q}/I^q\}}$, which is defined as
$$f_{A,\{\widetilde {I^q}/I^q\}}(x) = f_{A,\{\widetilde {I^q}\}}(x) - f_{A,\{I^q\}}(x)\quad \mbox{gives}\quad
\int_{\R} f_{A,\{\widetilde {I^q}/I^q\}}(x)dx = \lim_{q\to \infty}
\frac{\ell_k(\widetilde {I^q}/I^q)}{q^{d-1}/d!}.$$
Now the result \cite[Theorem $2.3$]{CHST05} or \cite[Corollary $6.3$]{Cut14} about the existence of the epsilon multiplicity $\varepsilon(I)$ as a limit implies that $f_{\{\widetilde {I^q}/I^q\}}$ is the density function for $\varepsilon(I)$.

In the case of arbitrary characteristic, we prove in Theorem \ref{t2sat} the convergence of $\{g_n\}_{n\in \N}$ using several results about volume functions of $\R$-divisors (following \cite{Laz04a}, \cite{LM09} and \cite{Cut14}). In fact, it turns out that the limit function $f_{A,\{\widetilde {I^n}\}}(x)$ equals $d\cdot\mbox{vol}_X(xH-E)$ for all $x\in \R_{\geq 0}$.
This approach gives another proof of the existence of the limit 
$\varepsilon(I) = \lim\limits_{n\to \infty}\frac{\ell_A(\widetilde {I^n}/I^n)}{n^{d-1}/d!}$ (in the graded case), and also implies that $f_{A,\{\widetilde {I^q}\}} = f_{A,\{\widetilde {I^n}\}}$ when $A$ has characteristic $p>0$.

In the last two sections we give some computations and properties of density functions.

\begin{quest}In fact, Example \ref{exr6} shows that $f_{A,\{I^n\}}$ can be discontinuous at $x=d_j$ if $A$ is not a domain.
Let $A$ be a standard graded ring over a field $k$ and of dimension $d\geq 2$ and let $I$ be an ideal with homogeneous generators occuring in degrees $d_1<d_2<\cdots<d_l$. Is it true that the function $f_{A,\{I^n\}}$ is continuous on the set $\mathbb{R}_{\geq 0}\setminus \{d_1\}$.
\end{quest}
We know the answer is affirmative (Corollary \ref{limit}) if $I$ has a nonzerodivisor of degree $d_1$.
We also know that the answer is negative if $d=1$ and $A$ is not a domain, as
in  Example \ref{exr6} where   $f_{A,\{I^n\}}$ is discontinuous at $x=d_1, d_2$ and $d_3$.

\vspace{5pt}

Since the invariant $\varepsilon$-multiplicity is defined over local rings, it is natural to ask the following question.

\begin{quest}
Let $(A,{\bf m})$ be a local ring and $I$ be an ideal in $A$. Can we define a notion of density functions associated to the filtrations $\{\widetilde{I^n}\}_{n\in\mathbb{N}}$ and $\{I^n\}_{n\in\mathbb{N}}$ respectively?
\end{quest}

\section{Vector partition function}

\begin{defn}
Let $M$ be an $m_1\times n_1$ matrix of rank $m_1$ with entries in $\N$. Then the corresponding \emph{vector partition function} $\phi_M\colon \mathbb{N}^{m_1} \to \mathbb{N}$ is defined as
$$\phi_M({\bf u}) = \#\left\{[\lambda]=\colvec{\lambda_1\\ \vdots\\ \lambda_{n_1}}\in\mathbb{N}^{n_1} ~\big|~ M\cdot [{\bf \lambda}] = {\bf u}  \right\}, \quad \mbox{where}\;{\bf u} = \colvec{u_1\\ \vdots\\ u_{m_1}}\in \N^{m_1}$$
\end{defn}

We recall the notations and a result due to Sturmfels \cite{Stu95} which gives a polynomial expression of the function $\phi_M({\bf u})$ in terms of the vector ${\bf u}$. All vectors will be treated as column vectors.

\begin{notations}\label{posh}
Let ${\bf v_i}$ denote the $i^{\rm th}$ column vector of the matrix $M$. Define the polyhedral cone
$$\mathrm{pos}(M) = \Big\{\sum\limits_ {i=1}^{n_1} \lambda_i{\bf v_i}\in\mathbb{R}^{m_1} ~\big|~ \lambda_1,\ldots,\lambda_{n_1} \in \mathbb{R}_{\geq 0}\Big\}.$$

For $\sigma \subset \{1,\ldots,n_1\}$, 
\begin{enumerate}
\item $M_{\sigma} = [{\bf v_i} \mid i\in\sigma]$ is the submatrix of $M$.
\item $\mathrm{pos}(M_{\sigma})$ denotes the polyhedral cone generated by the vectors $\{{\bf v_i} \mid i\in\sigma\}$.
\item $M_{\sigma}\mathbb{Z} \subset \Z^{m_1}$ is the  abelian group spanned by the columns of $M_{\sigma}$.
\end{enumerate}

A subset $\sigma$ of $\{1,\ldots,n_1\}$ is a \emph{basis} if $\#\sigma = \mathrm{rank}(M_{\sigma})= m_1$. The \emph{chamber complex} is the polyhedral subdivision of the cone $\mathrm{pos}(M)$ which is defined as the common refinement of cones $\mathrm{pos}(M_{\sigma})$, where $\sigma$ runs over all bases. By a \emph{chamber}, we mean the polyhedral cone corresponding to a basis $\sigma$. For a chamber ${\mathfrak C}$, let $$\Delta({\mathfrak C}) = \left\{\sigma\subset\{1,\ldots,n_1\}\mid
{\mathfrak C}\subseteq \mathrm{pos}(M_{\sigma})\right\}.$$ 

Since the group $M_{\sigma}\mathbb{Z}$ has finite index in $\mathbb{Z}^{m_1}$ for each $\sigma\in\Delta({\mathfrak C})$, the group of residue classes
$G_{\sigma}:=\mathbb{Z}^{m_1}/M_{\sigma}\mathbb{Z}$ is a finite group. The image of the element ${\bf u}\in \Z^{m_1}$ under the quotient map $\Z^{\sigma}\rightarrow  G_{\sigma}$ is denoted by $[{\bf u}]_{\sigma}$. We say  $\sigma \in\Delta({\mathfrak C})$ is \emph{non-trivial} if $G_{\sigma}\neq \{0\}$.
\end{notations}

\begin{thm}\label{corrector}\cite[Theorem $1$]{Stu95}
If $M$ is an $m_1\times n_1$ matrix of nonnegative integers with $m_1\leq n_1$ and $\phi_M\colon \N^{m_1}\longrightarrow \N$ is the corresponding vector partition function given by
$${\bf u} \longmapsto \#\bigg\{[\lambda] = \colvec{\lambda_1\\ \vdots\\ \lambda_{n_1}}\in \N^{n_1}\mid M\cdot [\lambda] = {\bf u}\bigg\}.$$
Then for  each chamber ${\mathfrak C}$ there exists a polynomial $P_{\mathfrak C}(X_1, \ldots, X_{m_1})\in\mathbb{Q}[X_1,\ldots,X_{m_1}]$ of degree $n_1-m_1$ and for each non-trivial $\sigma\in\Delta({\mathfrak C})$ there exists a polynomial $Q^{\mathfrak C}_{\sigma}\in\mathbb{Q}[X_1, \ldots,X_{m_1}]$ of degree
$(\#\sigma) -m_1$ and a function $\Omega_{\sigma}\colon G_{\sigma}\setminus\{0\}\to \mathbb{Q}$ such that, for  ${\bf u}\in {\mathfrak C} \cap \mathbb{Z}^{m_1}$,
$$\phi_M({\bf u}) = P_{\mathfrak C}({\bf u}) + \sum\limits_{\sigma\in\Delta({\mathfrak C}),\;[{\bf u}]_{\sigma}\neq 0} \Omega_{\sigma}\left([{\bf u}]_{\sigma}\right)\cdot Q^{\mathfrak C}_{\sigma}({\bf u}).$$
\end{thm}

The polynomials $Q^{\mathfrak C}_{\sigma}$ in the above theorem are called \emph{corrector polynomials}. In this paper we will be using the following particular case of vector partition function.

\begin{notations}\label{VPT}
$M$ is a $2\times (r+s)$ matrix  with entries
\begin{equation}\label{matrixA}M = \begin{bmatrix}
1 &\dots &1 &d_1 &\dots &d_s\\
0 &\dots &0 &e_1 &\dots &e_s
\end{bmatrix},
\end{equation}
where $d_i \in\N$ and $e_i \in\N_{>0}$ for $i=1,\ldots,s$.

The corresponding \emph{vector partition function} $\phi_M\colon \mathbb{N}^{2} \to \mathbb{N}$ is defined as
$$\phi_M(m,n) = \#\left\{[\lambda]=\colvec{\lambda_1\\ \vdots\\ \lambda_{r+s}}
                                                                                    \in\mathbb{N}^{r+s} ~\Big|~ M\cdot [{\bf \lambda}] = {\bf u} = \colvec{m\\
n}\right\}.$$

Without any loss of generality we may assume that ${d_1}/{e_1}<{d_2}/{e_2}<\cdots <d_l/e_l$ are all the distinct elements in the set $\{d_j/e_j\mid 1\leq j\leq s\}$. It is easy to see that the  maximal cells in the chamber complex, is the polyhedral cone given by the vectors  $\{({d_j}, {e_j}),~({d_{j+1}}, {e_{j+1}})\}$ or $\{({d_l}, {e_l}), (1,0)\}$. We denote the vector $(1, 0)$ by $(d_{l+1}, e_{l+1})$. Define ${\mathfrak C_{j}} = \{\lambda_1(d_j, e_j)+\lambda_2(d_{j+1}, e_{j+1})\mid \lambda_1 , \lambda_2\in \R_{\geq 0}\}$ for $1\leq j\leq l$. Then we can write
$$\mbox{pos}(M) = {\mathfrak C_1}\cup {\mathfrak C_2}\cup\cdots\cup {\mathfrak C_l}.$$
\end{notations}

In fact in this set each group $G_{\sigma}= \Z^2/M_{\sigma}\Z$ is a quotient of a product of two cyclic groups.

\begin{lemma}\label{density1}
There exists an integer $h>0$ such that for every $\sigma\subset\{1,\ldots,r+s\}$ the quotient map $\Z^2\to G_{\sigma}$ factors through $\Z/h\Z\oplus \Z/h\Z\to G_{\sigma}$. In particular, if $[(m, n)]_{\sigma}$ denotes the image of $(m, n)\in \Z^2$ in $G_{\sigma}$ then
$$[(m,n)]_{\sigma} = [(m+h,n)]_{\sigma} = [(m,n+h)]_{\sigma}.$$
\end{lemma}
\begin{proof}
Note that the matrix $M_{\sigma}$ has full rank. So $G_{\sigma}$ is a finite abelian group. Therefore, $M_{\sigma}\Z$ must contain vectors of the form $(h,0)$ and $(0,h)$ for some integer $h>0$. In particular, the quotient map $\Z^2\to G_{\sigma}$ factors through $\Z/h\Z\oplus \Z/h\Z\to G_{\sigma}$. The second assertion is obvious.
\end{proof}

\begin{rmk}\label{trivial}
Suppose that $(d_j, e_j) = (d_j, 1)$ for all $1\leq j\leq s$. Then the group $G_{\sigma}$ is trivial for $\sigma \in \Delta({\mathfrak C_l})$. This is because
$\sigma$ contains linearly independent vectors $(1,0)$ and $(d_j,1)$ which also generates $\Z^2$ as $\Z$-modules. So $[(m,n)]_{\sigma} = 0$ for all $(m,n)\in {\mathfrak C_l}$. In particular, $\phi_M({\bf u}) = P_{{\mathfrak C_l}}({\bf u})$.
\end{rmk}

The following subsection is independent of the above part of the section.

\subsection{Membership in a cone}

\begin{notations}\label{chambers}
\hspace{2em}
\begin{enumerate}
\item Let $(d_1, e_1), \ldots, (d_l, e_l)\in \N\times \N_{> 0}$ such that $d_1/e_1<d_2/e_2<\cdots <d_l/e_l$.
\item Let $\mathfrak{C}_j\subset \R^2$ be the cone generated by the vectors $(d_j, e_j)$ and $(d_{j+1}, e_{j+1})$, where we define $(d_{l+1}, e_{l+1}) = (1,0)$.
We define $\mathrm{interior}\,(\mathfrak{C}_j) = \{r_1(d_j, e_j)+r_2(d_{j+1}, e_{j+1})\mid r_1, r_2 \in \R_{>0}\}$.
\item Further let $\lambda_j, \beta_j \in \Q_{\geq 0}$ such that
$$\lambda_j-\beta_jd_j/e_j \geq 0\quad\mbox{and}\quad \beta_{j}d_{j+1}-\lambda_je_{j+1}\geq 0\quad\mbox{for}\quad 1\leq j\leq l.$$
\item Let ${\mathfrak R}C_j={\mathfrak C}_j+(\lambda_j, \beta_j)$, for $1\leq j \leq l$ such that ${\mathfrak R}C_j\subseteq {\mathfrak C}_j$, where ${\mathfrak C}_j+(\lambda_j, \beta_j)$ denotes the translation of the cone ${\mathfrak C}_j$ from the point $(0, 0)$ to $(\lambda_j, \beta_j)$.
\item We define $d_{j}/{e_{j}} = \infty$ if the vector $(d_{j}, e_{j}) = (1, 0)$.
\item To avoid confusion with a tuple $(r_1, r_2) \in \R^2$ we will denote the open interval $(r_1, r_2)$ by $I(r_1, r_2)$ and the closed interval $[r_1, r_2]$ by $I[r_1, r_2]$.
\end{enumerate}
\end{notations}

\begin{lemma}\label{density0}
Let $m_1$ and $m_2$ be two positive real numbers. Then
\begin{enumerate}
\item $\frac{m_1}{m_2} \in I[\frac{d_j}{e_j},\frac{d_{j+1}}{e_{j+1}}]\iff (m_1, m_2)\in \mathfrak{C}_j$, and $\frac{m_1}{m_2} \in I(\frac{d_j}{e_j},\frac{d_{j+1}}{e_{j+1}})\iff (m_1, m_2)\in \mathrm{interior}\,(\mathfrak{C}_j).$
\item $(a, b)\in \mathfrak{R}C_{j}\;\;\mbox{and}\;\; (m_1, m_2)\in\mathfrak{C}_j \implies (a+m_1, b+m_2)\in \mathfrak{R}C_{j}.$
\end{enumerate}
\end{lemma}
\begin{proof}
\emph{Assertion $(1)$}.\quad If $\frac{m_1}{m_2} \in I(\frac{d_{j}}{e_{j}},\frac{d_{{j}+1}}{e_{{j}+1}})$ then there exists $0< \eta <1$ such that
$$\tfrac{m_1}{m_2} = \eta\tfrac{d_{j}}{e_{j}}+(1-\eta)\tfrac{d_{{j}+1}}{e_{{j}+1}}$$
which implies that
$$(m_1, m_2) = \eta\tfrac{m_2}{e_{j}}(d_{j}, e_{j})+(1-\eta) \tfrac{m_2}{e_{{j}+1}}(d_{{j}+1}, e_{{j}+1}) \in \mbox{interior}\,(\mathfrak{C}_{j}).$$

Conversely, let $(m_1, m_2) \in \mbox{interior}\,(\mathfrak{C}_{j})$. Then there exist $\lambda_1,\lambda_2 > 0$ such that
\begin{align*}
 (m_1, m_2) &= \lambda_1(d_{{j}},e_{{j}}) + \lambda_2(d_{{j}+1},e_{{j}+1})\\
 &= \Big(\lambda_1e_{j}\tfrac{d_{{j}}}{e_{j}}+\lambda_2 e_{j+1}\tfrac{ d_{j+1}}{e_{j+1}}, \lambda_1e_{j}+\lambda_2e_{j+1}\Big).
\end{align*}
Let $\eta = \frac{\lambda_1e_j}{\lambda_1e_j+\lambda_2e_{j+1}}$. Then $0<\eta <1$ and so
$$\tfrac{m_1}{m_2} = \eta\tfrac{d_j}{e_j}+(1-\eta)\tfrac{d_{j+1}}{e_{j+1}}\in I\big(\tfrac{d_j}{e_j},\tfrac{d_{j+1}}{e_{j+1}}\big).$$

The arguments regarding the equivalent criteria for $\tfrac{m_1}{m_2}\in I[\tfrac{d_j}{e_j}, \tfrac{d_{j+1}}{e_{j+1}}]$ are identical.

\vspace{5pt}

\noindent{\em Assertion $(2)$}.\quad It follows from the definition that ${\mathfrak{R}}C_{j} = \mathfrak{C}_{j} + (\lambda_{j}, \beta_{j})$ for some $\lambda_{j}, \beta_{j} \in \Q_{\geq 0}$ and $\mathfrak{C}_j$ is closed under pointwise addition of the vectors.
\end{proof}

\begin{lemma}\label{density2}
Let $x\in I(\frac{d_{j}}{e_{j}}, \frac{d_{{j}+1}}{e_{{j}+1}})$ be a fixed real number. Then for any given finite set $\mathcal{S}\subset \mathbb{\R}^2$ there exists an integer $n_{\mathcal{S}}$ such that $\left(\lfloor xn\rfloor + \alpha, n+\beta\right)\in \mathfrak{R}C_{j}$ for every $n\geq n_{\mathcal{S}}$ and all $(\alpha, \beta)\in \mathcal{S}$.
\end{lemma}
\begin{proof}
Let $\mathfrak{R}C_j = \mathfrak{C}_j + (\lambda_j, \beta_j)$. It is enough to prove that for given $(\alpha, \beta)\in \mathcal{S}$ there is an integer

\vspace{0.15cm}
\begin{minipage}{.25 \columnwidth}
\centering
\captionsetup{type=figure}
\begin{tikzpicture}[scale=0.8, font=\tiny]
\begin{axis}[axis equal image,
xmin=-0.75, xmax=6,
ymin=-0.75, ymax=8,
axis lines=middle,
axis line style={-latex},
ytick=\empty,
xtick=\empty,
xlabel=$m$,ylabel=$n$,
]
\node[red] at (1.5,1.5){\tiny \textbullet};
\node at (0.45,-0.3) {(0,0)};
\node[black] at (2.25,1.75){\tiny \textbullet};
\addplot[domain=0:3, blue, densely dashdotted] {(3*x)/(1)};
\addplot[domain=0:5, blue, densely dashdotted] {(3*x)/(2)};
\addplot[domain=0:7, blue, densely dashdotted] {(2*x)/(3)};
\draw[blue, densely dashdotted] (0,0)--(5.5,0);
\addplot[domain=1.5:5, red, densely dashed] {(3*x-3)/(1)};
\addplot[domain=1.5:5, red, densely dashed] {(6*x-3)/(4)};
\addplot[domain=1.5:7, red, densely dashed] {(4*x+3)/(6)};
\draw[red, densely dashed] (1.5,1.5)--(5.5,1.5);
\addplot[domain=0:7, black] {(7*x)/(9)};
\end{axis}
\end{tikzpicture}
\vspace{-0.35cm}
\captionof{figure}{}
\label{restrictedCone}
\end{minipage}
\begin{minipage}{.7\textwidth}
$n_{\alpha, \beta}$ such that $\left(\lfloor xn\rfloor + \alpha-\lambda_j, n+\beta-\beta_j\right)\in {\mathfrak C_{j}}$, for all $n\geq n_{\alpha, \beta}$. Let  $\alpha_1 = \alpha - \lambda_j$ and $\beta_1 = \beta -\beta_j$. Now
$$\lim\limits_{n\to\infty}\tfrac{\lfloor xn\rfloor + \alpha_1}{n+\beta_1} = \lim\limits_{n\to\infty}\tfrac{xn + \delta_n+\alpha_1}{n+\beta_1} = \lim\limits_{n\to\infty}\tfrac{\lfloor x(n+\beta_1)\rfloor +\delta_n+ \alpha_1-x\beta_1}{n+\beta_1} = x,$$
where $\delta_n$ is any number such that $|\delta_n|< 1$. Choose a real number $\epsilon>0$ such that $I(x-\epsilon,x+\epsilon)\subset I(\frac{d_{j}}{e_{j}}, \frac{d_{j+1}}{e_{j+1}})$. Then there exists an integer $n_{\alpha, \beta} > 0$ such that for all $n\geq n_{\alpha, \beta}$
$$\Big|\tfrac{\lfloor xn\rfloor + \alpha_1}{n+\beta_1}-x\Big|<\epsilon \;\; \text{and therefore}\;\; \tfrac{\lfloor xn\rfloor + \alpha_1}{n+\beta_1} \in
I(\tfrac{d_{j}}{e_{j}},~ \tfrac{d_{{j}+1}}{e_{{j}+1}}).$$
\end{minipage}

Now by Lemma \ref{density0} the assertion follows.
\end{proof}

The following two lemmas will be used in Section $\ref{continuity}$ to show that the density function $f_{A,\{I_n\}}$ is continuous at the points $x=d_i/e_i$ for $1<i\leq l$.

\begin{lemma}\label{i-1}
Let the hypotheses be as in Notations~\ref{chambers}. Let $\mathcal{S}\subset \R$ be a finite set and $2\leq i\leq l$. Then for all $a\in \mathcal{S}$ and all $n\gg 0$,
\begin{enumerate}
\item there exists $m_0\in e_1\N_{>0}$ such that $(a, 0)+n(\frac{d_i}{e_i}, 1) + m_0(\frac{d_1}{e_1}, 1) \in {\mathfrak R}C_{i-1}$ and
\item there exists $m_1\in \N_{>0}$ such that $(a, 0)+n(\frac{d_i}{e_i}, 1) - m_1(1,0)
\in {\mathfrak R}C_{i-1}$,
\end{enumerate}
\end{lemma}
\begin{proof}{\em Assertion $(1)$}. \quad It is enough to prove that there exist $m_0\in e_1\N_{>0}$ and $n_{m_0}\in \N$ such that for every $n\geq n_{m_0}$ and all $a\in \mathcal{S}$,
$$(a, 0)+n(\tfrac{d_i}{e_i}, 1) + m_0(\tfrac{d_1}{e_1}, 1)-(\lambda_i, \beta_i) \in {\mathfrak C}_{i-1},$$
which is equivalent to proving that
$$\frac{a+n\tfrac{d_i}{e_i}+m_0\tfrac{d_1}{e_1}-\lambda_i}{n+m_0-\beta_i} \in I[\tfrac{d_{i-1}}{e_{i-1}}, \tfrac{d_i}{e_i}].$$
But this holds 
\begin{align*}
 &\iff \frac{a+(n+m_0-\beta_i)\tfrac{d_i}{e_i}+(m_0-\beta_i)(\tfrac{d_1}{e_1}-\tfrac{d_i}{e_i})+\beta_i\tfrac{d_1}{e_1}-\lambda_i}{n+m_0-\beta_i} \in I[\tfrac{d_{i-1}}{e_{i-1}},\tfrac{d_i}{e_i}]\\
 &\iff \frac{a+(m_0-\beta_i)(\tfrac{d_1}{e_1}-\tfrac{d_i}{e_i})+\beta_i\tfrac{d_1}{e_1}  -\lambda_i}{n+m_0-\beta_i} \in I[\tfrac{d_{i-1}}{e_{i-1}}-\tfrac{d_i}{e_i},0]\\
 &\iff \frac{-a+(m_0-\beta_i)(\tfrac{d_i}{e_i}-\tfrac{d_1}{e_1})-\beta_i\tfrac{d_1}{e_1}  +\lambda_i}{n+m_0-\beta_i} \in I[0, ~~\tfrac{d_{i}}{e_{i}}-\tfrac{d_{i-1}}{e_{i-1}}].
\end{align*}

Since $d_i/e_i - d_1/e_1 >0$, we can choose $m_0\in e_1\N_{>0}$ such that
$$(m_0-\beta_i)(\tfrac{d_i}{e_i}-\tfrac{d_1}{e_1})\geq a+\beta_i\tfrac{d_1}{e_1}-\lambda_i \quad \mbox{for all}\;\; a\in \mathcal{S}.$$
Having chosen $m_0$, we can now choose $n_{m_0}$ such that
$$\frac{-a+(m_0-\beta_i)(\tfrac{d_i}{e_i}-\tfrac{d_1}{e_1})-\beta_i\tfrac{d_1}{e_1}  +\lambda_i}{n_{m_0}+m_0-\beta_i} \leq  \tfrac{d_{i}}{e_{i}}-\tfrac{d_{i-1}}{e_{i-1}}.$$
This proves the first assertion.
 
\vspace{5pt}

\noindent{\em Assertion $(2)$}.\quad We need to find $m_1\in \N_{>0}$ and $n_{m_1}$ such that for every $n\geq n_{m_1}$ and all $a\in \sS$
$$(a,0)+n(\tfrac{d_i}{e_i}, 1)-m_1(1, 0)-(\lambda_i, \beta_i)\in \mathfrak{C}_j.$$
This holds if and only if
$$\frac{-a-\beta_i\tfrac{d_i}{e_i}+m_0+\lambda_i}{n-\beta_i} \in I[0,~~\tfrac{d_i}{e_i}-\tfrac{d_{i-1}}{e_{i-1}}].$$
Now we choose $m_1$ such that $m_1 \geq a+\beta_i\tfrac{d_i}{e_i}-\lambda_i$ for all $a\in \sS$. The assertion follows by choosing $n_{m_1}$ such that
$$\frac{-a-\beta_i\tfrac{d_i}{e_i}+m_0+\lambda_i}{n_{m_1}-\beta_i} \leq  \tfrac{d_i}{e_i}-\tfrac{d_{i-1}}{e_{i-1}}.$$
\end{proof}

\begin{lemma}\label{lrr1}
Let $\{d_j/e_j\}_{j=1}^l$ be the set of elements as in Notations $\ref{chambers}$. Let $\sS\subset \R$ be a finite set. Then the following holds.
\begin{enumerate}
\item If $d_1/e_1 < d_{i}/e_{i} < d_{l+1}/e_{l+1}$, {\em i.e.}, $2\leq i\leq l$ then there exists $m_0\in e_1\N_{>0}$ such that
$$(a, 0)+ n(\tfrac{d_i}{e_i}, 1) - m_0(\tfrac{d_1}{e_1}, 1)\in {\mathfrak R}C_{i}\quad\mbox{for every}\;\; a\in \sS\;\;\mbox{and all}\;\; n\gg 0.$$
\item If $d_1/e_1\leq d_i/e_i < d_{l+1}/e_{l+1}$, {\em i.e.}, $1\leq i \leq l$ then there exists $m_1\in \N_{>0}$ such that
$$(a, 0) + n(\tfrac{d_i}{e_i}, 1) + m_1(1, 0)\in {\mathfrak R}C_{i},\quad\mbox{for every}\;\; a\in \sS\;\;\mbox{and all}\;\; n\gg 0.$$
\end{enumerate} 
\end{lemma}
\begin{proof}
From Notations $\ref{chambers}$, we recall that $(d_{l+1}, e_{l+1}) = (1, 0)$ and $d_{l+1}/e_{l+1} = \infty$.

\vspace{5pt}

\noindent{\em Assertion $(1)$}.\quad We need to find $m_0$ and $n_{m_0}$ such that
$$(a, 0)+ n(\tfrac{d_i}{e_i}, 1)-m_0(\tfrac{d_1}{e_1}, 1)-(\lambda_i, \beta_i)\in {\mathfrak C}_{i},$$
for every $n\geq n_{m_0}$ and all $a\in \sS$. This is same as showing that
$$\frac{a+m_0(\tfrac{d_i}{e_i}-\tfrac{d_1}{e_1})+\beta_i\tfrac{d_i}{e_i}-\lambda_i}{n-m_0-\beta_i}\in I[0,\tfrac{d_{i+1}}{e_{i+1}}-\tfrac{d_i}{e_i}].$$
Since $\tfrac{d_i}{e_i}-\tfrac{d_1}{e_1}>0$ we can choose $m_0\in e_1\N_{>0}$ such that
$$a_0+ m_0(\tfrac{d_i}{e_i}-\tfrac{d_1}{e_1})+\beta_i\tfrac{d_i}{e_i}-\lambda_i > 0.$$
Now the assertion follows for any choice of $n_{m_1}$ with the property that
$$\frac{a+m_0(\frac{d_i}{e_i}-\tfrac{d_1}{e_1})+\beta_i\tfrac{d_i}{e_i}-\lambda_i}{n-m_0-\beta_i} <\tfrac{d_{i+1}}{e_{i+1}}-\tfrac{d_i}{e_i}.$$

\vspace{5pt}

\noindent{\em Assertion $(2)$}.\quad To prove this, we need to find $m_1$ and $n_{m_1}$ such that
$$(a, 0) + n(\frac{d_i}{e_i}, 1) + m_1(1, 0)-(\lambda_i, \beta_i)\in{\mathfrak C}_{i},$$
for every $n\geq n_{m_1}$ and all $a\in \sS$. This is same as showing that
$$\frac{a+\beta_i\tfrac{d_i}{e_i} + m_1-\lambda_i}{n-\beta_i}\in I[0,\tfrac{d_{i+1}}{e_{i+1}}-\tfrac{d_i}{e_i}].$$
We can choose $m_1\in \N_{>0}$ such that
$$m_1 > -a-\beta_i\tfrac{d_i}{e_i}+\lambda_i,\quad\mbox{for all}\;\; a\in \sS.$$
Now any choice of $n_{m_1}$ satisfying
$$\frac{a+\beta_i\tfrac{d_i}{e_i} + m_1-\lambda_i}{n_{m_1}-\beta_i} < \tfrac{d_{i+1}}{e_{i+1}}-\tfrac{d_i}{e_i}$$
establishes the assertion.
\end{proof}

\section{The  length function  of  bigraded algebra and 
vector partition theorem}

\subsection{The length function of a polynomial bigraded algebra}

Let $k$ be a field. Let $S = k[X_1, \ldots, X_r, Y_1, \ldots, Y_s]$ be a bigraded polynomial $k$-algebra with grading $\deg X_i = (1,0)$ and $Y_j = (d_j,e_j)\in \N\times \N_{>0}$ for $i=1, \ldots, r$ and $j=1, \ldots, s$.

One can relate the $\ell_k(S_{m,n})$ with a  vector partition function as follows:
\begin{align*}\label{poly}
\ell_k(S_{m,n}) =~& \#\big\{(\lambda_1, \ldots, \lambda_{r+s})\in \N^{r+s}\mid
x^{\lambda_1}\cdots
x^{\lambda_r}(x^{d_1}y^{e_1})^{\lambda_{r+1}}\cdots 
(x^{d_s}y^{e_s})^{\lambda_{r+s}} = x^my^n\big\}\\
=~& \#\big\{(\lambda_1, \ldots, \lambda_{r+s})\in \N^{r+s}\mid
e_1\lambda_{r+1}+\cdots+e_s\lambda_{r+s} = n\quad\mbox{and}\quad
\sum_{i=1}^r\lambda_i +   \sum_{j=1}^s\lambda_{r+j}d_{j} = m\big\}.
\end{align*}
Hence 
$\ell_k(S_{m,n}) = \phi_M(m, n)$, where $\phi_M:\N^2\longto \N$ is the vector partition function corresponding to the matrix $M$ as in (\ref{matrixA}). On the other hand, we have the power series expansion
\begin{multline*}
\dfrac{1}{(1-x)^r(1-x^{d_1}y^{e_1})\cdots (1-x^{d_s}y^{e_s})}
 = (1+x+x^2+\cdots )\cdots 
{\mbox{$r$ times}}\cdots (1+x+x^2+\cdots )\cdot\\
(1+x^{d_1}y^{e_1}+(x^{d_1}y^{e_1})^2+\cdots )
\cdots (1+x^{d_s}y^{e_s}+(x^{d_s}y^{e_s})^2+\cdots ).\end{multline*}
Therefore,
\begin{equation}\label{HSforP}\sum_{(m,n)\in \N^2} \ell_k(S_{m,n})x^my^n =
\dfrac{1}{(1-x)^r(1-x^{d_1}y^{e_1})\cdots (1-x^{d_s}y^{e_s})}.
\end{equation}

\subsection{The length function of a Noetherian bigraded algebra over \texorpdfstring{$k$}{k}}\label{ss3}

Let $A = \oOplus_{n\geq 0}A_n$  be  a standard graded algebra over  a field $A_0 = k$ with
homogeneous maximal ideal $\mathfrak{m} = \oOplus_{n\geq 1}A_n$.
Therefore, we can write $A = k[x_1, \ldots, x_r]$, where $x_i$'s are degree $1$  elements.

Let $R = \mathop{\oOplus}_{(m,n)\in \N^2}R_{m,n}$ be a finitely generated bigraded
$A$-algebra generated by elements of degrees 
$\{(d_i, e_i)\in \N^2\mid 1\leq i \leq s\}$, where bidegree of $x_i$ is $(1,0)$. Then 
 ${R}$ is represented as a graded quotient of 
the bigraded polynomial ring 
$S=k[X_1,\ldots,X_r,Y_1,\ldots,Y_s]$, where 
$X_i$ maps to $x_i$ and $Y_i$ map to a degree $(d_i, e_i)$ generator of $R$. Therefore there exists \cite[Proposition $8.18$]{MS05} a bigraded minimal
free resolution of finite length  
$$0\to \mathop{\oOplus}_{j=1}^{\eta_t}
S(-a_{tj},-b_{tj})^{\beta_{tj}} \to \cdots \to \mathop{\oOplus}_{j=1}^{\eta_1}
S(-a_{1j},-b_{1j})^{\beta_{1j}} \to S \to {R}\to 0$$ 
of ${R}$ as an $S$-module. Hence the bigraded Hilbert series of $R$ is given by
$$\begin{array}{lcl}
\sum\limits_{(m,n)\in\mathbb{N}^2}\ell_k(R_{m,n}) x^my^n & = &
\sum\limits_{i=0}^t (-1)^i\Big(\sum\limits_{j=1}^{\eta_i}
\beta_{ij}\big(\sum\limits_{(m,n)\in\mathbb{N}^2}
\ell(S_{m-a_{ij},n-b_{ij}}) x^my^n\big)\Big)\\\
& = & \dfrac{p(x,y)}{(1-x)^r(1-x^{d_1}y^{e_1})\cdots (1-x^{d_s}y^{e_s})},
\end{array}$$
where $p(x,y) = \sum_{i=0}^t (-1)^i
\big(\sum_{j=1}^{\eta_i}\beta_{ij}x^{a_{ij}}y^{b_{ij}}\big)
\in\mathbb{Z}[x,y]$.
 We rewrite
$$p(x,y)  = \sum_{(\alpha, \beta)\in N(I)}c_{\alpha,
\beta}x^{\alpha}y^{\beta}\in \Z[x, y],\quad\mbox{where}\quad
 N(I) = \{(\alpha, \beta)\in \N\times \N\mid c_{\alpha, \beta}\neq 0\}.$$
This gives
\begin{equation}\label{pe}\sum\limits_{(m,n)\in\mathbb{N}^2}\ell_k(R_{m,n}) x^my^n =
\sum_{(\alpha,\beta)\in N(I)}c_{\alpha,\beta}~\ell_k(S_{m-\alpha, n-\beta})x^my^n.\end{equation}

\begin{notations}\label{ell(Rmn)} Let $R$ be a finitely generated bigraded $A$-algebra where $A$ is a standard graded ring over a field $k$.
 Let  
 $$\{(d_1, e_1), (d_2, e_2), \ldots, (d_s, e_s)\}$$ be a set of bidegrees of a set of generators of $R$ as an $A$-algebra.
We follow  Notations \ref{chambers} associated to this set.

Note that since the constant term of the polynomial $p(x,y)$ is  the length of bigraded component $R_{0,0}$ which is nonzero, we have  $(0,0)\in N(I)$.
In particular, ${\mathfrak R}C_j \subseteq  {\mathfrak C_{j}}$. If $R$ is the polynomial ring $S$ itself then each ${\mathfrak R}C_j = {\mathfrak C}_j$.
\end{notations}

\begin{rmk}\label{restcone}We can partition
$$\N\times \N = ({\mathfrak C}_0\cap \N\times \N)\cup
 ({\mathfrak C}_1\cap \N \times \N)\cup
 \cdots \cup({\mathfrak C}_l\cap \N\times\N),$$
 where ${\mathfrak C}_0$ is the cone generated by the vectors $(0, 1)$ and $(d_1, e_1)$. If  $(m,n)\in \N\times \N$ then
 $$(m,n)\in {\mathfrak C}_j\iff 
\tfrac{d_j}{e_j}\leq \tfrac{m}{n}\leq \tfrac{d_{j+1}}{e_{j+1}}.$$
 Therefore, if $1\leq j<l$ then
 $$(m,n)\in {\mathfrak R}C_j\iff 
 \tfrac{d_j}{e_j}+\tfrac{1}{n}\big(\lambda_j-\tfrac{\beta_jd_j}{e_j}\big)\leq \tfrac{m}{n}\leq \tfrac{d_{j+1}}{e_{j+1}}
 -\tfrac{1}{n}\big(\tfrac{\beta_jd_{j+1}}{e_{j+1}}-\lambda_j\big)$$
and $(m,n)\in {\mathfrak R}C_l\iff  \tfrac{d_l}{e_l}+\tfrac{1}{n}\big(\lambda_l-\frac{\beta_ld_l}{e_l}\big)\leq \tfrac{m}{n}.$
 \end{rmk}
 
\begin{lemma}\label{qp}Following Notations $\ref{ell(Rmn)}$, we have that for each cone $\mathfrak{C}_{j_0}$ there is a quasi-polynomial $QP_{j_0}(X, Y)$ which is periodic of period $h$ in both variables $X$ and $Y$ such that $$(m,n)\in {\mathfrak{R}}C_{j_0} \implies \ell(R_{m,n})= QP_{j_0}(m,n).$$
\end{lemma}

\begin{proof}
By (\ref{pe}),
$$\sum\limits_{(m,n)\in\mathbb{N}^2}\ell_k(R_{m,n}) x^my^n =
\sum_{(\alpha,\beta)\in N(I)}c_{\alpha,\beta}~
\ell_k(S_{m-\alpha, n-\beta})x^my^n.$$
If $(m,n)\in {\mathfrak{R}}C_{j_0}$ then $(m-\alpha, n-\beta) \in {\mathfrak C}_{j_0}$ for each $(\alpha,\beta)\in N(I)$. Therefore,
\begin{equation}\label{lengthformula}
\ell_k(R_{m,n}) = \sum_{(\alpha,\beta)\in N(I)} c_{\alpha,\beta}\cdot\phi_M(m-\alpha,n-\beta),
\end{equation}
where $\phi_M:\N^2\longto \N$ is the vector partition function corresponding to the matrix $M$ as in (\ref{matrixA}). In particular, by Theorem \ref{corrector}, for the set of chambers $\{\mathfrak C_j\}_{1\leq j\leq l}$ there is  a set of corresponding polynomials
$$\{P_{\mathfrak C_j},~ Q^{\mathfrak C_j}_{\sigma}\in 
Q[X, Y] \mid 1\leq j \leq l, \sigma \in \Delta({\mathfrak C_j})\}$$ 
such that $(m,n)\in {\mathfrak{R}}C_{j_0}$ implies
\begin{multline}\label{lformula}
\ell_k(R_{m,n})=
\sum_{(\alpha,\beta)\in N(I)} c_{\alpha,\beta}
\Big[P_{\mathfrak C_{j_0}}\left(m- \alpha, n-\beta\right)\\
  +\sum_{\big\{\sigma\in\Delta({\mathfrak C_{j_0}})\;\big\vert\;  [\left(m - \alpha,
n-\beta\right)]_{\sigma}\neq 0\big\}}\Omega_{\sigma}[\left(m -
\alpha, n-\beta\right)]_{\sigma}\cdot 
Q^{\mathfrak C_{j_0}}_{\sigma}\left(m - \alpha, n-\beta\right)
\Big].\end{multline}
Now for $h$  as in Lemma~\ref{density1}, each function $\Omega_{\sigma}$  
is periodic of period $h$ both as a function of $X$ and $Y$. 
Therefore we have the quasi-polynomial as stated in the lemma.
\end{proof}

Following corollary covers Theorem $1.1$ of \cite{HT03}.

\begin{cor}\label{HT} Let $R$ be  the bigraded ring as in Notations $\ref{ell(Rmn)}$ with the  condition that $e_j=1$ for $j=1, \ldots, s$.
Then there exists a polynomial $P(X, Y)\in \Q[X,Y]$ such that
$$m\geq d_{l}n + n_0 \implies \ell_k(R_{m,n}) = P(m,n),$$
$\mbox{where} ~n_0 =  \lceil{\lambda_l} - d_l \beta_l\rceil$ with $(\lambda_l, \beta_l) \in \Q^2_{\geq 0}$, asssociated to ${\mathfrak{R}{C_l}}= {\mathfrak{C_l}}+(\lambda_l, \beta_l)$.
\end{cor}

\begin{proof}
By Remark~\ref{restcone}, $m\geq d_{l}n + n_0 \iff (m,n)\in {\mathfrak R} C_l$.
 
Consider the expression for $\ell_k(R_{m,n})$ as in (\ref{lformula}). Now by Remark~\ref{trivial},
$[(m, n)]_{\sigma} =  0$ for all $(m,n)\in {\mathfrak C}_l$.
This gives for $m\geq d_ln+n_0$
$$\ell_k(R_{m,n})  = \sum_{(\alpha,\beta)\in N(I)} c_{\alpha,\beta}
P_{\mathfrak C_l}\left(m- \alpha, n-\beta\right),$$ 
where $\sum_{(\alpha,\beta)\in N(I)} c_{\alpha,\beta}
P_{\mathfrak C_l}\left(X- \alpha, Y-\beta\right)$ is a 
polynomial in $\Q[X, Y]$
\end{proof}

 \begin{notations}\label{noethfamily}
 From now on, we will consider the bigraded $A$-algebra $R =  \mathop{\oOplus}_{(m, n)\in \N^2} (I_n)_mt^n$, where  $\{I_n\}_{n\in \N}$ is a Noetherian filtration of homogeneous ideals in $A$.

We say that $\{I_n\}$ is a \emph{filtration} if for all $n,m\in\N$,
$$I_0 = A,\quad  I_{n+1}\subseteq I_n,\quad
I_n\cdot I_m \subseteq I_{n+m}.$$

Such a filtration is Noetherian if in addition there is an integer $m\geq 1$ such that $I_n\cdot I_m = I_{n+m}$ for all $n\geq m$, see \cite[Proposition $2.1$]{Sch88}. In this case, $I_{nm} = (I_m)^n$ for every $n\geq 1$.

The bigrading on  ${R} = \mathop{\oOplus}_{(m, n)\in \N^2} (I_n)_mt^n$ is given as follows: $(I_n)_m = R_{m,n}$ is the set of homogeneous degree  $m$ elements of $A$ which belong to $I_n$, and the  homogeneous elements of degree $m$ in $A$  have bidegree $(m,0)$ in the bigraded algebra $R$.
\end{notations}

Following are a few examples of Noetherian filtration of ideals which do not arise as a power of an ideal and are generated by bidegrees like $(d_i, e_i)$ where $e_i$ need not be $1$.

\begin{ex}\label{Nfiltration}
\hspace{2em}
\begin{enumerate}
\item[(1)] If $(A, {\bf m})$ is an analytically unramified local ring and $I\subseteq A$ an ideal, then the integral closure filtration $\{\overline {I^n}\}_{n\in\N}$ (\cite[Theorem 1.4]{Ree61}) is a Noetherian filtration. If $A$ has positive characteristic then  $I^n \subseteq {(I^n)}^* \subseteq \overline {I^n}$ by definition, where
$(I^n)^*$ denotes the tight closure of the ideal $I^n$ in $A$. In particular, $\{{(I^n)}^*\}_{n\in\N}$ is a Noetherian filtration. Moreover, if $(A,\bf m)$ is an excellent domain and the analytic spread of $I$ is not maximal then
$\{I^n\colon {\bf m}^{\infty}\}_{n\in\N}$ is a Noetherian filtration, see \cite{CHS10}.
\item[(2)]
Let  $A$ be a polynomial ring over a field $k$. Let $I$ and $J$ be two  monomial ideals in $A$. Then, by \cite[Theorem $3.2$]{HHT07},
$\{(I^n:J^{\infty})\}_{n\geq 0}$ is a Noetherian filtration. Particularly, in this case the ordinary symbolic Rees algebra $\mathop{\oOplus}_{n
\in \N}I^{(n)}t^n$ and the saturated Rees algebra $\mathop{\oOplus}_{n \in \N}(I^n:{\bf m}^{\infty})t^n$ are finitely generated bigraded algebra  over $A$.
\end{enumerate}
\end{ex}

\begin{rmk}\label{I_n} Some of the forthcoming results will hold true for the
following set of bigraded algebras
  $R =\mathop{\oOplus}_{(m,n)\in \N^2}(I_n)_mt^n$, where $\{I_n\}_n$ is a
 filtration of  ideals as above (not necessarily Noetherian) but  
$R$  does have a Hilbert-series of the form as in (\ref{pe}). 
For this kind of bigraded algebra we would assume 
the additional  condition that 
 $\mbox{depth A} >0$. The following two examples discuss two algebras of this kind.
 \end{rmk}

\begin{ex}
Let $k$ be an algebraically closed field of characteristic zero. Let $A=\mathop{\oOplus}_{m\geq 0}A_m$ be a two-dimensional standard graded Noetherian domain over $k$. Let $I\subseteq A$ be a height one homogeneous ideal. In this situation, it is possible that the saturated Rees algebra $\mathop{\oOplus}_{n\geq 0}\widetilde{I^n}t^n$ is non-Noetherian, see \cite{Ree58}. However, in view of \cite[comment after Theorem $4.1$]{Cut03}, it follows that there is a rational expression of the bivariate Hilbert series associated to the saturated Rees algebra $\mathop{\oOplus}_{(m,n)\in\mathbb{N}^2}\big(\widetilde{I^n}\big)_mt^n$.
\end{ex}

\begin{ex}[Nagata]\label{nagata}
Suppose that $A=\mathbb{C}[X,Y,Z]$ be a standard graded polynomial ring in three variables over $\mathbb{C}$. Let $s\geq 4$ be an integer. Let $p_1,\ldots,p_{s^2}$ be $s^2$ general points in $\mathbb{P}^2_{\mathbb{C}}$. Consider the homogeneous ideal $I = \bigcap_{i=1}^{s^2} I(p_i)$, where $I(p_i)\subseteq A$ is the ideal generated by all homogeneous polynomials vanishing at $p_i$. It was shown by Nagata \cite{Nag59} that the saturated Rees algebra $\mathop{\oOplus}_{n\geq 0}\widetilde{I^n}t^n$ is non-Noetherian. However the Segre-Harbourne-Gimigliano-Hirschowitz (SHGH) conjecture which
is known to be true in our setup (see \cite{Mir06}, \cite{Eva07} and \cite{Roe06}) gives  that $\ell_k\big((\widetilde{I^n})_m\big) = \max\{0,\binom{m+2}{2} - s^2\binom{n+1}{2}\}$ for all integers $m\geq 0$ and $n\geq 0$. In particular, the bigraded Hilbert series of the saturated Rees algebra $\mathop{\oOplus}_{(m,n)\in\mathbb{N}^2}(\widetilde{I^n})_mt^n$ has a rational expression.
\end{ex}

\begin{thm}\label{t1} Let $R= \mathop{\oOplus}_{(m,n)\in \N^2}R_{m,n} =  \mathop{\oOplus}_{(m,n)\in \N^2}(I_n)_mt^n$ be a bigraded $A$-algebra, which is either Noetherian or as given in Remark~\ref{I_n}. Let $\dim A = d\geq 1$ and $\mathrm{depth}~A >0$. Then for a given cone  ${\mathfrak C_{j_0}}$ there exist a homogeneous polynomial $P_{j_0}(X, Y)$ and a quasi-polynomial $Q_{j_0}(X,Y)$ in $\Q[X, Y]$ with $\deg~Q_{j_0}(X,Y) < \deg~P_{j_0}(X, Y)$ such that
 \begin{equation}\label{et1}
 (m,n)\in {\mathfrak{R}C_{j_0}} \cap \N^2 \implies  
\ell_k\big((I_n)_m\big) = P_{j_0}(m,n) + Q_{j_0}(m, n).\end{equation}
Further $\deg P_{j_0}(X,Y) \leq d-1$.
\end{thm}

\begin{center}
\captionsetup{type=figure}
\begin{tikzpicture}[scale=0.8, font=\tiny]
\begin{axis}[axis equal image,
xmin=-0.75, xmax=9,
ymin=-0.75, ymax=5.5,
axis lines=middle,
axis line style={-latex},
ytick=\empty,
xtick=\empty,
xlabel=$m$,ylabel=$n$,
]
\draw[fill, olive3] (0.75,1.1)--(4.65,5)--(2.9,5.35)--(0.75,1.1);
\draw[fill, olive2] (1,0.7)--(5.25,4.9)--(7,3.75)--(1,0.7);
\draw[fill, olive1] (1,0.27)--(7.2,3.3)--(8.25,0.27)--(1,0.27);
\node at (3,5.25) {$\frac{1}{d_1}$};
\node at (7.25,3.3) {$\frac{1}{d_s}$};
\node[blue] at (1,3.5) {\small{$0$}};
\node at (0.45,-0.3) {(0,0)};
\node[blue] at (6,1.7) {$P_{\mathcal{R}(I)}(m,n)$};
\node[blue] at (6,1.2){(HT03)};
\addplot[domain=0:3, blue, densely dashdotted] {(x)/(0.5)};
\addplot[domain=0:5, blue, densely dashdotted] {(x)/(1)};
\addplot[domain=0:7, blue, densely dashdotted] {(x)/(2)};
\end{axis}
\end{tikzpicture}
\vspace{-0.35cm}
\captionof{figure}{}
\label{cone_polynomials}
\end{center}

\begin{proof} Here $\ell_k(R_{m,n}) = \ell_k\big((I_n)_m\big)$. By Lemma~\ref{qp}, for a given cone ${\mathfrak C}_{j_0}$, there is a quasi-polynomial $QP_{j_0}(X, Y)$
which is periodic of period $h$ in both variables $X$ and $Y$. Thus there exists a set
$$\sS_{j_0} =  \{P_{i,j}[X, Y]\in \mathbb{Q}[X,Y] \mid 0\leq i, j<h\}$$ with $P_{i+h, j}(X, Y)  = P_{i, j}(X, Y)$
and $P_{i, j}(X, Y) = P_{i, j+h}(X, Y)$, such that given $(m, n)\in \mathfrak{R}C_{j_0}\cap \N^2$,
\[\ell_k\big((I_n)_{m}\big) = P_{i', j'}(m, n) \;\; \mbox{for some}\;\; P_{i',j'}(X,Y)\in \sS_{j_0}.\]
Then
\begin{enumerate}
\item[(i)] $\ell_k\big((I_n)_{m+i}\big) =
P_{i'+i, j'}(m+i, n)$ if $(m+i, n)\in {\mathfrak R}C_{j_0}$, and
\item[(ii)] $\ell_k\big((I_{n+j})_{m}\big) =
P_{i', j'+j}(m+i, n)$ if $(m, n+j)\in {\mathfrak R}C_{j_0}$.
\end{enumerate}

Let $\deg P_{i,j}(X, Y) = r_{i,j}$ and set $r_0 = \max\{r_{i,j}\}_{i,j}$. We fix  an element $x\in I(\frac{d_{j_0}}{e_{j_0}}, \frac{d_{{j_0}+1}}{e_{{j_0}+1}})$ and the finite set
$$\sS=\{(\alpha, \beta)\in \N^2\mid 0\leq \alpha\leq (r_0+2)h,\quad 0\leq \beta\leq h\}.$$
Then by Lemma~\ref{density2}, there is  $n_x>0$ such that for all $n\geq n_x$
\begin{equation}\label{n_x}
 (\lfloor xn\rfloor + \alpha, n+\beta) \in {\mathfrak R}C_{j_0}\cap 
\N^2\quad\mbox{for each}\quad
(\alpha, \beta)\in \sS.\end{equation}
Moreover, for given $\alpha_0, \beta_0\in \{1, \ldots, 2h\}$ each of the  following sets
 $$\left\{\tfrac{\lfloor xn\rfloor + \alpha_0+ih}{n+\beta_0}\mid 
0\leq i\leq 
r_0\right\}\quad \mbox{ and}\quad 
\left\{\tfrac{n+\beta_0}{\lfloor xn\rfloor + \alpha_0+ih}\mid 
0\leq i\leq r_0\right\}$$ 
consists of $r_0+1$ distinct elements.

\vspace{5pt}

\noindent{\bf Step} $1$.\quad Here we pick a polynomial, say, $P_{0,0}(X, Y)$ in $\sS_{j_0}$ and prove the following.

\vspace{5pt}

\noindent{\bf Claim}
\begin{enumerate}
\item[$(i)$] The top degree terms of  $P_{0,0}(X, Y)$ = the top degree terms of $P_{j,0}(X, Y)$, for all $1\leq j\leq h$.
\item[$(ii)$] The top degree terms of $P_{0,0}(X, Y)$ = the top degree terms of $P_{0,j}(X, Y)$, for all $1\leq j\leq h$.
\end{enumerate}

Since the choice of $P_{0,0}$ is random, the same proof goes through by replacing $P_{0,0}$ by any $P_{i_0, j_0}$. This will prove the first assertion of the theorem.

Now for $P_{0,0}(X,Y)$ and given $(\lfloor xn\rfloor, n)\in {\mathfrak R}C_{j_0}$ there exist
 $\alpha_0, \beta_0 \in \{1, \ldots, h\}$ such that 
 $$\ell_k\big((I_{n+\beta_0})_{\lfloor xn\rfloor + \alpha_0}\big) =
P_{0, 0}(\lfloor xn\rfloor + \alpha_0, n+\beta_0).$$
Further with these choices of $\alpha_0, \beta_0$ we have
$\ell_k\big((I_{n+\beta_0+j})_{\lfloor xn\rfloor + \alpha_0+i}\big)  =
P_{i, j}(\lfloor xn\rfloor + \alpha_0+i, n+\beta_0+j)$. Therefore,
\begin{equation}\label{e2_poly}
\ell_k\big((I_{n+\beta_0+j})_{\lfloor xn\rfloor +
\alpha_0+i+i_1h}\big) =
P_{i, j}(\lfloor xn\rfloor + \alpha_0+i+i_1h, n+\beta_0+j),\end{equation}
for all $0\leq i_1\leq r_0$.

\vspace{5pt}

\noindent{\emph {Proof of the Claim $(i)$}}.\quad Let ${\tilde r} = \max\{r_{0,0}, r_{j,0}\}$. Then we can express
$$P_{0,0}(X,Y) = \sum_{l=0}^{\tilde r}\lambda_l^0X^lY^{{\tilde r}-l}+R_{0,0}(X,Y)\quad\mbox{and}\quad   
P_{j,0}(X,Y) = \sum_{l=0}^{\tilde r}\lambda_l^jX^lY^{{\tilde r}-l}+R_{j,0}(X,Y),$$
where degrees of $R_{0,0}(X,Y)$ and $R_{j,0}(X,Y)$ are 
strictly less than ${\tilde r}$.

Therefore we can write 
$$\tfrac{P_{0,0}(X,Y)}{Y^{\tilde r}} =
\sum_{l=0}^{\tilde r}\lambda^0_l(\tfrac{X}{Y})^l + \sum_{i=1}^{\tilde r}
\tfrac{R^0_{{\tilde r}-i}({X}/{Y})}{Y^i},\quad\quad
\tfrac{P_{j,0}(X,Y)}{Y^{\tilde r}} =
\sum_{l=0}^{\tilde r}\lambda^j_l(\tfrac{X}{Y})^l + \sum_{i=1}^{\tilde r}
\tfrac{R^j_{{\tilde r}-i}({X}/{Y})}{Y^i},$$
where $R^0_i({X}/{Y})$ and $R^j_i({X}/{Y})$ are polynomials in $\Q[{X}/{Y}]$ of degrees $\leq i$. Note that if $r_{j,0} < {\tilde r}$ then all $\lambda^j_l = 0$ for all $l$. Similar statement holds if $r_{0,0} < {\tilde r}$.

If the claim does not hold then
$\sum_{l=0}^{\tilde r}(\lambda^0_l-\lambda^j_l)(\tfrac{X}{Y})^l$ is a nonzero polynomial. Hence there is $i_1\in \{0, 1, \ldots, {\tilde r}\}$  (here ${\tilde r}\leq r_0$), and by \eqref{e2_poly}
$$(m_1, n_1) = (\lfloor xn\rfloor + \alpha_0+i_1h, n+\beta_0)~\mbox{such that}~ \sum_{l=0}^{{\tilde r}}(\lambda^0_l-\lambda^j_l)
\left(\tfrac{m_1}{n_1}\right)^l \neq 0.$$
If $\sum_{l=0}^{\tilde r}(\lambda^0_l-\lambda^j_l)(\tfrac{X}{Y})^l$ is a constant then we simply choose $(m_1, n_1) = (\lfloor xn\rfloor + \alpha_0+i_1h, n+\beta_0)$. Since $(m_1, n_1)\in {\mathfrak R}C_{j_0}$ and $(m_1h^i, n_1h^i) \in \mathfrak{C}_{j_0}$, by Lemma~\ref{density0}
$$\{(m_1h^i+m_1, n_1h^i+n_1)\mid i\geq 1\} \subset {\mathfrak R}C_{j_0}\cap \N^2.$$
For the similar reason $(m_1+j, n_1)$ and $(m_1+h, n_1)\in 
\mathfrak{R}C_{j_0}$  implies
\begin{multline}
\{(m_1h^i+m_1+j, n_1h^i+n_1)\mid i\geq 1\}
\cup \{(m_1h^i+m_1+h, n_1h^i+n_1)\mid i\geq 1\}
\subseteq {\mathfrak R}C_{j_0}\cap \N^2.\end{multline}
Since  $\ell_k\big((I_{n_1})_{m_1}\big) = P_{0,0}(m_1, n_1)$, we have
$$\begin{array}{lll}
\ell_k\big((I_{n_1h^i+n_1})_{m_1h^i+m_1}\big) & = & P_{0,0}(m_1h^i+m_1, n_1h^i+n_1),\\\
\ell_k\big((I_{n_1h^i+n_1})_{m_1h^i+m_1+j}\big) & = & P_{j,0}(m_1h^i+m_1+j, n_1h^i+n_1),\\\
\ell_k\big((I_{n_1h^i+n_1})_{m_1h^i+m_1+h}\big) & = & P_{0,0}(m_1h^i+m_1+h, n_1h^i+n_1).
\end{array}$$
On the other hand,  since $\mbox{depth}~A\geq 1$, the bigraded ring $R$ has a nonzerodivisor of bidegree $(1,0)$. This gives the inequalities
$$\ell_k\big((I_{n_1h^i+n_1})_{m_1h^i +m_1}\big) \leq \ell_k\big((I_{n_1h^i+n_1})_{m_1h^i +m_1+j}\big)
\leq \ell_k\big((I_{n_1h^i+n_1})_{m_1h^i+m_1+h}\big)$$
 for all $i\geq 1$. Therefore,
\begin{align*}
\lim_{i\to\infty}\dfrac{P_{0,0}(m_1h^i+m_1,n_1h^i+n_1)}{(n_1h^i+
n_1)^{\tilde r}} &\leq
\lim_{i\to\infty}\dfrac{P_{j,0}(m_1h^i+m_1+j,n_1h^i+n_1)}{(n_1h^i+
n_1)^{\tilde r}}\\
&\leq \lim\limits_{n\to\infty}
\dfrac{P_{0,0}(m_1h^i+m_1+h,n_1h^i+n_1)}{(n_1h^i+n_1)^{\tilde r}},
\end{align*}
which implies 
$$\sum_{l=0}^{{\tilde r}}
\lambda^0_l\left(\tfrac{m_1}{n_1}\right)^l \leq 
\sum_{l=0}^{{\tilde r}}
\lambda^j_l\left(\tfrac{m_1}{n_1}\right)^l 
\leq \sum_{l=0}^{{\tilde r}}
\lambda^0_l\left(\tfrac{m_1}{n_1}\right)^l.$$
This contradicts the choice of $(m_1, n_1)$ and thus proves Claim $(i)$.

\vspace{5pt}

\noindent{\emph{Proof of the Claim $(ii)$}}.\quad Let ${\bar r} = \max\{r_{0,0}, r_{0,j}\}$. Then we can express
$$\tfrac{P_{0,0}(X,Y)}{Y^{\bar r}} =
\sum_{l=0}^{\bar r}\nu^0_l(\tfrac{X}{Y})^l + \sum_{i=1}^{\bar r}
\tfrac{S^0_{{\bar r}-i}(\tfrac{X}{Y})}{Y^i},\quad\quad
\frac{P_{0,j}(X,Y)}{Y^{\bar r}} =
\sum_{l=0}^{\bar r}\nu^j_l(\tfrac{X}{Y})^l + \sum_{i=1}^{\bar r}
\tfrac{S^j_{{\bar r}-i}(\tfrac{X}{Y})}{Y^i},$$
where $S^0_i(\tfrac{X}{Y})$ and $S^j_i(\tfrac{X}{Y})$ are polynomials in $\Q[\tfrac{X}{Y}]$ of degree $\leq i$.
Note that if $r_{0,j} < {\bar r}$ then all $\nu^j_i = 0$. Same goes for $r_{0,0}$.

If the claim does not hold then $\sum_{l=0}^{\bar r}(\nu^0_l-\nu^j_l)(\tfrac{X}{Y})^l$ is a nonzero polynomial. Thus there is $0\leq i_2\leq {\bar r}$ such that
$$\sum_{l=0}^{\bar r}(\nu^0_l-\nu^j_l)\left(\tfrac{n_2}{m_2}\right)^l \neq 
0\quad\mbox{for}\quad 
(m_2, n_2) = (\lfloor xn\rfloor + \alpha_0+i_2h, n+\beta_0).$$
Since $(m_2, n_2)$,  $(m_2, n_2+j)$ and $(m_2, n_2+h) \in \mathfrak{R}C_{j_0}$  we have from Lemma~\ref{density0} that
\begin{multline}
\{(m_2h^i+m_2, n_2h^i+n_2)\mid i\geq 1\} \cup 
\{(m_2h^j+m_2, n_2h^i+n_2+j)\mid i\geq 1\}\\\
\cup \{(m_2h^i+m_2, n_2h^i+n_2+h)\mid i\geq 1\}
\subseteq {\mathfrak R}C_{j_0}\cap \N^2.\end{multline}
By (\ref{e2_poly}), $\ell_k\big((I_{n_2})_{m_2}\big) = P_{0,0}(m_2, n_2)$. Therefore
$$\ell_k\big((I_{n_2h^i+n_2+j})_{m_2h^i+m_2}\big) = P_{0,j}(m_2h^i+m_2, n_2h^i+n_2+j).$$
As $I_{n_2h^i+n_2+h}\subset I_{n_2h^i+n_2+j}\subset I_{n_2h^i+n_2}$ so we have
$$(I_{n_2h^i+n_2+h})_{m_2h^i+m_2} \subset (I_{n_2h^i+n_2+j})_{m_2h^i+m_2}
\subset (I_{n_2h^i+n_2})_{m_2h^i+m_2},$$
which  gives 
$$\begin{array}{lcl}
P_{0,0}(m_2h^i+m_2, n_2h^i+n_2+h) & \leq & P_{0,j}(m_2h^i+m_2, n_2h^i+n_2+j)\\\
& \leq & P_{0,0}(m_2h^i+m_2, n_2h^i+n_2).\end{array}$$
In particular,
\begin{align*}
\sum_{l=0}^{\bar r}\nu^0_l\left(\tfrac{n_2}{m_2}\right)^l 
&=
\lim_{i\to \infty} \frac{P_{0,0}(m_2h^i+m_2, n_2h^i+n_2+h)}{(m_2h^i)^{\bar r}}\\
&=\lim_{i\to \infty} \frac{P_{0,j}(m_2h^i+m_2, n_2h^i+n_2+j)}{(m_2h^i)^{\bar r}} \\
&= \sum_{l=0}^{\bar r}\nu^j_l\left(\tfrac{n_2}{m_2}\right)^l,
\end{align*}
which contradicts the choice of $(m_2, n_2)$. This proves Claim $(ii)$ and hence the assertion of Step $1$.

\vspace{5pt}

\noindent{\bf Step} $2$.\quad It now remains to show that the degree of the polynomial ${P}_{0,0}(X, Y)$ is bounded above by $d-1$. Let $r_0=\deg P_{0, 0}(X,Y)$. Then
$$\tfrac{P_{0, 0}(X, Y)}{Y^{r_0}} = 
\sum_{l=0}^{r_{0}}\lambda_l(\tfrac{X}{Y})^l+ 
\tfrac{Q_{{r_0}-1}(\tfrac{X}{Y})}{Y}
+\cdots + \frac{Q_1(\tfrac{X}{Y})}{Y^{{r_0}-1}} + \tfrac{Q_0(
\tfrac{X}{Y})}{Y^{r_0}},$$
where $Q_i(\tfrac{X}{Y})$ is a polynomial in $\Q[\frac{X}{Y}]$ of degree $\leq i$ and 
$\sum_{l=0}^{r_0}\lambda_l(\frac{X}{Y})^l$
is a nonzero polynomial of degree $\leq r_0$. Hence there exists $i_0\in \{0, 1, \ldots, r_0\}$ such that
$$\sum_{l=0}^{r_{i,j}}\lambda_l^{i,j}\left({m_0}/{n_0}\right)^l \neq 
0\quad\mbox{for}\quad (m_0, n_0) = (\lfloor xn\rfloor + 
\alpha_0+i+i_0h, n+\beta_0+j).$$
Again by Lemma~\ref{density0},
$$ (m_0, n_0)\in \mathfrak{R}C_{j_0}\cap \N^2 \implies \{(m_0((2h)^k+1), n_0((2h)^k+1)\mid k\geq 1\}\subset \mathfrak{R}C_{j_0}\cap \N^2.$$
Therefore,
$\ell_k\big((I_{n_0((2h)^k+1)})_{m_0((2h)^k+1)}\big) = P_{0,0}(m_0((2h)^k+1),n_0((2h)^k+1))$,
whereas  
$$ \ell_k\big((I_{n_0((2h)^k+1)})_{m_0((2h)^k+1)}\big) \leq \ell_k(A_{m_0((2h)^k+1)}) \leq
O((m_0((2h)^k+1))^{d-1}).$$ 
Now if $r_{0} > d-1$ then in particular $r_{0}\geq 1$. Therefore,
$$\lim\limits_{k\to\infty}
\dfrac{P_{0,0}(m_0((2h)^k+1),n_0((2h)^k+1))}{(n_0(2h)^k)^{r_0}} = 
\sum\limits_{l=0}^{r_0}\lambda_l\left(\dfrac{m_0}{n_0}
\right)^l\neq 0.$$
On the other hand 
$$0 \leq \lim\limits_{k\to\infty}
\dfrac{P_{0,0}(m_0((2h)^k+1),n_0((2h)^k+1))}{(n_0(2h)^k)^{r_0}} 
\leq \lim\limits_{k\to\infty}
\dfrac{\ell_k(A_{m_0((2h)^k+1)})}{(n_0(2h)^k)^{r_0}}=0,$$
which contradicts the choice of  $(m_0,n_0)$.
\end{proof}

\begin{notations}\label{H^0(I)} Let $R = 
 \mathop{\oOplus}_{(m,n)\in \N^2}R_{m,n} = \mathop{\oOplus}_{(m,n)\in \N^2}{(I_n)}_mt^n$ be a Noetherian bigraded $A$-algebra. We define
 $$H^0_{\{I_n\}}(A) =  \{x\in A\mid I_nx = 0~~\mbox{for some}~~n\geq 0\}.$$
 Set $d^{\sF} = \dim A/H^0_{\{I_n\}}(A)$, where $\sF$ denotes the filtration $\{I_n\}_{n \in \N}$.
\end{notations}

\begin{thm}\label{t2} Following  Notation $\ref{H^0(I)}$, let $d^\sF \geq 1$. Then there exists an integer $n_1$ such that for each cone  ${\mathfrak C_{j_0}}$ similar conclusion as in Theorem $\ref{t1}$ holds provided $n\geq n_1$ and $d$ is replaced by $d^\sF$.
\end{thm}

\begin{proof}Since $\{I_n\}$ is a Noetherian filtration, there exists an integer $c$ such that $I_{n+c} = I_nI_c$ for all $n\geq c$. In particular, $I_{cn} = (I_c)^n$. It can be verified that $H^0_{I_c}(A) = H^0_{\{I_n\}}(A)$. So there exists $c_0 \geq 0$ such that $I_c^{c_0}\cdot H^0_{I_c}(A) = 0$.

Further we can show that there exists $n_1 \geq 0$ such that $I_n\cap H^0_{I_c}(A) = 0$, for all $n\geq n_1$. By the Artin-Rees lemma there exists $c_2 \geq 0$ such that for all $n\geq c_2$, we have $I^n_{c}\cap H^0_{I_c}(A) = I_c^{n-c_2}(I_{c}^{c_2}\cap H^0_{I_c}(A))$. Therefore, $I^n_{c}\cap H^0_{I_c}(A) = 0$ for every $n-c_2 \geq c_0$. Hence $I_m\cap H^0_{I_c}(A) = 0$ for $m\geq n_1= c(c+c_2)$. In particular,  the canonical bigraded map of rings
\begin{equation}\label{h^0}
\mathop{\oOplus}_{n\geq 0}I_nt^n
\longto
\mathop{\oOplus}_{n\geq 0}\frac{I_n+
H^0_{\{I_n\}}(A)}{H^0_{\{I_n\}}(A)}t^n
= \mathop{\oOplus}_{n\geq 0}\frac{I_n}{I_n\cap  H^0_{\{I_n\}}(A)}t^n\end{equation}
is an isomorphism for the graded components of degree $n\geq n_1$.

Consider the natural quotient map $\varphi\colon A \to A/H^0_{\{I_n\}}(A)$. Since $d^\sF\geq 1$, we have $\mathrm{grade}~\varphi (I_n)\geq 1$ for all $n\gg 0$ and hence
$\mbox{depth}~(A/H^0_{\{I_n\}}(A))\geq 1$. Now the result follows from Theorem~\ref{t1}.
\end{proof}

\begin{rmk}\label{7.4}The length function $\ell_k\big((I^n)_m\big)$ need not be a polynomial function (but can be only a quasi-polynomial) for $(m,n)\in \mathfrak{R}C_j$, if $j < l$. For example, $I=(X^3Y, XZ^2, Y^4Z)\subseteq k[X,Y,Z]$.
\end{rmk}

\section{Continuity}\label{continuity}

\subsection{Density function for a filtration}

\begin{notations}\label{den}
Suppose that $A=\oOplus_{n\geq 0}A_n$ is a standard graded ring over a field $A_0=k$ and of dimension $d$. Let $\sF = \{I_n\}_{n\in\N}$ be a Noetherian filtration of homogeneous ideals in $A$ with the associated bigraded Rees algebra $R = \oOplus_{n\in \N}I_nt^n$. We adopt Notations \ref{ell(Rmn)}. We further assume that
${d^{\sF}} := \dim \big(A/H^0_{\{I_n\}}(A)\big) \geq 1$.
 
If in addition we have $\dim A/P = \dim A$ for every minimal prime $P$ of $A$ and $I_1$ is not nilpotent then $d = \dim A = \dim \big(A/H^0_{\{I_n\}}(A)\big) = d^{\sF}$.
\end{notations}
 
\begin{defn}\label{ds}
Let $(A, \{I_n\})$ be as in Notations \ref{den}. We define the {\em density function} for the given pair $(A, \{I_n\})$ as the function
$$f_{A,\{I_n\}} \colon \R_{\geq 0}\longto \R_{\geq 0}\quad\mbox{given by}\quad x\longmapsto \limsup_{n\to \infty}\frac{\ell_k\big((I_n)_{\lfloor xn\rfloor}\big)}{n^{d^{\sF}-1}/d^{\sF}!},$$
where ${d^{\sF}} = \dim \big(A/H^0_{\{I_n\}}(A)\big)$.

Further, let $\{f_n\colon \R_{\geq 0}\to \R_{\geq 0}\}_{n\in \N}$ denote the sequence of functions given by $f_n(x) = \frac{\ell_k\big((I_n)_{\lfloor xn\rfloor}\big)}{n^{d^{\sF}-1}/d^{\sF}!}$. Hence $f_{A,\{I_n\}}(x) = \limsup_{n\to \infty}f_n(x)$.
\end{defn}

\begin{lemma}\label{l1cont}
Consider the set $X = \R_{\geq 0}\setminus \{d_1/e_1, d_2/e_2, \ldots, d_l/e_l\}$. Then the sequence $\{f_n\}_{n\in\N}$ converges pointwise to the density function
$f_{A,\{I_n\}}$ on $X$, and $\left.f_{A,\{I_n\}}\right|_X$ is a continuous function. Moreover, the following holds.
\begin{enumerate}
 \item The function $f_{A,\{I_n\}}$ can be described as $$f_{A,\{I_n\}}(x) = \begin{cases}
              0 & \text{if}\;\;x\in I[0, \frac{d_{1}}{e_{1}}),\\
              {\bf p}_j(x) & \text{if}\;\;x\in I(\frac{d_j}{e_j}, \frac{d_{j+1}}{e_{j+1}})\;\;\text{where}\;\;1\leq j<l,\\
              {\bf p}_l(x) & \text{if}\;\;x\in I(\frac{d_l}{e_l},\infty),
             \end{cases}$$
where each ${\bf p}_j(x)$ is polynomial in $\Q[x]$ of degree $\leq d^{\sF}-1$ and ${\bf p}_l(x)$ is a nonzero polynomial of degree $d^\sF-1$.
\item If $R = \oOplus_{n\in\N}I_nt^n$ has a nonzerodivisor of bidegree $(d_i, e_i)$ for some $1\leq i \leq l$ then $f_{A,\{I_n\}}$ is a nowhere vanishing function on the interval $I(\frac{d_i}{e_i}, \infty)$.
\end{enumerate}
\end{lemma}
\begin{proof}
Consider the natural quotient map $\varphi\colon A \to A/H^0_{\{I_n\}}(A)$. Then by \eqref{h^0}, for $x\in \R_{x>0}$ there exists an integer $n_{x_1}$ such that $\ell_k\big((I_n)_{\lfloor xn\rfloor}\big) = \ell_k\big((\varphi(I_n))_{\lfloor xn\rfloor}\big)$ for all integers $n\geq n_{x_1}$. Therefore, for any statement regarding the density function $f_{A,\{I_n\}}$ we can always replace $A$, $\{I_n\}$ and $d$ by $A/H^0_{\{I_n\}}(A)$, $\{\varphi(I_n)\}$ and $d^\sF$ respectively. In particular, we can assume that $R$ has a nonzerodivisor of bidegree $(d_i, e_i)$ for some $1\leq i\leq l$.

Let $x\in I(\tfrac{d_j}{e_j},\tfrac{d_{j+1}}{e_{j+1}})$. By Lemma \ref{density2} there exists an integer $n_x>0$ such that $\left(\lfloor xn\rfloor, n\right) \in {\mathfrak R}C_j$ for all $n\geq n_x$. Therefore by Theorem \ref{t1} and Theorem \ref{t2}, there is a homogeneous polynomial $P_j(X,Y)$ of degree $r_j\leq d-1$ and a quasi-polynomial $Q_j(X,Y)$ of degree $<r_j$ such that for all $n\gg 0$,
$$\ell_k\big((I_n)_{\lfloor xn\rfloor}\big) = P_j({\lfloor xn\rfloor}, n) + Q_j({\lfloor xn\rfloor},n).$$
Hence the sequence $\left\{\frac{\ell_k\big((I_n)_{\lfloor xn\rfloor}\big)}{n^{d-1}}\right\}_{n\in\N}$ converges for all $x\in  I(\frac{d_j}{e_j}, \frac{d_{j+1}}{e_{j+1}})$ and
$$f_{A, \{I_n\}}(x) = \lim_{n\to \infty}\frac{\ell_k\big((I_n)_{\lfloor xn\rfloor}\big)}{n^{d-1}/d!} = \lim_{n\to\infty}\dfrac{P_j({\lfloor xn\rfloor}, n)}{n^{d-1}/d!} =:{\bf p}_j(x).$$

Now we first prove  assertion $(2)$ and then deduce that ${\bf p}_l(x)$ is a polynomial of degree $d-1$. Supppose that $R$ has a nonzerodivisor, say $h$, of bidegree $(d_i, e_i)$. Let $HP_A(x) = a_0x^{d-1}+O(x^{d-2})$ denote the Hilbert polynomial of $A$. Then for all $n\gg 0$, we have $\ell_k(A_n) = a_0n^{d-1}+O(n^{d-2})$. By choice, one has $d_1/e_1< \cdots < d_l/e_l$. Let ${\tilde e} = e_i\cdots e_l$ and for $i\leq j\leq l$, let ${\tilde e_j} = {\tilde e}/{e_j}$. If $x\in I(\frac{d_i}{e_i}, \infty)$ then $d_i{\tilde e_i} < x{\tilde e}$. Now $h \in I_{e_i}$ implies that $h^{n{\tilde e_i}}\in I_{{\tilde e}n}$ and we have $\ell_k\big((I_{{\tilde e}n})_{\lfloor x{\tilde e}n\rfloor}\big) \geq \ell_k\big((h^{n{\tilde e_i}})_{\lfloor x{\tilde e}n\rfloor-n{\tilde e_i}d_i}\big)$ for all $n\gg 0$. Therefore,
$$f_{A, \{I_n\}}(x) = \limsup_{n\to \infty}\frac{\ell_k\big((I_{{\tilde e}n})_{\lfloor x{\tilde e}n\rfloor}\big)}{({\tilde e}n)^{d-1}/d!} \geq \frac{a_0(x-d_i/e_i)^{d-1}}{d!}> 0.$$
This proves assertion $(2)$.

Since  $R$  has a nonzerodivisor of bidegree $(d_{i}, e_{i})$ for some $1\leq i\leq l$, we have
$${\bf p}_l(x) \geq \frac{a_0(x-\tfrac{d_i}{e_i})^{d-1}}{d!}\quad\mbox{for all}\;\; x\in I(\tfrac{d_l}{e_l}, \infty),$$
which implies that ${\bf p}_l(x)$ is a polynomial of degree $d-1$.
\end{proof}

\begin{thm}\label{t3}
Let $(A, \{I_n\}_n)$ be as in Notations \ref{den}. Further assume that $R$ has a nonzerodivisor of bidegree $(d_1, e_1)$. Then the density function
$$f_{A,\{I_n\}} \colon \R_{\geq 0}\longto \R_{\geq 0}\quad\mbox{is given by}\quad x\longmapsto \lim_{n\to \infty}\frac{\ell_k\big((I_n)_{\lfloor xn\rfloor}\big)}{n^{d-1}/d!},\quad\mbox{for all}\quad x\neq d_1/e_1$$
and is a continuous function on the interval $I(\frac{d_1}{e_1}, \infty)$. More precisely,
$$f_{A,\{I_n\}}(x) = \begin{cases}
              0 & \text{if}\;\;x\in I[0, \frac{d_{1}}{e_{1}}),\\
              {\bf p}_1(x) & \text{if}\;\;x\in I(\frac{d_1}{e_1},\frac{d_2}{e_2}],\\
              {\bf p}_j(x) & \text{if}\;\;x\in I[\frac{d_j}{e_j}, \frac{d_{j+1}}{e_{j+1}}]\;\;\text{where}\;\;2\leq j<l,\\
              {\bf p}_j(x) & \text{if}\;\;x\in I[\frac{d_l}{e_l},\infty),
             \end{cases}$$
where each ${\bf p}_j(x)$ is a nonzero polynomial in $\Q[x]$ of degree $\leq d-1$ and ${\bf p}_l(x)$ is a nonzero polynomial of degree $d-1$.
\end{thm}
\begin{proof}
By our hypotheses $\oOplus_{n\geq 0}I_nt^n$ is generated in degrees $\tfrac{d_1}{e_1}<\cdots <\tfrac{d_l}{e_l}$. So we have ${(I_{n})}_{m}=0$ for all $\tfrac{m}{n}<\tfrac{d_1}{e_1}$. Consequently, $f_{A,\{I_n\}}(x)=0$ for all $x<\tfrac{d_1}{e_1}$.

The existence of a nonzerodivisor of bidegree $(d_1, e_1)$ implies that $\mbox{grade}(I_n)>0$ for all $n\gg 0$ which further implies that $H^0_{\{I_n\}}(A) = 0$. In particular, $A= A/H^0_{\{I_n\}}(A)$ and $d= d^\sF$.

In veiw of Lemma \ref{l1cont} we only need to show that $\lim_{n\to \infty}f_n(d_i/e_i) = {\bf p}_{i-1}(d_i/e_i) = {\bf p}_{i}(d_i/e_i)$, for $2\leq i\leq l$.

Let $x=d_i/e_i$, where $2\leq i \leq l$. Consider the set $\sS = \{0, -1/e_i, -2/e_i, \ldots, (1-e_i)/e_i\}$. Let $z_0$ and $z_1$ be nonzero elements in $R$ of degree $(1, 0)$ and $(d_1, e_1)$ respectively. Then for a given integer $n$ there exists $b_n\in \sS$ such that $xn + b_n = \lfloor xn \rfloor$. By Lemma \ref{i-1} there exist integers $m_0 = {\tilde m_0}e_1 \in e_1\N_{>0}$ and $m_1\in \N_{>0}$ such that for all $n\gg 0$,
\begin{align*}
 (\lfloor xn\rfloor, n)+ ({\tilde m_0}d_1,{\tilde m_0}e_1) &= (b_n, 0)+ n(x, 1) + {\tilde m_0}(d_1, e_1) \in {\mathfrak{R}}C_{i-1},\\
 (\lfloor xn\rfloor, n)- (m_1, 0) &= (b_n, 0)+ n(x, 1) -  m_1(1, 0) \in {\mathfrak{R}}C_{i-1}.
\end{align*}
On the other hand, the elements $z_1^{\tilde m_0}$ and $z_0^{m_1}$ are nonzerodivisors of bidegrees $({\tilde m_0}d_1,{\tilde m_0}e_1)$ and $(m_1, 0)$ respectively,
which gives injective maps
\begin{equation}\label{conte1}
(I_n)_{\lfloor xn\rfloor-m_1}\longby{z_0^{m_1}}
(I_n)_{\lfloor xn\rfloor}\longby{z_1^{\widetilde m_0}}
(I_{n+{\tilde m_0}e_1})_{\lfloor xn\rfloor+{\tilde m_0}d_1}.
\end{equation}
Now following the notations of Theorem \ref{t1} we have
\begin{multline*}
 {\bf p}_{i-1}(x) = \lim_{n\to \infty}\frac{\ell_k\big((I_n)_{\lfloor xn\rfloor-m_1}\big)}{n^{d-1}/d!} \leq \liminf_{n\to \infty}\frac{\ell_k\big((I_n)_{\lfloor xn\rfloor}\big)}{n^{d-1}/d!} \leq \limsup_{n\to \infty}\frac{\ell_k\big((I_n)_{\lfloor xn\rfloor}\big)}{n^{d-1}/d!}\\ \leq \lim_{n\to \infty}\frac{\ell_k\big((I_{n+{\tilde m_0}e_1})_{\lfloor xn\rfloor+{\tilde m_0}d_1}\big)}{n^{d-1}/d!} = {\bf p}_{i-1}(x).
\end{multline*}
Therefore,
$$ f_{A, \{I_n\}}(x) = \lim_{n\to \infty}\frac{\ell_k\big((I_n)_{\lfloor xn\rfloor}\big)}{n^{d-1}/d!} = {\bf p}_{i-1}(x).$$

Similarly, by Lemma \ref{lrr1} there exist integers $m_0 = {\tilde m_0}e_1 \in e_1\N_{>0}$ and $m_1\in \N_{>0}$ such that for all $n\gg 0$,
\begin{align*}
 (\lfloor xn\rfloor, n) - ({\tilde m_0}d_1,{\tilde m_0}e_1) &= (b_n, 0)+ n(x, 1) - {\tilde m_0}(d_1, e_1) \in {\mathfrak{R}}C_{i},\\
 (\lfloor xn\rfloor, n)  + (m_1, 0) &= (b_n, 0)+ n(x, 1) + m_1(1, 0) \in {\mathfrak{R}}C_{i}.
\end{align*}
On the other hand there are injective maps
\begin{equation}\label{conte2}
(I_{n-{\tilde m_0}e_1})_{\lfloor xn\rfloor-{\tilde m_0}d_1}\longby{z_1^{\widetilde m_0}}
(I_n)_{\lfloor xn\rfloor}\longby{z_0^{m_1}}
(I_{n})_{\lfloor xn\rfloor+m_1}.
\end{equation}
Again by Theorem \ref{t1} we know that
$$\lim_{n\to \infty}\frac{\ell_k\big((I_n)_{\lfloor xn\rfloor+m_1}\big)}{n^{d-1}/d!} = {\bf p}_i(x) = \lim_{n\to \infty}\frac{\ell_k\big((I_{n-{\tilde m_0}e_1})_{\lfloor xn\rfloor - {\tilde m_0}d_1}\big)}{n^{d-1}/d!}.$$
Therefore,
$$ f_{A, \{I_n\}}(x) = \lim_{n\to \infty}\frac{\ell_k\big((I_n)_{\lfloor xn\rfloor}\big)}{n^{d-1}/d!} = {\bf p}_i(x).$$
This proves the continuity of $f_{A, \{I_n\}}$ at $x=d_i/e_i$.
\end{proof}

\begin{rmk}\label{r6}
Suppose that $R=\oOplus_{(m,n)\in \N^2}R_{m,n} = \oOplus_{(m,n)\in \N^2}(I_n)_mt^n$ is a bigraded $A$-algebra as in Remark \ref{I_n} with $\mbox{depth}~A>0$. Then Lemma \ref{l1cont} and Theorem \ref{t3} holds  for such a (not necessarily Noetherian) filtration of ideals.
\end{rmk}

\begin{rmk}\label{r4}
The following are true.
\begin{enumerate}
\item Let $R$ be as in Notations \ref{den}. If $(d_i, e_i)$ is the degree of one of the generators of $R$ such that the generators of degrees $(d_j,e_j)$ satisfying the inequality $d_j/e_j< d_i/e_i$ are nilpotent elements, then $f_{A, \{I_n\}}(x) = 0$ for all $x\in I[0, d_i/e_i)$.
\item We can further generalise the statement of Theorem \ref{t3} as follows: {\em Let $A$ and $\{I_n\}_{n\in\N}$ be as in Notations \ref{den} such that $\dim A = \dim A/P$ for every minimal prime ideal $P$ of $A$, i.e., $A$ is equidimensional. Let $d_j/e_j$ be the smallest element in the set
$$\{d_i/e_i\mid \mbox{there is a non-nilpotent generator of $R$ of degree}~ (d_i, e_i)\}.$$
Then  $f_{A, \{I_n\}}(x) = 0$ for 
$x\in I[0, d_j/e_j)$ and 
$f_{A, \{I_n\}}$ 
is a nonzero monotonically increasing function on the interval $I(d_j/e_j, \infty)$.}
\end{enumerate}
\end{rmk}

We recall some basic facts about diagonal subalgebras.

\begin{defn}\label{dg}
Let $A$ be a standard graded ring over a field $k$ and  $I$ be a nonzero homogeneous ideal in $A$ generated by homogeneous elements of degrees $d_1< d_2<\cdots < d_l$. Let $R=\oOplus_{(m,n)\in\N^2}{(I^n)}_mt^n$ be the bigraded Rees algebra of $I$.  Given a tuple $(p,q)\in\N^2$, we consider the $(p,q)$-diagonal subalgebra $$R_{\Delta_{I(p,q)}} = \oOplus_{n\geq 0}R_{(pn,qn)} = \oOplus_{n\geq 0} {(I^{qn})}_{pn}t^{qn}.$$
\end{defn}

Note that $R_{\Delta_{I(p,q)}}$ is a finitely generated graded $k$-algebra. To see this, represent ${R}$ as a graded quotient of a bigraded polynomial ring
$S=k[X_1,\ldots,X_r,Y_1,\ldots,Y_s]$ (see \ref{ss3}), where $\deg X_i = (1,0)$ and $\deg Y_j = (d_j,1)$. Then $R_{\Delta_{I(p,q)}}$ is a graded quotient of $S_{\Delta_{(p,q)}}$. From \cite[Section 1.A]{CHTV97} we have that $S_{\Delta_{(p,q)}}$ is an affine semigroup ring. Thus $S_{\Delta_{(p,q)}}$ is a finitely generated $k$-algebra and so is $R_{\Delta_{I(p,q)}}$. By the dimension formula \cite[Theorem 23]{Mat70} it follows that $\dim R_{\Delta_{I(p,q)}} \leq d$. So we have a well-defined notion of Hilbert-Samuel multiplicity $$e\big(R_{\Delta_{I(p,q)}}\big) = \lim_{n\to\infty}\dfrac{\ell_k\left({(I^{qn})_{pn}}\right)}{n^{d-1}/(d-1)!}$$ which could be zero. Since $f_{A,\{I^n\}}(p/q) = \tfrac{d}{q^{d-1}}\cdot e\big(R_{\Delta_{I(p,q)}}\big)$, now using Theorem \ref{t3} we can define the density function $f_{A,\{I^n\}}$ by replacing `$\limsup$' by `$\lim$' everywhere.

Following is a special but more relevant case  of Theorem~\ref{t3}.

\begin{cor}\label{limit}For the pair $(A,I)$ as in Definition~\ref{dg}, if the ideal $I$ has a homogeneous generator of degree $d_1$ which is a nonzerodivisor then
 the density function $f_{A,\{I^n\}}:\R_{\geq 0}\longto \R_{\geq 0}$
 is continuous  and  piecewise polynomial on the set $\R\setminus \{d_1\}$, satisfying the properties as in Theorem \ref{t3}.
 \end{cor}

\begin{rmk}
Let $(A,I)$ be the pair as in Definition \ref{dg}. Let $\mathbf{p}_j(x)$ be the polynomial as in Lemma \ref{l1cont}, where we take $(d_i,e_i) = (d_i,1)$ and $d^{\mathcal{F}}=d$. Further let $\mathbf{c}_i=e\big(R_{\Delta_{I(d_i,1)}}\big)$. Then the multiplicities of the diagonal subalgebras are determined by $\mathbf{p}_1(x), \ldots, \mathbf{p}_l(x)$ and $\mathbf{c}_1,\ldots,\mathbf{c}_l$. Given $(p,q)\in \N^2$,
\begin{align*}
 p/q = d_i &\implies e(R_{\Delta_{I(p,q)}}) = q^{d-1}\mathbf{c}_i,\\
 p/q \in I(d_i,d_{i+1}) &\implies e(R_{\Delta_{I(p,q)}}) = \tfrac{q^{d-1}}{d}\cdot \mathbf{p}_i(p/q).
\end{align*}
We recall that such a formula for $e(R_{\Delta_{I(p,q)}})$, where $p/q\in I(d_l,\infty)$, is proved in \cite[Corollary 4.4]{HT03}. The coefficients of the polynomial $\mathbf{p}_l(x)$ are known as the mixed multiplicities of $A[It]$.
\end{rmk}

\begin{ex}\label{r5}
The density function $f_{A,\{I^n\}}$ may not be continuous at $x= d_1$ even in a polynomial ring of any dimension. For example, let $A = k[X_1, \ldots, X_d]$ be a polynomial ring in $d$ variables ($d\geq 2$) over a field  $k$ and consider the ideal $I = (X_1^2, X_2^2, \ldots, X_d^2)$. Then $\ell_k\big((I^n)_{2n}\big) = \binom{n+d-1}{d-1}$ and $\ell_k\big((I^n)_{2n+1}\big) = d\cdot\binom{n+d-1}{d-1}$. Therefore
$$f_{A, \{I^n\}}(x) = \begin{cases}
                       0\quad &\mbox{if}\quad x\in I(0,2)\\
                       d\quad &\mbox{if}\quad x=2\\
                       dx^{d-1} \quad &\mbox{if}\quad x\in I(2,\infty)
                      \end{cases}.$$
\end{ex}

\begin{ex}\label{exr6}
In the following example $R = A[It]$ does not have a nonzerodivisor of bidegree $(1,1)$ or $(2,1)$ but does have a nonzerodivisor of bidegree $(3,1)$, namely $x^3+y^3+z^3$ which belongs to $I$, and the density function is indeed not continuous at $x=1$ and $x=2$, and is continuous on the interval $I(3, \infty)$. Let $A = k[X, Y, Z]/(XY, YZ, ZX)$ over a field $k$, and $I = (x, y^2, z^3)$. Then $$f_{A, \{I^n\}}(x) = \begin{cases}
0 \quad &\mbox{if}\quad x<1\\
1 \quad &\mbox{if}\quad 1\leq x<2\\
2 \quad &\mbox{if}\quad 2\leq x<3\\
3 \quad &\mbox{if}\quad x \geq 3
\end{cases}.
$$
\end{ex}

\begin{ex}
In this example, we show that the polynomials $\mathbf{p}_j$ occurring in Theorem \ref{t3} can have degree $<d-1$ for some $j<l$. Let $A=k[X,Y,Z]$ be a polynomial ring over a field $k$, and $I=(X,Y^2,Z^3)$. Notice that $I$ is generated in distinct degrees $\{1,2,3\}$. Then
$$f_{A, \{I^n\}}(x) = \begin{cases}
0 \quad &\mbox{if}\quad x\leq 1\\
9x^2-18x+9 &\mbox{if}\quad 1\leq x\leq 2\\
18x-27 \quad &\mbox{if}\quad 2\leq x\leq 3\\
3x^2 \quad &\mbox{if}\quad x \geq 3
\end{cases}.
$$
\end{ex}

\section{Density functions for saturated powers of an ideal}\label{sec5}

Throughout this section we have the following.

\begin{notations}\label{nsat}
Let the pair $(A,I)$ denote a standard graded domain $A=\oOplus_{n\geq 0}A_n$ over a field $A_0 = k$ and a homogeneous ideal $I$ in $A$. Also assume that $d = \dim A \geq 2$. Let ${\bf m}$ denote the homogeneous maximal ideal of $A$. Let $\widetilde {I^n} = I^n \colon_A {\bf m}^{\infty} = \bigcup_{i\geq 1} (I^n:_A{\bf m}^i)$ denote the saturation of the ideal $I^n$.

We define the density function for the pair $(A, \{\widetilde {I^n}\})$ as
$$f_{A, \{\widetilde {I^n}\}} \colon \R_{\geq 0}\longto \R_{\geq 0}\quad\mbox{given by}\quad f_{A, \{\widetilde {I^n}\}}(x) = \limsup_{n\to \infty}\dfrac{\ell_k\big((\widetilde {I^n})_{\lfloor xn\rfloor}\big)}{n^{d-1}/d!}.$$
Further we consider a sequence of functions
$$\left\{g_n \colon \R_{\geq 0}\longto \R_{\geq 0}\right\}_{n\in \N}\quad\mbox{given by}\quad g_n(x) =  \dfrac{\ell_k\big((\widetilde {I^n})_{\lfloor xn\rfloor}\big)}{n^{d-1}/d!}.$$
Hence $f_{A, \{\widetilde {I^n}\}} = \limsup\limits_{n\to \infty}g_n(x)$.
\end{notations}

However, it is known that the filtration $\{\widetilde {I^n}\}_{n\in\N}$ is not necessarily Noetherian (see Examples \ref{eg_Nagata} and \ref{eg_Cutkosky}). So we can not apply Lemma \ref{l1cont} to conclude any nice properties of the density function $f_{A, \{\widetilde {I^n}\}}$. Hence a different approach is needed.

\begin{notations}\label{densat}
Without any loss of generality we can assume that $0 < \mbox{height}~I <d$. This is because if $\mbox{height}~I = 0$ then $\widetilde {I^n} = 0$ for all $n$, and if $\mbox{height}~I = d$ then $\widetilde {I^n} = A$ for all $n$.

Let $V = \mbox{Proj}~A$ and let $\sI$ be the sheaf of ideals, associated to the ideal $I$, on $V$. Also let $\sO_V(1)$ be the natural very ample line bundle on $V$. We have $\dim V = d-1$.

Let $X = \mathbf{Proj} \left(\oplus_{n\geq 0}\mathcal{I}^n\right)$ be the blow up of $V$ along $\sI$ with exceptional divisor $E$. Let $H$ be the pull back of a hyperplane section on $V$. Then the natural map $\pi\colon X \longto V$ is birational and $$\sO_X(H)= \pi^*\sO_V(1)\quad\mbox{and}\quad \sO_X(-E) = \sI\sO_X.$$

There exists an integer $m_1\geq 0$ such that $\big(H^1_{\bf m}(A)\big)_m = 0$ for all $m\geq m_1$. We then have an identification $$H^0\left(V,\mathcal{I}^n\otimes \mathcal{O}_V(m)\right) = {\big(\widetilde{I^n}\big)}_m$$ for all $n\geq 0$ and $m\geq m_1$. Moreover, there exists an integer $n_0\geq 0$ such that $\pi_*\sO_X(-nE) = \sI^n$ for all $n\geq n_0$ (see \cite[Lemma $5.4.24$]{Laz04a}). This, combined with the projection formula for the map $\pi$, gives us that for all $m\geq m_1$ and $n\geq n_0$,
\begin{equation}\label{rsat}
h^0(X,\sO_X(mH-nE)) = h^0(V, \sO_V(m)\tensor \pi_*\sO_X(-nE)) = h^0(V, \sI^n(m)) = \ell_k\big({\big(\widetilde {I^n}\big)}_m\big).
\end{equation}
First note that $H^0(X,\sO_X(mH-nE))=0$ if $\deg (mH-nE) = m\deg H - n\deg E<0$, where the \emph{degree of a divisor} is as defined in \cite[Chapter II, Exercise 6.2]{Har77}. Here $\deg H>0, \deg E>0$ since $H, E$ are effective divisors. Moreover, it can shown that there exists a $\gamma\in \mathbb{Q}_{>0}$ such that ${\big(\widetilde {I^n}\big)}_m = 0$ for all $m<\gamma n$, see \cite[Lemma 6.3]{DDRV24}. So by increasing $n_0$ if required, i.e., $n_0>\max\{\tfrac{m_1\deg H}{\deg E},\tfrac{m_1}{\gamma}\}$, we get $h^0(X,\sO_X(mH-nE)) = 0 = (\widetilde {I^n})_m$ for all $0\leq m < m_1$ and $n\geq n_0$. So we may further assume that the equalities in \eqref{rsat} hold for all $m\geq 0$ and $n\geq n_0$.
\end{notations}

\vspace{5pt}

\subsection{Density function in characteristic \texorpdfstring{$p>0$}{}}
In the case $A$ is of characteristic $p>0$, using elementary methods we show (Proposition \ref{p1nd}) that the subsequence $\{g_q\mid q=p^n,~~n\in \N\}$ of the sequence $\{g_n\mid n\in \N\}$ converges and the limit $\lim_{q\to\infty}g_q(x)$ is a continuous function. The inspiration comes from the ideas developed by the third author in \cite{Tri18}.

Throughout this subsection we assume that the underlying field $k$ is perfect. We denote by $\sM = \sO_X(H)$ and $\sL = \sO_X(-E)$. Let $F:X\longto X$ denote the Frobenius morphism.

\begin{lemma}\label{l1sat}
There exists an integer $m_0\geq 0$ and an exact  of $\sO_X$-modules
\begin{equation}\label{e2}
0 \longto \mathop{\oOplus}_{i=1}^{p^{d-1}}\sO_X \longto F_*(\sO_X)\tensor \sM^{m_0} \longto Q\longto 0,
\end{equation}
where $\dim\;\mathrm{Supp}(Q) <\dim~X = d-1$.
\end{lemma}

\begin{proof}
Since $V = \mbox{Proj}\;A$, any nonzero homogeneous element $f$ in ${\bf m}$ gives an affine open set $V_f := D_{+}(f)$ of $V$. Further we can choose such an $f$ in $I$.
Then we have the  isomophism of varieties $\pi\mid_{X_f} \colon X_f = \pi^{-1}(V_f)\longto V_f$.

If $\eta$ denotes the generic point of $V$ then we have an isomophism
$$\phi \colon \mathop{\oOplus}_{i=1}^{p^{d-1}}(\sO_V)_{\eta} \longrightarrow (F_*\sO_V)_{\eta},$$
where $\phi$ is an $\sO_V$-linear map. Hence there exists a nonzero homogeneous element $g$ in $A$ such that $D_+(g)\subset V$ is an affine open set and the map $\phi$ extends to
a map $$\phi_1 \colon \mathop{\oOplus}_{i=1}^{p^{d-1}}\left.(\sO_V)\right|_{D_+(g)} \longrightarrow \left.(F_*\sO_V)\right|_{D_+(g)}.$$
Now we have an affine open set $U= D_+(fg)$ of $V$ such that $\pi^{-1}(U) \cong U$ via the map $\pi$. The map $\phi_1$ restricts to the map
$\phi_2 \colon \oplus_{i=1}^{p^{d-1}} \left.(\sO_V)\right|_{U}\longrightarrow \left.(F_*\sO_V)\right|_{U}$ which is same as
$$ \phi_2 \colon \mathop{\oOplus}_{i=1}^{p^{d-1}}\left.(\sO_X)\right|_{\pi^{-1}(U)}\longrightarrow \left.(F_*\sO_X)\right|_{\pi^{-1}(U)}.$$

Let $\sF = \shom_{\sO_X}(\oplus_{i=1}^{p^{d-1}}\sO_X, F_*(\sO_X))$. Then $\phi_2\in H^0\left(\pi^{-1}(U),\sF\right)$. We note that $\left.\sO_V(1)\right|_{U} = (fg)\sO_{U}$ and therefore $\left.\sM\right|_{\pi^{-1}U} = (fg) \sO_{\pi^{-1}U}$. This gives $\pi^{-1}(U) = \{x\in X\mid fg\not\in {\bf m}_x\sM_x\}$. Hence, by Lemma $5.14$ of \cite[Chapter II]{Har77} there exists an integer $m_2\geq 0$ such that $(fg)^{m_2}\phi_2 \in H^0(X, \sF\tensor \sM^{m_0})$, where $m_0 = (\deg fg)\cdot m_2$. Therefore, we obtain an injective map of $\sO_X$-modules
$$(fg)^{m_2}\phi_2 \colon \mathop{\oOplus}_{i=1}^{p^{d-1}}\sO_X\longrightarrow F_*\sO_X\tensor \sM^{m_0}$$
which is an isomorphism at the generic point of $X$.
\end{proof}

\begin{propose}\label{p1nd}
Given any compact interval $I[0, e]$, where $e\in \R_{>0}$,
\begin{enumerate}
\item the sequence $\{g_q\}_{q=p^n}$ is eventually increasing and converges uniformly.
\item The density function $f_{A, \{\widetilde {I^q}\}}\colon \R_{\geq 0}\longto \R_{\geq 0}$ given by $x\mapsto \lim_{q\to \infty}g_q(x)$
is a well-defined continuous function.
\end{enumerate}
\end{propose}
\begin{proof} \emph{Assertion} $(1)$.\quad First we show that there exists a constant $C_e>0$ such that for all $q\geq n_0$ (recall $n_0$ from Notations \ref{densat}) and $x\in I(0,e]$,
\begin{equation}\label{charp}
 0\leq g_{qp}(x)-g_q(x)\leq \dfrac{C_e}{q}.
\end{equation}
From (\ref{rsat}), it follows that for all $q\geq n_0$, we have
$$g_q(x) = \dfrac{h^0\left(X, \sM^{\lfloor xq\rfloor}\tensor \sL^q\right)}{q^{d-1}/d!}.$$
Therefore we need to show that there is a constant $C_e>0$ such that for all $x\in I(0,e]$ and $q=p^n$,
\begin{equation}\label{**}
0\leq A_q(x):= \dfrac{h^0\left(X, \sM^{\lfloor xqp\rfloor} \tensor \sL^{qp}\right)}{(qp)^{d-1}} - \dfrac{h^0\left(X, \sM^{\lfloor xq\rfloor}\tensor\sL^q\right)}{q^{d-1}} \leq \dfrac{C_e}{q}.
\end{equation}
Let $m_0$ be the integer as in Lemma \ref{l1sat}. Tensoring \eqref{e2} with the line bundle $\sM^m\tensor \sL^q$ gives the short exact of $\sO_X$-modules
$$0 \longto \mathop{\oOplus}_{i=1}^{p^{d-1}}\sM^m\tensor \sL^q \longto F_*\sO_X\tensor \sM^{m+m_0}\tensor \sL^q \longto Q\tensor \sM^m\tensor \sL^q \longto 0.$$
Now the projection formula gives
$$F_*\sO_X\tensor \sM^{m+m_0}\tensor \sL^q = F_*(\sM^{mp+m_0p}\tensor \sL^{qp})$$
and 
$$H^0(X, F_*(\sM^{mp+m_0p} \tensor \sL^{qp})) = H^0(X, \sM^{mp+m_0p} \tensor \sL^{qp}).$$
Hence we have the induced long exact sequence
\begin{equation}\label{e3}
0 \longto \mathop{\oOplus}_{i=1}^{p^{d-1}} H^0(X, \sM^m\tensor \sL^q) \longto H^0(X, \sM^{mp+m_0p} \tensor \sL^{qp})\longto H^0(X,Q\tensor \sM^m\tensor \sL^q)
\end{equation}
and therefore
$$0 \leq \dfrac{h^0(X, \sM^{mp+m_0p} \tensor \sL^{qp})}{(qp)^{d-1}} - \dfrac{h^0(X, \sM^m\tensor \sL^q)}{q^{d-1}} \leq \frac{h^0(X, Q\tensor \sM^m\tensor \sL^q)}{(qp)^{d-1}}.$$

\noindent{\bf Claim}.
\begin{enumerate}
\item[$(i)$] If ${\mathcal{G}}$ is a coherent sheaf of $\sO_X$-modules with $\dim\; \mathrm{Supp}(\mathcal{G})<d-1$, then there is a constant $C_e(\mathcal{G})>0$ such that for all integers $m\leq eq$, we have
$$h^0(X, \mathcal{G}\tensor \sM^m \tensor \sL^q) \leq C_e(\mathcal{G})q^{d-2}.$$
\item[$(ii)$] For a given integer $l_0>0$, there is a constant $C_{l_0}>0$ such that for all integers $0\leq m_1\leq m_2 \leq l_0$ and $m\leq eq$, we have
$$\left\vert h^0(X, \sM^{mp+m_1} \tensor \sL^{qp}) - h^0(X, \sM^{mp+m_2} \tensor \sL^{qp})\right\vert \leq C_{l_0}(qp)^{d-2}.$$
\end{enumerate}

\vspace{5pt}

\noindent{\emph{Proof of Claim $(i)$}}.\quad By \cite[Example $1.2.33$]{Laz04a}, there exists a constant $C_e(\mathcal{G})>0$ such that $$h^0(X, \mathcal{G}\tensor (\sM^e \tensor \sL)^q)\leq C_e(\mathcal{G})q^{d-2}\quad \text{for all}\;\; q=p^n.$$
We also note that, for any integer $l\geq 0$, we have a nonzero section $a\in H^0(X, \sM^{l})$, which is possible as $H^0(V, \sO_V(l))\neq \{0\}$ and isomorphism of $X$ to $V$ on a dense open set implies $\pi^*a \neq 0$. Now for $m\leq eq$ we can choose a nonzero section $a\in H^0(X, \sM^{eq-m})$ which induces an injective map
$$0\longto \mathcal{G}\tensor \sM^m \tensor \sL^{q} \longto \mathcal{G}\tensor (\sM^e \tensor \sL)^{q}$$
of $\sO_X$-modules and hence the desired inequality
$$h^0(X, \mathcal{G}\tensor \sM^m \tensor \sL^{q}) \leq h^0(X, \mathcal{G}\tensor (\sM^e \tensor \sL)^q)\leq C_e(\mathcal{G})q^{d-2}.$$

\vspace{5pt}

\noindent{\emph{Proof of Claim $(ii)$}}.\quad For a given integer $l_0\geq 1$ we choose a nonzero section $a\in H^0(X, \sM^{l_0})$ which induces an injective map
$$0\longto \sO_X \longto  \sM^{l_0} \longto Q_{l_0}\longto 0$$
of $\sO_X$-modules, where $\dim\;\mathrm{Supp}(Q_{l_0}) < \dim X = d-1$. Hence
$$\left\vert h^0(X, \sM^{mp}\tensor \sL^{qp})-h^0(X, \sM^{mp+l_0}\tensor\sL^{qp})\right\vert \leq h^0(X, Q_{l_0}\tensor\sM^{mp}\tensor\sL^{qp}).$$
By Claim $(i)$ there exists a constant $C_{l_0}>0$ such that
$$h^0(X, Q_{l_0}\tensor\sM^{mp}\tensor\sL^{qp})\leq C_{l_0}({qp})^{d-2}.$$
Now the assertion follows as the existence of a nonzero global section of $\sM^{m_2-m_1}$ gives
$$h^0(X, \sM^{mp+m_1} \tensor \sL^{qp})\leq h^0(X, \sM^{mp+m_2} \tensor \sL^{qp}).$$

\vspace{5pt}

Now, for a given $q=p^n$ and $x\in I(0, e]$, to prove the bounds on $A_q(x)$ as given in (\ref{**}), let $m = \lfloor xq\rfloor$ then $\lfloor xqp\rfloor = mp+m_1$ for some integer $0\leq m_1 <p$. By \eqref{e3} the nonnegativity of $A_q(x)$ follows. On the other hand, \eqref{e3} and the above claim implies that
\begin{multline*}
 A_q(x) \leq \left\vert\dfrac{h^0\left(X, \sM^{mp + m_1} \tensor \sL^{qp}\right)}{(qp)^{d-1}} - \dfrac{h^0\left(X, \sM^{mp + m_0p}\tensor\sL^{qp}\right)}{(qp)^{d-1}}\right\vert\\ + \left\vert \dfrac{h^0\left(X, \sM^{mp + m_0p} \tensor \sL^{qp}\right)}{(qp)^{d-1}} - \dfrac{h^0\left(X, \sM^{m}\tensor\sL^q\right)}{q^{d-1}}\right\vert \leq \dfrac{C_{(m_0+1)p}(qp)^{d-2}}{(qp)^{d-1}} + \dfrac{C_e(Q)q^{d-2}}{(qp)^{d-1}}.
\end{multline*}
Now we can choose $C_e = \frac{C_{(m_0+1)p}}{p} + \frac{C_e(Q)}{p^{d-1}}$.

\vspace{5pt}

\noindent
\emph{Assertion} $(2)$.\quad To prove the continuity of the limiting function $f_{A, \{\widetilde {I^q}\}}$, we consider the sequence $\left\{\widetilde {g_q}\colon \R_{\geq 0}\longto \R_{\geq 0}\right\}_{q=p^n}$ of continuous functions defined by
$$\widetilde {g_q}(x) = \lambda g_q\left(\tfrac{\lfloor xq\rfloor}{q}\right) + (1-\lambda) g_q\left(\tfrac{\lfloor xq\rfloor+1}{q}\right),\quad\mbox{if}\;\; x = \lambda \tfrac{\lfloor xq\rfloor}{q} + (1-\lambda) \tfrac{\lfloor xq\rfloor+1}{q}\;\;\mbox{where}\;\;\lambda\in[0,1).$$
Now  the uniform convergence of $\{g_q\}_q$ implies that $\{\widetilde {g_q}\}_q$ is a uniformly convergent sequence of continuous functions with the same limit. Hence $\lim_{q\to\infty} g_q$ exists and is a continuous function.
\end{proof}

Proposition \ref{p1nd} along with Theorem \ref{t3} gives a canonical notion of a density function for the family of graded finite length $R$-modules $\{\widetilde {I^q}/I^q\}_{q=p^n}$ as follows.

\begin{thm}\label{t1sat}
There exists a sequence of functions
$$\left\{h_q\colon \R_{\geq 0}\longto \R_{\geq 0}\right\}_{q=p^n}\quad\mbox{given by}\quad h_q(x) = \dfrac{\ell_k\big((\widetilde {I^q}/I^q)_{\lfloor xq\rfloor}\big)}{q^{d-1}/d!}$$
such that $\{h_q\}_q$ is pointwise convergent and the limit function $f_{A, \{\widetilde {I^q}/I^q\}}$ is a continuous function on $\R_{\geq 0}\setminus \{d_1\}$, where $d_1$ is the least degree among generators of the ideal $I$. Moreover, the function $f_{A, \{\widetilde {I^q}/I^q\}}$ has compact support and
$$\int_0^{\infty} f_{A,\{\widetilde {I^q}/I^q\}}(x)dx = \lim_{q\to \infty}\dfrac{\ell_A\big(\widetilde {I^q}/I^q\big)}{q^d/d!}.$$
\end{thm}
\begin{proof}
By Proposition \ref{p1nd}, the sequence of functions $\{g_q\}_q$  converges to the limit function $f_{A, \{\widetilde {I^q}\}}$ which is continuous. By Theorem \ref{t3}, the sequence $\left\{f_q\colon \R_{\geq 0}\longto \R_{\geq 0}\right\}_{q}$ given by $f_q(x) = \frac{\ell_k\big((I^q)_{\lfloor xq\rfloor}\big)}{q^{d-1}/d!}$ converges to $f_{A,\{I^q\}}$ which is continuous except possibly at $x=d_1$. This implies that
$$f_{A,\{\widetilde {I^q}/I^q\}}(x) = \lim_{q\to \infty} h_q(x) = \lim_{q\to \infty}\frac{\ell_k\big((\widetilde {I^q})_{\lfloor xq\rfloor}\big)}{q^{d-1}/d!} - \lim_{q\to \infty} \frac{\ell_k\big((I^q)_{\lfloor xq\rfloor}\big)}{q^{d-1}/d!} = f_{A, \{\widetilde {I^q}\}}(x) - f_{A, \{I^q\}}(x).$$
is a well-defined function which is continuous except possibly at $x= d_1$.
 
For every integer $e\geq 0$, let $\Delta_e =\int_0^e f_{A,\{\widetilde {I^q}/I^q\}}(x)dx$. Then $\{\Delta_e\}_e$ is a monotonically increasing sequence.

\noindent{\bf Claim}. $$\Delta_e = \lim\limits_{q\to \infty}\sum\limits_{m=0}^{eq}
\frac{\ell_k\big((\widetilde {I^q}/I^q)_m\big)}{q^d/d!}.$$

\noindent\emph{Proof of Claim}.\quad We showed that the sequence $\{h_q\}_q$ converges to $f_{A, \{\widetilde {I^q}/I^q\}}$ pointwise, where each $h_q$, being a step function, is a (Riemann) integrable function on the closed interval $I[0,e]$.
By definition, $h_q(x)\leq g_q(x)$. Again by Proposition \ref{p1nd}, the sequence $\{g_q\}_q$ is eventually increasing and converges to the continuous function $f_{A, \{\widetilde {I^q}\}}$.
Therefore on a compact interval $I[0,e]$, there exists a constant $C_e>0$ such that $0\leq h_q(x)\leq C_e$ for all $q\gg 0$. Now from the Lebesgue's dominated convergence theorem, we have
$$\lim_{q\to \infty}\int_0^e h_q(x)dx = \int_0^e\big(\lim_{q\to \infty}h_q(x)\big)dx.$$
Hence the claim follows as 
$$\lim_{q\to \infty}\int_0^e h_q(x)dx = \lim\limits_{q\to \infty}\sum\limits_{m=0}^{eq} \frac{\ell_k\big((\widetilde {I^q}/I^q)_m\big)}{q^d/d!}\quad\mbox{and}\quad \int_0^e\big(\lim_{q\to \infty}h_q(x)\big)dx = \Delta_e.$$

Now $\{\Delta_e\}_e$ is a monotonically increasing sequence which is bounded above by \cite{UV08}. So by taking $\lim_{e\to \infty}$ on both the sides we get
$$\lim_{q\to \infty} \frac{\ell_A({\tilde I^q}/I^q)}{q^d/d!} = \lim_{e\to \infty} \Delta_e = \int_0^{\infty} f_{A,\{\widetilde {I^q}/I^q\}}(x)dx.$$
\end{proof}

\begin{rmk}
We recall that the {\em epsilon multiplicity} $\varepsilon(I)$ of an ideal $I$ is
$$\varepsilon(I) = \limsup_{n\to \infty}\dfrac{\ell_A(H^0_{{\bf m}}(A/I^n))}{n^{d}/d!} = \lim_{n\to \infty}\dfrac{\ell_A(H^0_{\bf m}(A/I^n))}{n^{d}/d!},$$
where the last equality, in a more general setup, was proved by Cutkosky in \cite{Cut14}.
 
Therefore we have 
$$\varepsilon(I) = \lim_{n\to \infty}\frac{\ell_A(\widetilde {I^n}/I^n)}{n^d/d!} = \lim_{q\to \infty}\frac{\ell_A(\widetilde {I^q}/I^q))}{q^d/d!} = \int_0^{\infty}
f_{A,\{\widetilde {I^q}/I^q\}}(x)dx.$$
In other words, $f_{A,\{\widetilde {I^q}/I^q\}}$ is a density function for the epsilon multiplicity $\varepsilon(I)$.
\end{rmk}

\subsection{Density function in arbitrary characteristic}
Here we construct a density function for $(A, \{\widetilde {I^n}\})$ in arbitrary characteristic based on several results about the volume function of $\R$-divisors. For this we recall some relevant terminology and results from \cite{Laz04a} and \cite{LM09}.

\subsubsection{Preliminaries}\label{prelim}
Let $X$ be a $(d-1)$-dimensional projective variety over an algebraically closed field $k$. Let $\Div(X)$ denote the group of integral (Cartier) divisors on $X$.
Henceforth, an integral divisor will be referred as a divisor.

Let $N^1(X) = \Div(X)/\equiv_{\mathrm{num}}$ be the numerical equivalence classes of divisors on $X$, where two divisors $D_1$ and $D_2$ are \emph{numerically equivalent}, written $D_1\equiv_{\mathrm{num}}D_2$, if $\left(D_1\cdot C\right) = \left(D_2\cdot C\right)$ for every closed integral curve $C\subset X$. The first basic fact about $N^1(X)$ is the {\em theorem of base} which states that $N^1(X)$, the so-called N\'eron-Severi group, is a free abelian group of finite rank.

A $\Q$-Cartier (resp. $\R$-Cartier) divisor is, by definition, an element of $\mbox{Div}(X)\tensor_{\Z}\Q$ (resp. $\mbox{Div}(X)\tensor_{\Z}\R$). The resulting vector space of numerically equivalence classes is denoted by $N^1(X)_{\Q}$ (resp. $N^1(X)_{\R}$) and we have the equalities $N^1(X)_{\Q} = N^1(X)\tensor_{\Z}\Q$ and $N^1(X)_{\R} = N^1(X)\tensor_{\Z}\R$.

\begin{defn}\label{vol}
The {\em volume} of a divisor $D\in \mbox{Div}(X)$ is the nonnegative real number given by
$$\mbox{vol}_X(D) = \limsup_{n\to \infty} \dfrac{h^0\left(X, \sO_X(nD)\right)}{n^{d-1}/(d-1)!}.$$
For a fixed $a\in \N$ we have $\mbox{vol}_X(aD) = a^{d-1}\mbox{vol}_X(D)$. When $D$ is a $\Q$-divisor, the volume $\mbox{vol}_X(D)$ of $D$ can be defined by choosing some integer $a>0$ such that $aD$ is integral and then set
$$\mbox{vol}_X(D) = \frac{1}{a^{d-1}}\mbox{vol}_X(aD).$$
This definition is independent of the choice of $a$. Also numerically equivalent divisors have the same volume. Therefore we have a well-defined function $\mathrm{vol}_X\colon N^1(X)_{\Q} \longto \R_{\geq 0}$. Now this and the following theorem about the continuity property of the volume function implies that the volume function extends uniquely to a continuous function
$$\mbox{vol}_X\colon N^1(X)_{\R}\longto \R_{\geq 0}.$$
\end{defn}

\begin{thm}[Continuity of volume]\cite[Theorem $2.2.44$]{Laz04a}\label{contivol}
Let $X$ be as in Preliminaries \ref{prelim}. Fix any norm $\|\;\|$ on $N^1(X)_{\R}$ inducing the usual topology on that finite dimensional vector space. Then there is a constant $C>0$ such that
$$\left\vert\mathrm{vol}_X(\psi) - \mathrm{vol}_X(\psi')\right\vert \leq C \cdot \left(\max\left(\|\psi\|, \|\psi'\|\right)\right)^{d-2}\cdot\|\psi-\psi'\|,$$
for any two classes of $\psi$, $\psi' \in  N^1(X)_{\Q}$.
\end{thm}

\begin{defn}
We know that $N^1(X)_{\R}$ can be considered as a finite dimensional $\R$-vector space with the Euclidean topology. By a {\em cone} in $N^1(X)_{\R}$ we mean a set which is  stable under multiplication by positive scalars.

An $\R$-divisor $D$ is {\em big} if $\mbox{vol}_X(D)>0$. The {\em big cone} $\mbox{Big}(X)\subseteq N^1(X)_{\R}$ of $X$ is the (open) convex cone of all big $\R$-divisor classes on $X$. Similarly, the {\em ample cone} $\mbox{Amp}(X) \subseteq N^1(X)_{\R}$ of $X$ is the (open) convex cone of all ample $\R$-divisor classes on $X$. Also the {\em nef cone} $\mbox{Nef}(X) \subseteq N^1(X)_{\R}$ of $X$ is the (closed) convex cone of all nef $\R$-divisor classes on $X$, where a divisor $D\in \mbox{Div}(X)$ is {\em nef} (numerically effective) if $\left(D\cdot C\right)\geq 0$ for all closed integral curves $C\subset X$. The pseudoeffective cone $\overline{\mbox{Eff}}(X) \subseteq N^1(X)_{\R}$ is the closure of the convex cone spanned by the classes of effective $\R$-divisors on $X$. One has $N^1(X)_{\R} \supseteq \overline{\mbox{Eff}}(X) \supseteq \mbox{Nef}(X)$ and a basic fact is
$$\mbox{interior}\big(\overline{\mbox{Eff}}(X)\big) = \mbox{Big}(X) \quad\mbox{and}\quad \mbox{interior}\big(\mbox{Nef}(X)\big) = \mbox{Amp}(X).$$
 \end{defn}

\begin{thm}[Asymptotic Riemann-Roch]\label{ARR}
Let $X$ be as in Preliminaries \ref{prelim} and $D$ be a nef divisor on $X$. Then
$$h^0(X, \sO_X(nD)) = \dfrac{(D^{d-1})}{(d-1)!}\cdot n^{d-1} + O(n^{d-2}).$$
\end{thm}
In particular, if $D$ is a (integral) divisor which is big and nef then $\mbox{vol}_X(D) = (D)^{d-1} >0$.

The next theorem is a fundamental result which shows that the $\limsup$ in the definition of volume can be replaced by a limit. Under the assumptions that $k$ is algebraically clsoed, it has been proven by Okounkov \cite{Oko03} for an ample divisor $D$ and later by Lazarsfeld-Musta\c{t}\u{a} \cite{LM09} when $D$ is a big divisor. In the following generality, it was proved by Cutkosky \cite{Cut14}.

\begin{thm}\label{volume} Let $X$ be a $(d-1)$-dimensional projective variety over a field $k$ and $D\in \Div(X)$ be an integral divisor. Then
 $$\mathrm{vol}_X(D) = \limsup_{n\to \infty} \frac{h^0(X, \sO_X(nD))}{n^{d-1}/(d-1)!} = \lim_{n\to \infty} \frac{h^0(X, \sO_X(nD))}{n^{d-1}/(d-1)!}.$$
\end{thm}

\subsubsection{Existence of density function for $(A,\{\widetilde {I^n}\})$}

Recall Notations \ref{nsat} and \ref{densat}. Throughout the rest of this section we further assume that the underlying field $k$ is algebraically closed.

\begin{lemma}\label{VF}
Let $H$ and $E$ be the divisors on $X$ as in Notations \ref{densat}. Then for all $x\in \R_{\geq 0}$,
\begin{equation}\label{volformula}
\mathrm{vol}_X(xH-E) = \lim_{n\to \infty} \frac{h^0\left(X, \sO(\lfloor xn\rfloor H -nE)\right)}{n^{d-1}/(d-1)!}.
\end{equation}
\end{lemma}
\begin{proof}
If $x\in \N$ then it is a direct consequence of Theorem \ref{volume}. If $x\in\Q_{\geq 0}$ the lemma follows by the arguments similar to \cite[Lemma $2.2.38$]{Laz04a}. If $x\in \R_{\geq 0}$ the assertion follows from the continuity property of the volume function.
\end{proof}

\begin{notations}\label{alphabeta}
Let $H$ and $E$ be the divisors on $X$ as in Notations \ref{densat}. Define
$$\alpha_I = \min\left\{x\in \R_{\geq 0}\mid xH-E~~~\mbox{is pseudoeffective}\right\}\quad\mbox{and}\quad \beta_I = \min\left\{x\in \R_{\geq 0}\mid xH-E~~~\mbox{is nef}\right\}.$$
\end{notations}

Since $\mbox{Nef}(X) \subseteq \overline{\mbox{Eff}}(X)$ we have $\alpha_I\leq \beta_I$.

\begin{lem}\label{alpha_beta}
Let the hypotheses be as in Notations \ref{nsat}, \ref{densat} and \ref{alphabeta}. If $I\subset A$ is a homogeneous ideal such that $0 < \mathrm{height}(I)< d$ then
\begin{enumerate}
 \item $\alpha_I = \lim\limits_{n\to\infty}\dfrac{\min\big\{m\in\N \mid (\widetilde{I^n})_m\neq 0\big\}}{n}$, and
 \item $\beta_I = \lim\limits_{n\to\infty} \dfrac{\mathrm{reg}\big(\widetilde{I^n}\big)}{n}$.
\end{enumerate}
\end{lem}
\begin{proof}
The assertion $(2)$ follows directly from \cite[Theorem $5.4.22$]{Laz04a} or \cite[Theorem $3.2$]{CEL01}.

To prove assertion $(1)$, first define $\gamma_{I,n} = \min \big\{m\in\mathbb{N} \mid (\widetilde{I^n})_m \neq 0\big\}$. As $A$ is a domain it follows that $\{\gamma_{I,n}\}_{n\in\mathbb{N}}$ is a subadditive sequence, i.e., $\gamma_{I,n_1} + \gamma_{I,n_2} \geq \gamma_{I,n_1+n_2}$ for all $n_1,n_2\in\N$. Fekete's subadditive lemma then shows that $\gamma_I = \lim\limits_{n\to\infty} \frac{\gamma_{I,n}}{n}$ exists. Using equation \eqref{rsat} and Lemma \ref{VF}, one has for all $x\in \R_{\geq 0}$,
$$f_{A,\{\widetilde{I^n}\}}(x) = \lim_{n\to\infty}\dfrac{\ell_k\big((\widetilde{I^n})_{\lfloor xn\rfloor}\big)}{n^{d-1}/d!} = d\cdot \lim_{n\to\infty}\dfrac{h^0\left(X,\mathcal{O}_X\left(\lfloor xn\rfloor H -nE\right)\right)}{n^{d-1}/(d-1)!} = d\cdot \mathrm{vol}_X(xH-E).$$ Clearly, if $x<\gamma_I$ then $f_{A,\{\widetilde{I^n}\}}(x) = 0$ which implies that $\mathrm{vol}_X(xH-E)=0$, that is, $xH-E\nin \mathrm{Big}(X)$. Thus $\gamma_I\leq \alpha_I$.

For the converse, let $\{x_t\}_{t\geq 1}$ be a strictly decreasing sequence of rational numbers converging to $\gamma_I$. Write each $x_t = p_t/q_t$ for some integers $p_t, q_t>0$. There exists an integer $N_t>0$ such that for all $n\geq N_t$, $$x_t>\dfrac{\gamma_{I,q_tn}}{q_tn} \implies p_tn>\gamma_{I,q_tn} \implies \big(\widetilde{I^{q_tn}}\big)_{p_tn} \neq 0.$$ For each $t\geq 1$, define $y_t = x_t + \frac{1}{t} = \frac{tp_t+q_t}{tq_t}$ and observe that $\{y_t\}_{t\geq 1}$ is also a strictly decreasing sequence of rational numbers converging to $\gamma_I$ with $y_t>x_t$. Moreover,
$$f_{A,\{\widetilde{I^n}\}}(y_t) = \lim_{n\to\infty}\dfrac{\ell_k\big((\widetilde{I^n})_{\lfloor y_tn\rfloor}\big)}{n^{d-1}/d!} = \lim_{n\to\infty}\dfrac{\ell_k\big((\widetilde{I^{tq_tn}})_{(tp_t+q_t)n}\big)}{(tq_tn)^{d-1}/d!} = \lim_{n\to\infty}\dfrac{\ell_k\big((\widetilde{I^{tq_tN_tn}})_{(tp_t+q_t)N_tn}\big)}{(tq_tN_tn)^{d-1}/d!}.$$ Pick a nonzero homogeneous element $\theta \in \big(\widetilde{I^{tq_tN_t}}\big)_{tp_tN_t}$. Then
$$\lim_{n\to\infty}\dfrac{\ell_k\big((\widetilde{I^{tq_tN_tn}})_{(tp_t+q_t)N_tn}\big)}{(tq_tN_tn)^{d-1}/d!} \geq \lim_{n\to\infty}\dfrac{\ell_k\Big(\Big({(\widetilde{I^{tq_tN_t}})}^n\Big)_{(tp_t+q_t)N_tn}\Big)}{(tq_tN_tn)^{d-1}/d!} \geq \lim_{n\to\infty}\dfrac{\ell_k\left((\theta^n)_{(tp_t+q_t)N_tn}\right)}{(tq_tN_tn)^{d-1}/d!}>0.$$
Therefore, $\mathrm{vol}_X(y_tH-E)>0$ and from the continuity of the volume function we get $\alpha_I\leq \gamma_I$.
\end{proof}

\begin{rmk}
In contrast with Lemma \ref{alpha_beta} $(1)$ we have $\beta_I=\lim\limits_{n \to \infty} \frac{d(\widetilde {I^n})}{n}$ (see \cite[Theorem $3.2$]{CEL01}),
where $d(\widetilde {I^n})$ denotes the maximum among the minimal generating degrees of $\widetilde {I^n}$.
\end{rmk}

\begin{thm}\label{t2sat}
Let the hypotheses be as in Notations \ref{nsat}, \ref{densat} and \ref{alphabeta}. The following are true.
\begin{enumerate}
\item The sequence $\{g_n\}_{n\in\N}$ converges uniformly to the density function $f_{A, \{\widetilde {I^n}\}}$ on any given compact interval in $\R$. Moreover, the function $f_{A, \{\widetilde {I^n}\}}\colon \R_{\geq 0}\longto \R_{\geq 0}$ given by $$f_{A, \{\widetilde {I^n}\}}(x) = \lim\limits_{n\to \infty}\frac{\ell_k \big((\widetilde {I^n})_{\lfloor xn\rfloor}\big)}{n^{d-1}/d!}$$ is continuous on the interval $I[0,\infty)$.
\item The function $f_{A, \{\widetilde {I^n}\}}$ is continuously differentiable on the interval $I(\alpha_I, \infty)$ and also supported on $I(\alpha_I,\infty)$.
\item Further, for all $x\in I[\beta_I, \infty)$,
  $$f_{A, \{\widetilde {I^n}\}}(x) = \sum_{i=0}^{d-1}(-1)^i\frac{d!}{(d-1-i)!i!}(H^{d-1-i}\cdot E^i)x^{d-1-i}.$$
\end{enumerate}
\end{thm}

\begin{proof} \emph{Assertion $(1)$}. Let $n_0$ be an integer as in \eqref{rsat}. Then for  $x\in\R_{\geq 0}$, we have
$$ g_n(x) = \frac{\ell_k\big((\widetilde {I^n})_{\lfloor xn \rfloor}\big)}{n^{d-1}/d!} = \frac{h^0(X, \sO_X(\lfloor xn \rfloor H - nE))}{n^{d-1}/d!}\quad\mbox{if}\quad n\geq n_0.$$
Therefore by Lemma \ref{VF}, the function $f_{A, \{\widetilde {I^n}\}}$ given by
$$f_{A, \{\widetilde {I^n}\}}(x) = \lim_{n\to \infty} g_n(x) = d\cdot\mbox{vol}_X(xH - E)$$
is well-defined. Now the uniform convergence of $\{g_n\}_{n\in\N}$ on any compact interval and the continuity property of the density function $f_{A, \{\widetilde {I^n}\}}$ follows from the continuity property of the volume function (see Theorem \ref{contivol}).

\vspace{5pt}

\noindent{\emph{Assertion $(2)$}}.\quad  By \cite[Remark $4.29$]{LM09}, the volume function on $\mathrm{Big}(X)$ and therefore its restriction to the closed subset
$\{xH-E\mid x>\alpha_I\}$ is $\sC^1$-differentiable.

\vspace{5pt}

\noindent{\emph{Assertion $(3)$}}.\quad Note that $xH-E$ is nef for all $x\geq \beta_I$. By the asymptotic Riemann-Roch (see Theorem \ref{ARR}), for all $x\geq \beta_I$,
$$f_{A,\{\widetilde {I^n}\}}(x) = d\cdot\mbox{vol}_X(xH - E) = d (xH-E)^{d-1} = \sum_{i=0}^{d-1}(-1)^i\frac{d!}{(d-1-i)!i!}(H^{d-1-i}\cdot E^i)x^{d-1-i}.$$
\end{proof}

\begin{rmk}\label{densatu} 
Let the hypotheses be as above.
\begin{enumerate}
\item The density function $f_{A, \{\widetilde {I^n}\}}$ is strictly increasing on the interval $I(\alpha_I, \infty)$ such that for all $y>x>\alpha_I$,
$${f_{A, \{\widetilde {I^n}\}}(y)}^{\frac{1}{d-1}}\geq {f_{A, \{\widetilde {I^n}\}}(x)}^{\frac{1}{d-1}} + (y-x){e_0(A)}^{\frac{1}{d-1}},$$
where $e_0(A)$ is the multiplicity of $A$ with respect to its homogeneous maximal ideal $\mathbf{m}$.

To prove this we fix such an $x$ and $y$. Then by the log-concavity relation for volume functions (see \cite[Remark $2.2.50$]{Laz04a} or \cite[Corollary $4.12$]{LM09}), we get
\begin{align*}
 {f_{A, \{\widetilde {I^n}\}}(y)}^{\frac{1}{d-1}} = {(d\cdot\mbox{vol}_X(yH - E))}^{\frac{1}{d-1}} &\geq {(d\cdot\mbox{vol}_X(xH - E))}^{\frac{1}{d-1}} + {(d\cdot\mbox{vol}_X((y-x)H))}^{\frac{1}{d-1}}\\
 &= {f_{A, \{\widetilde {I^n}\}}(x)}^{\frac{1}{d-1}} + (y-x){(d\cdot\mbox{vol}_X(H))}^{\frac{1}{d-1}},
\end{align*}
where it is known that $(H)^{d-1} = (\sO_V(1))^{d-1} = e_0(A)$.
\item Assume that the ideal $I$ has a set of homogeneous generators of degrees $d_1< d_2<\ldots <d_l$. Then
 \begin{enumerate}
  \item $\alpha_I \leq d_1$. This is obvious as $(I^n)_{\lfloor xn\rfloor}\subset (\widetilde {I^n})_{\lfloor xn\rfloor}$, for all $x\in\R_{\geq 0}$ and $n\in\N$. But we can have  $\alpha_I < d_1$ (see the Example \ref{mono}, where $\alpha_I = 4$ and $d_1 = 5$).
\item  $\beta_I\leq d_l$.
To see the inequality $\beta_I\leq d_l$, we recall the arguement used in \cite[Lemma $1.4$]{CEL01}. For the associated ideal sheaf $\sI$ on $V=\mbox{Proj}~A$, we have a natural surjective map of $\sO_V$-modules (obtained by mapping it to the generators of $I$)
$$\sO_V(-d_1)^{\oplus a_1}\oplus \cdots \oplus \sO_V(-d_l)^{\oplus a_l}\longto \sI.$$
Hence $\sI(d_l)$ is generated by global sections. So for $\pi:X = \mathbf{Proj}\left(\oplus_{n\geq 0}\mathcal{I}^n\right)\longto V$,
$$\pi^*(\sI(d_l)) = \pi^{*}(\sI\tensor \sO_V(d_l)) = \sI\sO_X\tensor \pi^*(\sO(d_l)) = \sO_X(-E+d_lH)$$
is globally generated and therefore nef and hence $\beta_I \leq d_l$.
\end{enumerate}
\end{enumerate}
\end{rmk}

Now we are ready to give a well-defined notion of $\varepsilon$-density function.

\begin{defn}\label{epsilonden}
Let $A=\oOplus_{n\geq 0}A_n$ be a standard graded domain over an algebraically closed field $A_0=k$ of dimension $d\geq 2$. Let $I\subset A$ be a homogeneous ideal, which is generated in degrees $\{d_1<d_2<\cdots <d_l\}$. Then the $\varepsilon$-density function of $A$ with respect to $I$ is the function $f_{\varepsilon(I)}\colon \R_{\geq 0}\longto \R_{\geq 0}$ given by
$$x \longmapsto  \lim_{n\to \infty}\frac{\ell_k\big({(\widetilde {I^n}/I^n)_{\lfloor xn\rfloor}\big)}}{n^{d-1}/d!}.$$
\end{defn}

\begin{thm}\label{MT} Let $A$ and $I$ be as in Definition \ref{epsilonden}. Also assume that $0<\mathrm{height}\,(I)<d$.
\begin{enumerate}
 \item The function $f_{\varepsilon(I)}$ is continuous everywhere possibly except at $x=d_1$. Further, it is continuously differentiable outside the finite set $\{\alpha_I, d_1, \ldots, d_l\}$.
\item The support of $f_{\varepsilon(I)}$ is contained in the interval $I(\alpha_I, d_l]$, and $\alpha_I >0$.
\item Further, $$\int_{0}^{\infty}f_{\varepsilon(I)}(x)dx = \varepsilon(I).$$
\end{enumerate}
\end{thm}
\begin{proof}{\em Assertion $(1)$}.\quad Consider the sequence
$$h_n\colon R_{\geq 0} \longto \R_{\geq 0}\quad\mbox{given by}\quad x\mapsto \frac{\ell_k\big((\widetilde {I^n}/I^n)_{\lfloor xn\rfloor}\big)}{n^{d-1}/d!}.$$
Then by Theorem \ref{t3} and Theorem \ref{t2sat},
$$\lim_{n\to \infty}h_n(x) = \lim_{n\to \infty}\frac{\ell_k\big((\widetilde {I^n})_{\lfloor xn\rfloor}\big)}{n^{d-1}/d!} - \lim_{n\to \infty}\frac{\ell_k\big(({I^n})_{\lfloor xn\rfloor}}{n^{d-1}/d!} = f_{A, \{\widetilde {I^n}\}}(x) - f_{A, \{I^n\}}(x)$$
is a well-defined continuous function on $\R_{\geq 0} \setminus \{d_1\}$ and is continuously differentiable outside the finite set $\{\alpha_I, d_1, \ldots, d_l\}$.

\vspace{5pt}

\noindent{\em Assertion $(2)$}. By Remark \ref{densatu} (2)(b) and Theorem \ref{t3}, there exist polynomials ${\tilde {\bf p}}(x)$ and ${\bf p}_l(x)$ in $\Q[x]$
such that
$$f_{A,\{\widetilde {I^n}\}}(x) = {\tilde {\bf p}(x)}\quad \text{and}\quad f_{A,\{I^n\}}(x) = {\bf p}_l(x) \quad \forall x\in I(d_l, \infty).$$
Therefore ${\tilde {\bf p}}(x)-{\bf p}_l(x)\geq 0$ for all $x\in I(d_l, \infty)$. Now as argued in the proof of Theorem \ref{t1sat} (replacing $q$ by $n$ everywhere) there is a constant $C_0$ such that for every $m\in \N$,
$$\int_{0}^{m}f_{\varepsilon(I)}(x)dx = \int_{0}^{m}(f_{A,\{\widetilde {I^n}\}}(x)-f_{A,\{I^n\}}(x))dx\leq C_0$$
which implies that
$\int_{d_l}^m ({\tilde {\bf p}}(x)-{\bf p}_l(x))dx \leq C_0 $, for all integers $m\geq d_l$. This forces
${\tilde {\bf p}}(x) = {\bf p}_l(x)$ and therefore $f_{\varepsilon(I)}(x) = 0$ for all $x\in I(d_l,\infty)$.

Now we need to prove that $\alpha_I> 0$. Recall Hartshorne's definition \cite[Chapter II, Exercise 6.2]{Har77} of the degree of a divisor. We know that $\deg H = e_0$, where $e_0 = e(A, {\bf m})$ is the multiplicity of $A$ with respect to the homogeneous maximal ideal ${\bf m}$. Since $E$ is an effective divisor, so $\deg E>0$. Therefore if we choose $\alpha = \frac{\deg E}{e_0}$ then $H^0(X,\sO_X(\lfloor xn\rfloor H-nE)) = 0$ for all $n\in \N_{>0}$ and $x \in I[0,\alpha)$. In particular, $\mbox{vol}_X(xH-E) = 0$ for all $x\in I[0,\alpha]$. Note that the volume function is monotonically increasing. Since $\overline{\mathrm{Big}(X)} = \overline{\mathrm{Eff}}(X)$, hence the $\R$-divisor divisor $xH-E$ is big for all $x\in I(\alpha_I,\infty)$ and has positive volume. Thus, $\alpha_I\geq \alpha = \frac{\deg E}{e_0}>0$.
 
\vspace{5pt}
 
\noindent{\em Assertion $(3)$} From the Lebesgue's dominated  convergence theorem, one gets that
$$\int_0^{\infty} f_{\varepsilon(I)}(x)dx = \lim_{n\to \infty}\int_0^{d_l}h_n(x)dx = \lim_{n\to \infty}\sum_{m=0}^{d_ln}\frac{\ell_k\big((\widetilde {I^n}/I^n)_m\big)}{n^{d}/d!} = \lim_{n\to \infty}\frac{\ell_A({\tilde I^n}/{I^n})}{n^{d}/d!} = \varepsilon(I).$$
\end{proof}

\begin{rmk}
Let $A$ and $I$ be as in Definition \ref{epsilonden}. Thus, by the continuity of volume function and the existence of density function $f_{\varepsilon(I)}$, we have shown that
$$\varepsilon(I) = \limsup_{n\to \infty}\frac{\ell_A\left(H^0_{\bf m}(A/{I^n})\right)}{n^{d}/d!} = \lim_{n\to \infty}\frac{\ell_A\left(H^0_{\bf m}(A/{I^n})\right)}{n^{d}/d!}.$$
\end{rmk}
 
\begin{rmk}\label{rmkintro}
From the proof of Theorem \ref{MT} $(2)$, one has $f_{A, \{\widetilde {I^n}\}}(x) = f_{A,\{I^n\}}(x)$ for all $x\in I(d_l,\infty)$, where it is same as the  polynomial given in Theorem \ref{t2sat} $(3)$. Hence the assertion follows by Theorem \ref{t1}.
\end{rmk}

\begin{rmk}\label{rmkfield}
Equation \eqref{rsat} continues to hold over an arbitrary field. The saturation density function $f_{A,\{\widetilde{I^n}\}}$ was then studied by interpreting it as a volume function. This was done following \cite{Laz04a} and \cite{Laz04b}, where it was assumed that the underlying field $k$ is algebraically closed. In general, extending these results to arbitrary fields cannot follow by simply making a flat base change, see \cite[Example 2.4]{Cut15}. However, we can eliminate this condition on $k$ by following later results of Cutkosky, see \cite[Theorem 10.7]{Cut14} and \cite[Theorem 2.5 and Theorem 5.6]{Cut15}.
\end{rmk}

\subsection{Mixed multiplicities}\label{mixed_mult}

We first recall the classical notion of mixed multiplicities of a pair of ideals. Let $(A,{\bf m})$ be a Noetherian local ring of dimension $d>0$. Let $I$ be an ${\bf m}$-primary ideal and $J$ be an arbitrary ideal with positive height. Bhattacharya \cite{Bha57} showed that for all $u\gg 0$ and $v\gg 0$, the bivariate Hilbert function $\ell_A \left(I^{u}J^v/I^{u+1}J^v\right)$ agrees with a bivariate numerical polynomial $Q(u,v)$ of total degree $d-1$. Moreover, we can write
\begin{equation*}
Q(u,v) = \sum_{i=0}^{d-1} \dfrac{e_i\left(I \vert J\right)}{i!(d -1-i)!}u^{d-1-i}v^i + \text{lower degree terms},
\end{equation*}
where the coefficients $e_i\left(I\vert J\right)$ are integers. Risler and Tessier \cite{Tei73} studied these coefficients when $J$ is ${\bf m}$-primary and called them the \emph{mixed multiplicities of $I$ and $J$}. Later in \cite{KV89}, Katz and Verma analyzed the coefficients when $J$ is an arbitrary ideal.

Now let $A$ and $I$ be as in Definition \ref{epsilonden}. It follows from Corollary \ref{HT} or \cite[Theorem 4.2]{HT03} that there exist constants $m_0,n_0>0$ such that for all $m\geq d_lm+m_0$ and $n\geq n_0$, the Hilbert function $\ell_k\left({(I^n)}_m\right)$ is given by a bivariate numerical polynomial $P(m,n)$ of total degree $d-1$. Further, we can write
\begin{equation*}
P(m,n) = \sum_{i=0}^{d-1} \dfrac{e_i\left(A[It]\right)}{i!(d -1-i)!}m^in^{d-1-i} + \text{lower degree terms},
\end{equation*}
where the coefficients $e_i\left(A[It]\right)$ are all integers.

The connection between intersection numbers and mixed multiplicities of $\mathbf{m}$-primary ideals in a local ring is well-known, for example see \cite[Chapter 1.6]{Laz04a}. The following result can be seen as a generalization in the graded framework.

\begin{prop}
 Let $A$ and $I$ be as in Definition \ref{epsilonden}. Let $\bf m$ denote the unique homogeneous maximal ideal of $A$. Then the following statements are true.
 \begin{enumerate}
  \item For all $i=0,\ldots,d-1$, we have $e_i\left(A[It]\right) = (-1)^{d-1-i} \left(H^i\cdot E^{d-1-i}\right)$.
  \item Suppose that $I$ is generated in equal degrees $d_l$. Then for all $i=0,\ldots,d-1$, we have $$e_i\left({\bf m}\vert I\right) = \sum_{j=0}^i \binom{i}{j}d_l^j e_{d-1-i+j}\left(A[It]\right).$$ In particular, $e_i\left({\bf m}\vert I\right) =\left((d_lH-E)^i\cdot H^{d-1-i}\right)$ for all such $i$.
 \end{enumerate}
\end{prop}
\begin{proof}
Part $(1)$ of this proposition follows from part $(3)$ of Theorem \ref{t2sat} and Remark \ref{rmkintro}. Henceforth, let $I$ be generated in equal degrees $d_l$. Then the required expression for $e_i\left({\bf m}\vert I\right)$ is given in \cite[Lemma 5.3]{DDRV24}. Using part $(1)$, we have $e_{d-1-i+j}\left(A[It]\right) = (-1)^{i-j}\left(H^{d-1-i+j}\cdot E^{i-j}\right)$. Now substituting this into the expression for $e_i\left({\bf m}\vert I\right)$, we get
$$e_i\left({\bf m} \vert I\right) = \sum_{j=0}^i (-1)^{i-j} \binom{i}{j} d_l^j \left(H^{d-1-i+j}\cdot E^{i-j}\right) = \left((d_lH-E)^{i}\cdot H^{d-1-i}\right),$$ where the last equality is a consequence of the binomial exapansion.
\end{proof}

\section{Some properties of density functions}\label{sec6}

\subsection{\texorpdfstring{Density function $f_{A,\{I^n\}}$ and diagonal subalgebras of $A[It]$}{Density function and diagonal subalgebras}}
We recall that the homogenity property of the volume function on the integral divisors $\mathrm{Div}(X)$ allowed the function to extend to $\Q$-Cartier divisors of $X$. Further the continuity property of the volume function on $\Q$-divisors allowed it to extend to  $\R$-Cartier divisors.

Here the  theory of $f_{A, \{I^n\}}$ asserts a similar phenomenon for the `rescaled' diagonal multiplicity function, denoted by ${\tilde e_{\mathrm{rHSm}}}$.

For each $(p,q)\in\N^2$, we have a $(p,q)$-diagonal subalgebra $R_{\Delta_{I(p,q)}}$ as in Definition \ref{dg}, where
$A$ is a domain in addition. If $p/q<d_1$ then $R_{\Delta_{I(p,q)}} = k$. Let $p/q > d_1$ then it is known that $R_{\Delta_{I(p,q)}}$ is of dimension $d$. This can also be seen here easily because $f_{A,\{I^n\}}(p/q)\neq 0$, which implies that the multiplicity of the $(p,q)$-diagonal subalgebra $e(R_{\Delta_{I(p,q)}}) = \tfrac{q^{d-1}}{d}\cdot f_{A,\{I^n\}}(p/q) \neq 0$.

Moreover, the homogenity property holds: If $a\in \N$ and $(p,q)\in \N^2$ then $e(R_{\Delta_{I(ap,aq)}}) = a^{d-1}e(R_{\Delta_{I(p,q)}})$. This gives a well-defined function
$${\tilde e_{\mathrm{rHSm}}}:\Q_{\geq 0} \longto \Q_{\geq 0}\quad\mbox{where}\quad \tfrac{p}{q} \longmapsto \tfrac{e(R_{\Delta_{I(p,q)}})}{q^{d-1}} = \tfrac{1}{d}\cdot f_{A,\{I^n\}}(p/q).$$
Now the continuity property of  the density function  $f_{A,\{I^n\}}$ on $\R\setminus \{d_1\}$ implies that this function extends uniquely to ${\tilde e_{\mathrm{rHSm}}}:\R_{\geq 0}\longto \R_{\geq 0}$. In particular, we have the following.

\begin{cor}
Let $I$ and $J$ be homogeneous ideals in a standard graded domain $A$ of dimension $d\geq 1$ over a field $k$.
Then 
\begin{equation*}
f_{A, \{I^n\}}(x) = f_{A, \{J^n\}}(x) \;\;\mbox{for all}\;\; x\in \R_{\geq 0}\;\;
\iff\;\; e(R_{\Delta_{I(p,q)}}) = e(R_{\Delta_{J(p,q)}})\;\;\mbox{for all}\;\; (p,q)\in \N^2.
\end{equation*}
 \end{cor}

\subsection{Density functions and integral closure}

\begin{thm}\label{6.1}
Let $I$ be a homogeneous ideal in a standard graded algebra $A$ over a field $k$. Let $J=\overline{I}$ be the integral closure of the ideal $I$ in $A$. Then
\begin{enumerate}
\item[(1)] $f_{A,\{I^n\}}(x) = f_{A,\{J^n\}}(x)$ for all $x\in \R_{\geq 0}\setminus \{d_1, \ldots, d_l\}$, where $\{d_1 < \cdots < d_l\}$ denote the distinct degrees of a set of homogeneous generators of $I$.
\item[(2)] If in addition $A$ is an integral domain then
           \begin{enumerate}
               \item[(a)] $f_{A,\{I^n\}}(x) = f_{A,\{J^n\}}(x)$ for all $x \in \R_{\geq 0}$ and
               \item[(b)] $f_{A,\{\widetilde {I^n}\}}(x) = f_{A,\{\widetilde {J^n}\}}(x)$ for all $x\in \R_{\geq 0}$.
           \end{enumerate}
\end{enumerate}
\end{thm}
\begin{proof}
{\em Assertion $(1)$}.\quad First we note that since $J$ has no term of degree $<d_1$, the function $f_{A,\{J^n\}} \equiv 0$ on $I[0, d_1)$. There exists integer an $c\geq 0$ such that $J^cI^{n-c} = J^n$ for all $n\geq c$. Now for $x\in I(d_i, d_{i+1})$ we can choose $n\gg 0$ such that both $(\lfloor xn \rfloor, n)$ and $(\lfloor xn \rfloor, n-c)$ lie in ${\mathfrak R}C_i$.
Then by Theorem \ref{t1}
$$\ell_k\big((I^n)_{\lfloor xn \rfloor}\big) = P_i(\lfloor xn \rfloor, n) + Q_i(\lfloor xn\rfloor, n)\quad\mbox{and}\quad
\ell_k\big((I^{n-c})_{\lfloor xn \rfloor}\big) = P_i(\lfloor xn \rfloor, n-c)+Q_i(\lfloor xn\rfloor, n-c),$$
where $\deg Q_i(X,Y) < \deg P_i(X,Y)$. Since $(I^n)_{\lfloor xn\rfloor} \subset (J^n)_{\lfloor xn\rfloor} \subset (I^{n-c})_{\lfloor xn\rfloor}$ we get that 
$f_{A,\{I^n\}}(x) = f_{A,\{J^n\}}(x)$ for $x\in I(d_i, d_{i+1})$.

\vspace{5pt}

\noindent{\em Assertion $(2)(a)$}.\quad Now if $A$ is a domain then the functions $f_{A,\{I^n\}}$ and $f_{A,\{J^n\}}$ are continuous on $I(d_1, \infty)$. Hence equality on the dense set $\R_{\geq 0}\setminus \{d_1, \ldots, d_l\}$ implies equality on $I(d_1, \infty)$. Further $J\subset {\bar I}$ implies that  the least degree of generator of $J$ is $d_1$.
Hence the equality of the density functions holds for $x<d_1$ too.

Now for $x=d_1$ we note that both $(d_1,1)$-diagonal subalgebra of
$A[It]$ and $A[Jt]$ are standard graded as generated by $I_{d_1}$ and $J_{d_1}$ respectively.
Therefore 
$R_{\Delta_{I(d_1,1)}} \longto R_{\Delta_{J(d_1,1)}}$ is a finite  extension of  integral domains. Hence 
$$f_{R,\{I^n\}}(d_1) = d\cdot e(R_{\Delta_{I(d_1,1)}}) = 
d\cdot e(R_{\Delta_{J(d_1,1)}} = f_{R, \{J^n\}}(d_1).$$

\vspace{5pt}

\noindent{\em Assertion $(2)(b)$}.\quad Note that we have $\widetilde {I^n} \subseteq \widetilde {J^n} \subseteq \widetilde {I^{n-c}}$ for all $n\geq c$. This gives, for $x\in \R_{\geq 0}$
$$\lim_{n\to \infty}\dfrac{\ell_k\big((\widetilde {I^n})_{\lfloor xn\rfloor}\big)}{n^{d-1}/d!}
\leq \lim_{n\to \infty}\dfrac{\ell_k\big((\widetilde {I^n})_{\lfloor xn\rfloor}\big)}{n^{d-1}/d!}
\leq \lim_{n\to \infty}\dfrac{\ell_k\big((\widetilde {I^{n-c}})_{\lfloor xn\rfloor}\big)}{n^{d-1}/d!}.$$
Now 
$$\lim_{n\to \infty}\dfrac{\ell_k\big((\widetilde {I^{n-c}})_{\lfloor xn\rfloor}\big)}{n^{d-1}/d!} = \lim_{n\to \infty}\frac{h^0\left(X,\mathcal{O}_X\left(\lfloor xn\rfloor H-(n-c)E\right)\right)}{n^{d-1}/d!} = d\cdot\mbox{vol}_X(xH-E),$$
where the last equality follows by the arguments as in Lemma \ref{VF}. In particular, $f_{A,\{\widetilde {J^n}\}}(x) = d\cdot\mbox{vol}_X(xH-E)$ and therefore $f_{A,\{\widetilde {I^n}\}}\equiv f_{A,\{\widetilde {J^n}\}}$.
\end{proof}

\begin{lemma}
Let $A=k[X,Y]$ be a standard graded polynomial ring in two variables over a field $k$. Let $I_1\subseteq I_2$ be two homogeneous ideals in $A$ such that $f_{A,\{I^n_1\}}(x)=f_{A,\{I^n_2\}}(x)$ for all real numbers $x\geq 0$. Then $I_1\subseteq I_2$ is a reduction.
\end{lemma}
\begin{proof}
Let $\mathbf{m}=(X,Y)$ denote the homogeneous maximal ideal of $A$. Since $A$ is a two-dimensional UFD, we may write $$I_j = (g_j)\cdot Q_j \quad \mbox{ for $j=1,2$,}$$
where $g_j$ is a homogeneous element and $Q_j$ is an $\mathbf{m}$-primary homogeneous ideal. Notice that $\widetilde{I_j^n} = (g_j^n)$ for all $n\geq 1$. So
\[\ell_A\left(H^0_{\mathfrak{m}}\left(A/I_j^n\right)\right)=\ell_A\big(\widetilde{I_j^n}/I_j^n\big)=\ell_A(A/Q_j^n),\]
as $g_j$ is a nonzerodivisor on $A$. It thus follows that $\varepsilon(I_j)=e(Q_j)$. There exists an integer $c>\max\{\deg g_1, \deg g_2\}$ such that $(I_j^n)_m={(\tilde I_j^n)}_m$ for every $m \geq cn$. So for all real numbers $x\geq c$,
$$f_{A,\{ I^n_j\}}(x)=\lim_{n\to\infty}\dfrac{\ell_k \big((I_j^n)_{\lfloor xn\rfloor}\big)}{n/2} = \lim_{n\to\infty}\dfrac{\ell_k \big((\widetilde {I_j^n})_{\lfloor xn\rfloor}\big)}{n/2} = \lim_{n\to\infty}\dfrac{\ell_k \big((g_j^n)_{\lfloor xn\rfloor}\big)}{n/2} = 2(x-\deg g_j).$$ The hypothesis then implies that $\deg g_1 = \deg g_2$. From the inclusion $(g_1) = \widetilde{I_1} \subseteq \widetilde{I_2} = (g_2)$, we conclude that $(g_1) = (g_2)$. Therefore, $Q_1 \subseteq Q_2$ and $f_{A,\{\widetilde {I^n_1}\}}(x)=f_{A,\{\widetilde {I^n_2}\}}(x)$. Hence $f_{\varepsilon(I_1)}(x)=f_{\varepsilon(I_2)}(x)$. Thus $e(Q_1)=e(Q_2)$. The statement now follows from Rees's theorem.
\end{proof}

\subsection{When the filtration \texorpdfstring{$\{\widetilde {I^n}\}_{n\in\N}$}{}  is Noetherian}

\begin{cor}\label{rational}
Let $A$ be a standard graded domain over a field $k$ and let $I\subseteq A$ be a homogeneous ideal such that $\{\widetilde {I^n}\}_{n\in\N}$ is a Noetherian filtration.
Then $f_{\varepsilon(I)}$ is a piecewise polynomial function in the following sense: There exists a set of nonnegative rational numbers $a_1< \cdots <a_r$ and polynomials ${\bf q}_1(x), \ldots, {\bf q}_r(x)$ in $\Q[x]$ of degrees at most $\dim A-1$ such that $$f_{\varepsilon(I)}(x) = \begin{cases}
                                   0 & \text{if}\;\; x<a_1 \;\text{or}\;x>a_r\\
                                   {\bf {q}}_i(x) & \text{if}\;\; a_i<x<a_{i+1}, \;\text{where}\; i=1,\ldots,r-1                                                                                                                                                                                                              \end{cases}.
$$
In particular, $\varepsilon(I)$ is a rational number.
\end{cor}

\begin{proof}
Let $h_1< \cdots <h_{r_1}$ be a distinct set of degrees of a homogeneous generating set of $I$. Hence the Rees algebra $A[It]$ is generated in bidegrees $(h_1, 1), \ldots, (h_{r_1},1)$. Set $h_{r_1+1} = \infty $. Let $\oOplus_{n\geq 0}\widetilde {I^n}t^n$ be generated in bidegrees $\{(d_1, e_1),\ldots, (d_s, e_s)\}$ as a bigraded $k$-algebra. Without any loss of generality, we may assume that $d_1/e_1 < \cdots < d_{r_2}/e_{r_2}$ are the possible distinct rationals from the set $\{d_i/e_i\}_{i=1}^s$. Further, let $d_{r_2+1}/e_{r_2+1} = \infty$.

Then the function $f_{A,\{I^n\}}$ agrees with a polynomial with rational coefficients on each open interval $(h_i, h_{i+1})$. Similarly, $f_{A,\{\widetilde {I^n}\}}$ is given by a polynomial with rational coefficients on every interval $(d_j/e_j, d_{j+1}/e_{j+1})$.
If we present the set
$$\{h_1, \ldots, h_{r_1}, d_1/e_1, \ldots, d_{r_2}/e_{r_2}\} = \{a_1 < \cdots <a_r\}$$
then the corollary follows as $f_{\varepsilon(I)} = f_{A,\{\widetilde {I^n}\}}-f_{A,\{I^n\}}$.
\end{proof}

\begin{rmk}\label{Herzog-hibi??}
The rationality of $\varepsilon(I)$ is known in a more general context: If $I$ is an ideal in a local ring $(A,\mathbf{m})$ such that $\{\widetilde {I^n}\}_{n\in\N}$ is a Noetherian filtration, then by \cite[Theorems 2.4 and 3.3]{HPV08} the function $\ell_A(\widetilde {I^n}/ I^n)$ is of quasi-polynomial type. Moreover, all the associated polynomials have the same degree and the same leading coefficient. Consequently, $\varepsilon(I)$ is a rational number.
\end{rmk}

\begin{ex}\label{mono}
Let $I$ be a monomial ideal in a standard graded polynomial ring $A$ over a field $k$. Then by \cite[Theorem 3.2]{HHT07} the filtration $\{\widetilde {I^n}\}_{n\in\N}$ is Noetherian. In particular, the density function $f_{A,\{\widetilde {I^n}\}}$ and therefore the $\varepsilon$-density function $f_{\varepsilon(I)}$ is a piecewise polynomial function and $\varepsilon(I)$ is a rational number. This fact also follows from \cite[Theorem 5.1]{JM13}.

Here we give a concrete  example. Let $A = k[X, Y, Z]$ and let
$$I= (X^2Y^3, X^3Y^2, XY^2Z^4, XY^3Z^3) = (X)\cap (Y^2)\cap (X^2, Z^3)\cap (X^3, Y^3, Z^4).$$
Then by \cite[Lemma $3.1$]{HHT07}
$$\widetilde {I^n} = {\big(\widetilde {I}\big)}^n = {(X^2Y^2, XY^2Z^3)}^n \quad\mbox{for all}\quad n\geq 1$$
and by {\sc Macaulay2} the bigraded Hilbert series for $\oOplus_{(m,n)\in\N^2}\big(\widetilde{I^n}\big)_mt^n$ is
$$\sum_{(m,n)\in \N^2}\ell_k\big((\widetilde {I^n})_m\big)x^my^n = \dfrac{1-x^7y}{(1-x)^3(1-x^6y)(1-x^4y)}.$$
Then 
$$f_{A,\{\widetilde {I^n}\}}(x) = \lim_{n\to \infty}\dfrac{\ell_k \big((\widetilde {I^n})_m\big)}{n^2/3!} = \begin{cases}
                                                                                                            0 & \mbox{if}\;\; x\leq 4\\
                                                                                                            \frac{9}{2}(x-4)^2 & \mbox{if}\;\; 4\leq x\leq 6\\
                                                                                                            3x^2-18x+18 & \mbox{if}\;\; x\geq 6
                                                                                                           \end{cases}.
$$
\end{ex}

\subsection{When \texorpdfstring{$\varepsilon(I) = 0$}{}}

Then $f_{A,\{\widetilde {I^n}\}} = f_{A,\{ I^n\}}$. In particular, $\alpha_I = d_1$ and the volume function $\mbox{vol}_X$ is a piecewise polynomial function with rational coefficients when restricted to the ray $\{xH-E\mid x\in \R_{\geq 0}\}$ of $N^1(X)_{\R}$.

\subsection{When \texorpdfstring{$I$}{I} is  generated in  equal degrees \texorpdfstring{$d_1 = \cdots = d_l$}{d1=...=dl}}
Then there are two possibilities:
\begin{enumerate}
\item  $\varepsilon(I) = 0$, and in this case $d_l=\alpha_I$,
\item $\varepsilon(I) \neq 0$, and in this case $d_l>\alpha_I$ and $\varepsilon(I) \geq (d_l-\alpha_I)^d\cdot e_0(A)$.
\end{enumerate}
Note that $d_1 = \cdots = d_l$ implies that $$f_{\varepsilon(I)}(x) = \begin{cases}
                                                                       f_{A,\{\widetilde {I^n}\}}(x) & \text{if}\;\;\alpha_I \leq x < d_l\\
                                                                       0 & \text{if}\;\; x\leq \alpha_I\;\;\text{or}\;\; x>d_l
                                                                      \end{cases}.
$$
On the other hand, the log-concavity relation for volume functions implies
$$f_{A,\{\widetilde{I^n}\}}(x) \geq d\cdot (x-\alpha_I)^{d-1}e_0(A)$$
for all real numbers $x>\alpha_I$. Hence, $\varepsilon(I) = 0$ if and only if $d_l = \alpha_I$. Further
$$\varepsilon(I) = \int_0^{\infty}f_{\varepsilon(I)}(x)dx = \int_{\alpha_I}^{d_l}f_{A, \{\widetilde{I^n}\}}(x)dx \geq de_0(A)\cdot \int_{\alpha_I}^{d_l}(x-\alpha_I)^{d-1}dx = (d_l-\alpha_I)^d\cdot e_0(A).$$

\subsection{When \texorpdfstring{$\alpha_I = \beta_I$}{alphaI=betaI}}
In this case $\varepsilon(I)$ is an algebraic number. This is because
\[f_{A, \{\widetilde {I^n}\}}(x)=
\begin{cases}
0 & \mbox{ if } x \leq \alpha_I\\
d \cdot (xH-E)^{d-1} & \mbox{ if } x \geq \alpha_I
\end{cases}.\]
Since $f_{A, \{\widetilde {I^n}\}}$ is a continuous function so $\alpha_I$ must be a root of the polynomial $(xH-E)^{d-1}$ and therefore is an algebraic number. Let $I$ be generated in degrees $d_1\leq \cdots \leq d_l$. Then
\[\varepsilon(I)=\int_{0}^{d_l}\big(f_{A, \{\widetilde {I^n}\}}(x)- f_{A, \{I^n\}}(x)\big)dx=\int_{\alpha_I}^{d_l}f_{A, \{\widetilde {I^n}\}}(x)dx-\int_{d_1}^{d_l} f_{A, \{I^n\}}(x) dx,\]
where the first term is an algebraic number and the second term is a rational number. In particular, if $\alpha_I= \beta_I$ and $\varepsilon(I)\neq 0$ then $\beta_I< d_l$.

\section{Density functions in dimensions two and three}\label{sec7}

\begin{propose}
Let $A$ be a two dimensional standard graded domain over an algebraically closed field $k$. Let $I\subseteq A$ be a nonzero homogeneous ideal. Then $\alpha_I = \beta_I$ and $\varepsilon(I)$ is a rational number.
\end{propose}
\begin{proof}
As $X = \mathbf{Proj} \left(\oplus_{n\geq 0}\mathcal{I}^n\right)$ is a projective curve, it follows from \cite[Example 1.2.3]{Laz04a} that $\overline{\mathrm{Eff}}(X)=\overline{\mathrm{Amp}(X)}=\mathrm{Nef}(X)$. Consequently, $\alpha_I=\beta_I$. Now as seen above $\alpha_I$ is a root of the polynomial $(\deg H)x-(\deg E)$ and hence it is a rational number. Thus $\varepsilon(I)$ is also rational.
\end{proof}

\begin{propose}\label{dimthree}
Let $A$ be a three dimensional standard graded domain over an algebraically closed field $k$. Let $I\subseteq A$ be a nonzero homogeneous ideal. Then $f_{A, \{\widetilde {I^n}\}}$ and $f_{\varepsilon(I)}$ are given by piecewise (possibly countably many pieces) real polynomial functions of degrees at most two.
\end{propose}
\begin{proof}
If $I$ is of finite colength then $\widetilde {I^n} = A$ for all $n$, hence the assertion is obvious. So we assume that $I$ is not of finite colength. In view of Theorems \ref{t3} and \ref{MT}, it suffices to show that $f_{A,\{\widetilde {I^n}\}}(x)$ is given by a piecewise real polynomials having countably many pieces. From the proof of Theorem \ref{t2sat}, it follows that there exist Cartier divisors $H$ and $E$ on the projective surface $X = \mathbf{Proj} \left(\oplus_{n\geq 0}\mathcal{I}^n\right)$ such that
$$3\cdot\mathrm{vol}_X(xH-E) = 3\cdot\lim\limits_{n\to\infty}\dfrac{h^0\left(X,\mathcal{O}_{X}\left(\lfloor xn\rfloor H - nE\right)\right)}{n^2/2} = \lim\limits_{n\to\infty}\dfrac{\ell_k\big({(\widetilde{I^{n}})}_{\lfloor xn\rfloor}\big)}{n^2/3!} =: f_{A,\{\widetilde{I^n}\}}(x)$$
for all real numbers $x\geq 0$. Let $\varphi \colon X^{\prime} \to X$ be a resolution of singularities and consider the pullbacks $H^{\prime} = \varphi^*H$ and $E^{\prime} = \varphi^*E$ to $X^{\prime}$.
Since the volume function is a birational invariant, $$\mathrm{vol}_{X^{\prime}}(xH^{\prime}-E^{\prime}) = \mathrm{vol}_{X}(xH-E) = \frac{1}{3}\cdot f_{A,\{\widetilde{I^n}\}}(x).$$ Note that
$$\alpha_I =\min\left\{x\in\mathbb{R}_{\geq 0} \mid xH^{\prime}-E^{\prime}\in \overline{\mathrm{Eff}}(X^{\prime})\right\}\quad\mbox{and}\quad \beta_I = \min\left\{x\in\mathbb{R}_{\geq 0} \mid xH^{\prime}-E^{\prime}\in \mathrm{Nef}(X^{\prime})\right\}.$$ Now if $\alpha_I = \beta_I$, then $f_{A,\{\widetilde {I^n}\}}(x)$ is given by a single quadratic polynomial. So we may further assume that $\alpha_I<\beta_I$.

Since $X^{\prime}$ is a nonsingular projective surface, by \cite{BKS04} there is a locally ﬁnite decomposition of $\mathrm{Big}(X^{\prime})$ into rational locally polyhedral (convex) subcones such that on each subcone, the support of the negative part of the Zariski decomposition of the divisors in the subcone is constant, and the volume function is given by a single polynomial of degree $2$. This means that given $x>\alpha_I$, there is an open interval $(x_0,x_1)$ with $\alpha_I<x_0<x<x_1$, such that if $$D^{\prime}_{x_0}:=x_0H^{\prime} - E^{\prime} = P_{D^{\prime}_{x_0}} + N_{D^{\prime}_{x_0}},\quad \text{and}\quad D^{\prime}_{x_1}:=x_1H^{\prime} - E^{\prime} = P_{D^{\prime}_{x_1}} + N_{D^{\prime}_{x_1}}$$ are their respective Zariski decompositions then for all $\lambda\in (0,1)$, $$\left(\lambda P_{D^{\prime}_{x_0}} + (1-\lambda) P_{D^{\prime}_{x_1}}\right) + \left(\lambda N_{D^{\prime}_{x_0}} + (1-\lambda) N_{D^{\prime}_{x_1}}\right)$$ is the Zariski decomposition of $\lambda D^{\prime}_{x_0} + (1-\lambda)D^{\prime}_{x_1} = \left(\lambda x_0 + (1-\lambda)x_1\right)H^{\prime} - E^{\prime}$. In other words,
\begin{align*}
 \mathrm{vol}_{X^{\prime}}\left(\left(\lambda x_0 + (1-\lambda)x_1\right)H^{\prime} - E^{\prime}\right) &= \left(\lambda P_{D^{\prime}_{x_0}} + (1-\lambda) P_{D^{\prime}_{x_1}}\right)^2\\
 &= \lambda^2 \big(P_{D^{\prime}_{x_0}}^2\big) + 2\lambda (1-\lambda) \big(P_{D^{\prime}_{x_0}}\cdot P_{D^{\prime}_{x_1}}\big) + (1-\lambda)^2 \big(P_{D^{\prime}_{x_1}}^2\big)
\end{align*}
for all $\lambda\in (0,1)$. Thus if $x=\lambda x_0 + (1-\lambda)x_1$, where $\lambda\in (0,1)$, then
\begin{align*}
 f_{A,\{\widetilde{I^n}\}}(x) &= 3\cdot\mathrm{vol}_{X^{\prime}}(xH^{\prime}-E^{\prime})\\
 &= \dfrac{3}{(x_1-x_0)^2}\left[(x_1-x)^2 \big(P_{D^{\prime}_{x_0}}^2\big) + 2(x_1-x)(x-x_0) \big(P_{D^{\prime}_{x_0}}\cdot P_{D^{\prime}_{x_1}}\big) + (x-x_0)^2 \big(P_{D^{\prime}_{x_1}}^2\big)\right]
\end{align*}
is a polynomial in $x$ of degree at most $2$.
\end{proof}

\begin{ex}[Nagata]\label{eg_Nagata}
Suppose that $A=\mathbb{C}[X,Y,Z]$ is the homogeneous coordinate ring of $\mathbb{P}^2_{\mathbb{C}}$, where $\mathbb{C}$ denotes the field of complex numbers. Let $s\geq 4$ be an integer and let $p_1,\ldots,p_{s^2}$ be $s^2$ general points in $\mathbb{P}^2_{\mathbb{C}}$. Consider the homogeneous ideal $I = \bigcap_{i=1}^{s^2} I(p_i)$, where $I(p_i)$ is the ideal generated by all homogeneous polynomials vanishing at $p_i$. Then the saturated Rees algebra $\oOplus_{n\geq 0}\widetilde{I^n}t^n$ is non-Noetherian and $$f_{A,\{\widetilde{I^n}\}}(x) = \begin{cases}
0 & \text{if}\;\; 0\leq x \leq s\\
3\cdot (x^2-s^2) & \text{if}\;\; x\geq s
\end{cases}.
$$
\end{ex}
\begin{proof}
Let $\pi \colon X \to \mathbb{P}^2_{\mathbb{C}}$ be the blow up of the points $p_1,\ldots,p_{s^2}$ with exceptional lines $E_1 = \pi^{-1}(p_1),\ldots, E_{s^2} = \pi^{-1}(p_{s^2})$ and let $E = E_1+\cdots + E_{s^2}$. Let $H^{\prime}$ be a linear hyperplane section of $\mathbb{P}^2_{\mathbb{C}}$ and let $H = \pi^*(H^{\prime})$. For all nonnegative integers $m$ and $n$,
\begin{align*}
H^0\left(X,\mathcal{O}_X\left(mH - nE\right)\right) &= H^0\left(\mathbb{P}^2_{\mathbb{C}}, \mathcal{O}_{\mathbb{P}^2_{\mathbb{C}}}(mH^{\prime})\otimes \mathcal{I}_{p_1}^{n}\otimes \cdots \otimes \mathcal{I}_{p_{s^2}}^{n}\right) = \bigcap\limits_{i=1}^{s^2} \left(I(p_i)^{n}\right)_m = \big(\widetilde{I^n}\big)_m.
\end{align*}
It was shown by Nagata \cite{Nag59} that the saturated Rees algebra $\oOplus_{n\geq 0}\widetilde{I^n}t^n$ is infinitely generated. The Riemann-Roch theorem for surfaces combined with \cite[Lemma $6.3$]{CS22} gives us that $$h^0\left(X,\mathcal{O}_X\left(mH - nE\right)\right) = \begin{cases}
0 & \text{if}\;\; 0\leq m \leq sn,\;n>0\\
\binom{m+2}{2} - s^2\binom{n+1}{2} & \text{if}\;\; m>sn+s-3,\;n\geq 0
\end{cases}.
$$ Our assertions are now immediate.
\end{proof}

\begin{ex}[Cutkosky]\label{eg_Cutkosky}
The following example was used in \cite[Example $4.4$]{CHT99} to show that $\lim_{n\to\infty} \frac{\mathrm{reg}(\widetilde{I^n})}{n}$ can be irrational. Let $k$ be an algebraically closed field. Then there exists a three dimensional standard graded domain $R$ over $k$ and a homogeneous ideal $I\subseteq R$ such that $$f_{R,\{\widetilde{I^n}\}}(x) = \begin{cases}
0 & \text{if}\;\; 0\leq x \leq \frac{1}{33}(6+\sqrt{3})\\
18\cdot (33x^2-12x+1) & \text{if}\;\; x\geq \frac{1}{33}(6+\sqrt{3})
\end{cases}.
$$ 
\end{ex}
\begin{proof}
Here we shall omit the details of the computations which can be found in \cite[Example $4.4$]{CHT99}. Let $C$ be an elliptic curve over $k$ and let $S=C\times C$. Let $\Delta \subset S$ be the diagonal, $p\in S$ a closed point and $A=\pi_1^{-1}(p), B=\pi_2^{-1}(p)$, where $\pi_i \colon S \to C$ are the projections. The intersection products on $S$ are given by $$\Delta^2 = A^2 = B^2 =0 \quad \text{and}\quad (A\cdot B) = (A\cdot \Delta) = (B\cdot \Delta) = 1.$$ Let $H=3A+6B+9\Delta$ and $D=A+B+\Delta$. The quadratic equation $(xH-D)^2 = 198x^2-72x+6 = 0$ has the roots $s_1 = \frac{1}{33}(6-\sqrt{3})$ and $s_2 = \frac{1}{33}(6+\sqrt{3})$. Note that $H$ is very ample on $S$ and set $R = \oOplus_{m\geq 0}H^0(S,\mathcal{O}_S(mH))$. Define $I = I_1\cap I_2\cap I_3$, where $I_1,I_2$ and $I_3$ are the defining homogeneous ideals of $A,B$ and $\Delta$ respectively. For all integers $m\geq 0$ and $n\geq 0$, we have $H^0(S,\mathcal{O}_S(mH-nD)) = \big(\widetilde{I^n}\big)_m$ and $$h^0\left(S,\mathcal{O}_S\left(mH - nD\right)\right) = \begin{cases}
0 & \text{if}\;\; 0\leq m < s_2n\\
\frac{(mH-nD)^2}{2} = 99m^2 - 36mn + 3n^2 & \text{if}\;\; m>s_2n
\end{cases}.
$$ Our claims are now immediate.
\end{proof}

\begin{ex}\label{edgeIdeal}
Let $k$ be an algebraically closed field and let $A=k[X,Y,Z]$ be the standard graded polynomial ring in the variables $X,Y,Z$ over $k$. Consider the monomial ideal $I=(XY,YZ,ZX)$ in $A$. Then $$f_{A,\{\widetilde{I^n}\}}(x) = \begin{cases}
0 & \text{if}\;\; x \leq \frac{3}{2}\\
3\cdot(2x-3)^2 & \text{if}\;\; \frac{3}{2}\leq x \leq 2\\
3\cdot(x^2 - 3) & \text{if}\;\; x\geq 2
\end{cases}$$
and $$f_{\varepsilon(I)}(x) = \begin{cases}
                              3\cdot (2x-3)^2 & \text{if}\;\; \frac{3}{2}\leq x \leq 2\\
                              0 & \text{if}\;\; x\leq \frac{3}{2}\;\;\text{or}\;\;x\geq 2
                             \end{cases}.$$
In particular, it reproves the assertion of \cite[Example $2.4$]{CHS10} that $\varepsilon(I) = 1/2$.
\end{ex}
\begin{proof}
Let $\pi \colon X \to \mathbb{P}^2_{k}$ be the blow up of the three points $p_1 = [1\colon 0\colon 0], p_2 = [0\colon 1\colon 0], p_3 = [0\colon 0\colon 1]$ with exceptional lines $E_1 = \pi^{-1}(p_1), E_2 = \pi^{-1}(p_2), E_3 = \pi^{-1}(p_3)$ and let $E = E_1+E_2+E_3$. Let $H^{\prime}$ be a linear hyperplane section of $\mathbb{P}^2_{k}$ and let $H = \pi^*(H^{\prime})$. For all integers $m\geq 0$ and $n\geq 0$, $$H^0\left(X,\mathcal{O}_X\left(mH - nE\right)\right) = \big(\widetilde{I^n}\big)_m.$$ The intersection products on $X$ are given by $$H^2 = 1, H\cdot E_i = 0, E_i^2 = -1,\;\text{and}\; E_i\cdot E_j = 0\;\text{for}\; i\neq j.$$ It is shown in \cite[Proposition $5.3$ and Example $4.7$]{HHT07} that for all $n\geq 0$, $$\widetilde{I^{2n}} = \big(\widetilde{I^2}\big)^n = {\left(X^2Y^2, Y^2Z^2, Z^2X^2, XYZ\right)}^n.$$ Thus, $\alpha_I = \frac{3}{2}$ and $\beta_I = 2$ as the ideal $\widetilde{I^2}$ is minimally generated in degrees $3,4$. So for $x\geq 2$, $$\mathrm{vol}_X(xH-E) = (xH-E)^2 = H^2x^2 - 2(H\cdot E)x + E^2 = x^2-3.$$ In view of Theorems \ref{t3} and \ref{t2sat}, there exists a quadratic polynomial $g(x)$ such that $$f_{A,\{\widetilde{I^n}\}}(x) = \begin{cases}
0 & \text{if}\;\; 0\leq x \leq \frac{3}{2}\\
g(x) & \text{if}\;\; \frac{3}{2}\leq x \leq 2\\
3\cdot (x^2 - 3) & \text{if}\;\; x\geq 2
\end{cases}
$$ and satisfying $g\left({3}/{2}\right) = 0$, $g(2) = 3$, and $g^{\prime}(2) = 12$. This gives that $g(x) = 3\cdot (2x-3)^2$. Since the ideal $I$ is generated in equal degrees $2$,
$$f_{A,\{I^n\}}(x) = \begin{cases}
0 & \text{if}\;\; 0\leq x \leq 2\\
3\cdot (x^2 - 3) & \text{if}\;\; x\geq 2
\end{cases}.$$ Thus $$\varepsilon(I) = \int\limits_{0}^{\infty}\left(f_{A,\{\widetilde{I^n}\}}(x)-f_{A,\{I^n\}}(x)\right)dx = 3\int\limits_{3/2}^2 (2x-3)^2 dx = \dfrac{1}{2}.$$
\end{proof}

\section{Acknowledgements}

The authors would like to thank Prof. Steven Dale Cutkosky for his valuable comments and suggestions. The authors would also like to thank the referee for  meticulously reading  the manuscript, and for numerous suggestions which greatly improved the exposition.

\bibliographystyle{alpha}
\bibliography{Reference}
\end{document}